\documentclass[11pt,UTF-8,reqno]{amsart}
\usepackage{enumerate, bbm}
\usepackage{amssymb,url,color, booktabs}
\usepackage{mathrsfs}
\usepackage{cite}
\usepackage[utf8]{inputenc}

\usepackage{color}
\usepackage[colorlinks=true]{hyperref}
\hypersetup{
    linkcolor=blue,          % color of internal links
    citecolor=red,        % color of links to bibliography
    filecolor=blue,      % color of file links
    urlcolor=cyan
}

\usepackage{color}
\definecolor{MyDarkBlue}{cmyk}{0.8,0.3,0.8,0.4}
\definecolor{yellow}{rgb}{0.99,0.99,0.70}
\definecolor{white}{rgb}{1.0,1.0,1.0}
\definecolor{black}{rgb}{0.00,0.00,0.00}

\numberwithin{equation}{section}

\newcommand{\be}{\begin{eqnarray}}
\newcommand{\ee}{\end{eqnarray}}
\newcommand{\ce}{\begin{eqnarray*}}
\newcommand{\de}{\end{eqnarray*}}
\newtheorem{theorem}{Theorem}[section]
\newtheorem{lemma}{Lemma}[section]

\newtheorem{definition}{Definition}[section]

\newtheorem{corollary}{Corollary}[section]

\newtheorem{remark}[theorem]{Remark}

\newtheorem{conditionA}{A\kern-0mm}

\newtheorem{conditionH}{H\kern-0mm}

\newtheorem{conditionC}{C\kern-0mm}

\def\[{{\Big[}}
\def\]{{\Big]}}
\def\<{{\langle}}
\def\>{{\rangle}}
\def\({{\big(}}
\def\){{\big)}}

\def\bb2{{\boldsymbol{2}}}

\def\={&\!\!=\!\!&}
\def\RR{\mathbb{R}}
\def\PP{\mathbb{P}}

\def\b1{{\mathbbm 1}}

\def\geq{\geqslant}
\def\leq{\leqslant}

\def\le{\leqslant}

\def\[{{\Big[}}
\def\]{{\Big]}}
\def\<{{\langle}}
\def\>{{\rangle}}

\def\={&\!\!=\!\!&}
\def\bt{\begin{theorem}}
\def\et{\end{theorem}}
\def\bl{\begin{lemma}}
\def\el{\end{lemma}}
\def\br{\begin{remark}}
\def\er{\end{remark}}
\def\bd{\begin{definition}}
\def\ed{\end{definition}}
\def\bc{\begin{corollary}}
\def\ec{\end{corollary}}

\def\geq{\geqslant}
\def\leq{\leqslant}

\def\le{\leqslant}

\def\<{\langle} \def\>{\rangle}

\def \eref#1{\hbox{(\ref{#1})}}

\allowdisplaybreaks

\begin{document}

\title[McKean-Vlasov Stochastic Partial Differential Equations]
{McKean-Vlasov Stochastic Partial Differential Equations: Existence, Uniqueness and Propagation of Chaos$\dagger$}

\thanks{$\dagger$ This work is supported by National Key R\&D program of China (No. 2023YFA1010101). The research of W. Hong is also supported by  NSFC (No.~12401177) and  NSF of Jiangsu Province (No.~BK20241048). The research of S. Li is also supported by NSFC (No.~12371147). The research of W. Liu is also supported by NSFC (No.~12171208, 12090011,12090010) and the PAPD of Jiangsu Higher Education Institutions.}

%\thanks{$\dagger$
% The research of W. Liu is supported by NSFC (No.~12171208, 11831014, 12090011) and the PAPD of Jiangsu Higher Education Institutions. }

%\thanks{$\dagger$
%This work is supported by National Key R\&D program of China (No. 2023YFA1010101). The research of W. Hong is also supported by  NSFC (No.~12401177) and  NSF of Jiangsu Province (No.~BK20241048).  The research of W. Liu is also supported by NSFC (No.~12171208, 12090011,12090010). }

\maketitle
\centerline{\scshape Wei Hong, Shihu Li,   Wei Liu\footnote{Corresponding author:  weiliu@jsnu.edu.cn}}

\medskip

\vspace{1mm}
 {\footnotesize
\centerline{  School of Mathematics and Statistics, Jiangsu Normal University, Xuzhou 221116, China,}}

\begin{abstract}
In this paper, we provide a general framework for investigating  McKean-Vlasov stochastic partial differential equations. We first show the existence of weak solutions by combining the localizing approximation, Faedo-Galerkin technique, compactness method and the
Jakubowski version of the Skorokhod representation theorem. Then under certain locally monotone condition we further investigate the existence and uniqueness of  (probabilistically) strong solutions.
The applications of the main results include a large class of  McKean-Vlasov stochastic partial differential equations such as stochastic 2D/3D Navier-Stokes equations, stochastic Cahn-Hilliard equations and stochastic Kuramoto-Sivashinsky equations. Finally, we show a propagation of chaos result in Wasserstein distance for weakly interacting stochastic 2D Navier-Stokes systems.


\end{abstract}
\maketitle \rm

%\pagecolor{MyDarkBlue}\color{yellow}
\tableofcontents

\section{Introduction}
\setcounter{equation}{0}
 \setcounter{definition}{0}
Stochastic systems including a large number of particles with weak interactions were initially investigated in statistical physics, and then were commonly studied in many other fields  such as biology, game theory and finance. When the number of particles gets very large, the system involves  more and more information so that  the analysis and simulation of the whole system  becomes too complicated. In order to overcome this difficulty, it is effective to replace all interactions with particles by an average interaction and to investigate the limiting behaviour of the empirical laws of the particles as the number of
particles grows to infinity. Such macroscopic behavior of the interacting particle
system is usually called the {\it{propagation of chaos}} in Kac's programme for kinetic theory (cf.~\cite{KAC}). On the other hand, it is also possible to characterize the limit through McKean-Vlasov stochastic differential equations (MVSDEs), also referred as  distribution dependent SDEs or mean-field SDEs in the literature, which first appeared in the seminal work \cite{M} by McKean.  Since McKean's work, MVSDEs have been applied extensively in the fields of queuing systems, stochastic control and mathematical finance, see e.g.\cite{CD1} and the references therein.

The typical form of a weakly interacting particle system is
  \begin{equation}\label{eq000}
  dX^i_t=\frac{1}{N}\sum_{j=1}^Nb(X^i_t,X^j_t)dt+dW^i_t,
  \end{equation}
where $W^i_t,i=1,\ldots,N,$ are independent Wiener processes. When $N$ goes to infinity, it leads to the following MVSDEs that interacts with the law of itself, i.e.
  $$dY^i_t=\int_{\mathbb{R}^d} b(Y^i_t,y)\mathscr{L}_{Y^i_t}(dy)dt+dW^i_t,$$
 where $\mathscr{L}_{Y^i_t}$ is the law of $Y^i_t$. From this viewpoint, if one observes the law of such type of $N$-particle system,  although the particles in system (\ref{eq000}) is interacting with each other, as $N\to\infty$ they become statistically independent, which can be viewed as a law of large numbers, i.e. for any $f\in C_b(\mathbb{R}^d)$,
 $$\mathscr{S}_t^N(f):=\frac{1}{N}\sum_{j=1}^Nf(X^j_t)\to \mu_t(f):=\int_{\mathbb{R}^d}f(y)\mu_t(dy),$$
where $\mu_t=\mathscr{L}_{Y^1_t}$.

The existence and uniqueness of strong solutions and the propagation of chaos of MVSDEs under global Lipschitz conditions have been intensively studied (cf.~e.g.~\cite{S1} and references therein).  Subsequently, many further investigations of MVSDEs under different assumptions on the coefficients have been done over the years. For instance, the existence and uniqueness of strong/weak solutions and ergodicity to MVSDEs  under the monotonicity conditions and with singular coefficients are systematically studied by Wang e.g.~\cite{WFY,WFY3,WFY2} in recent years.  Hao et al.~\cite{HRZ1,HRZ2,HRZ3} also studied the well-posedness and propagation of chaos for MVSDEs with singular coefficients.
On the other hand, Barbu and R\"{o}ckner \cite{BR1,BR2,BR3,BR4}  used the nonlinear Fokker-Planck equations to investigate the existence and uniqueness of weak solutions as well as existence of invariant measures for the Nemytskii type MVSDEs, see also \cite{BR5,RXZ} for the associated investigations in the L\'{e}vy noise case.
Although the global Lipschitz type conditions are very commonly used in the literature, many models do not belong to this class in the applications; for example, in the D'Orsogna et al. model (cf.~\cite{DCBC}) and the Cucker-Smale model (cf.~\cite{CS}), the interaction kernels only satisfy local Lipschitz conditions.
%Recently, Erny \cite{E1} extended the well-posedness and propagation of chaos for McKean-Vlasov equations with jumps noise and local Lipschitz coefficients.

In addition, McKean-Vlasov stochastic partial differential equations (MVSPDEs)  as the mean-field limit of weakly interacting SPDE systems have also attracted more and more attention in recent years due to their potential applications in e.g.~neurophysiology and quantum field theory.  For instance, Chiang et al. \cite{CKS} considered the propagation of chaos problem for the interacting system of $n$-SDEs taking
values in the dual of a countably Hilbertian nuclear space. As application, the results in \cite{CKS} can be used to describe random strings and the fluctuation of voltage potentials of interacting spatially extended neurons,  which are governed by the following weakly interacting SPDE systems
 $$dX^{N,i}_t=\Big(\Delta X^{N,i}_t+\frac{1}{N}\sum_{j=1}^Nb_t(X^{N,i}_t,X^{N,j}_t)\Big)dt+dW^i_t,~1\leq i\leq N,$$
where the latter is a more realistic model for large numbers of neurons in close proximity to each other. The interested reader is referred to \cite{K1} for more background in
neurophysiology, and also to \cite{BKKX,C1,KX} for the extensions to non-nuclear spaces or systems driven by Poisson random measures. Recently, Ren et al.~\cite{RTW}  derived the global existence and uniqueness for a class of distribution-path dependent  stochastic transport type equations, e.g.~stochastic Camassa-Holm equation with distribution coefficients.  We also refer to \cite{ES1,HHL,HLL1,HKXZ,SSZZ,SZZ22} for further studies on the MVSPDEs.

\subsection*{Weak and strong solutions}
Motivated by the aforementioned points, the first aim of this work is to develop a general solution theory to investigate the well-posedness of  MVSPDEs
 \begin{equation}\label{eq00}
dX_t=A(t,X_t,\mathscr{L}_{X_t})dt+\sigma(t,X_t,\mathscr{L}_{X_t})dW_t,
\end{equation}
where $W_t$ is a cylindrical Wiener process, the drift $A$ and diffusion $\sigma$ satisfy certain assumptions we will mention later.  This general setting is motivated by various applications, some examples will be given in Subsection \ref{example}.

We first investigate the existence of weak solutions to MVSPDEs (\ref{eq00}) under the coercivity and polynomial growth assumptions (cf.~$\mathbf{H1}$-$\mathbf{H3}$  below). The first difficulty lies in proving the corresponding result in finite dimensions. Indeed, to the best of our knowledge, there is no accessible reference in the literature which work under such general assumptions. We mention that the authors in \cite{H1,MV,RZ21} proved the existence of weak solutions to MVSDEs under linear growth or integrability conditions, which is mainly based on  truncation argument in the proof. In this work, we construct a cut-off function (cf.~(\ref{cutoff}) below), which is different from the one in \cite{H1,MV,RZ21} but more suitable in the current setting, to truncate both the solution and its law of (\ref{eq00}) and then get the convergence to a weak solution of (\ref{eq00}), see Subsection \ref{sec2.2} for more details.

Based on the existence result in finite dimensions, we further show the existence of weak solutions and the existence and uniqueness of strong solutions to MVSPDEs (\ref{eq00}).
We consider the MVSPDEs (\ref{eq00}) under general coercivity and growth conditions (cf.~$\mathbf{H1}$-$\mathbf{H3}$  below) within the following spaces and embeddings
\begin{equation}\label{eq02}
\mathbb{X}\subset {\mathbb{V}}\subset\mathbb{Y}\subset {\mathbb{H}}(\simeq {\mathbb{H}}^*)\subset \mathbb{X}^*.
\end{equation}
This relationship of embedding is mainly inspired by the work \cite{GRZ}, where the authors proved the existence of martingale solutions and Markov selections for classical  (i.e.~no distribution dependent) SPDEs.

Now, we outline the main ideas presented in the proof of the well-posedness.  Our proof for the existence of weak solutions to MVSPDEs (\ref{eq00})  is based on the  Faedo-Galerkin technique, compactness methods and Jakubowski's version
of the Skorokhod theorem for nonmetric spaces, which is quite different from the work \cite{GRZ}.
More precisely, we will first establish the tightness of the sequence  $\{X^n\}_{n\in\mathbb{N}}$ to the Faedo-Galerkin approximation of (\ref{eq00}) in the space
$$\mathbb{K}_2:=C([0,T];\mathbb{X}^*)\cap L^{q}([0,T];\mathbb{Y})\cap L^{q}_w([0,T];{\mathbb{V}}),$$
 where $ L^{q}_w([0,T];{\mathbb{V}})$ denotes the space $L^{q}([0,T];{\mathbb{V}})$ endowed with the weak topology. It should be pointed out that the space $\mathbb{K}_2$ is not a Polish space. In this case, we will apply  Jakubowski's version of the Skorokhod theorem for nonmetric spaces in the form given by Brze\'{z}niak and Ondrej\'{a}t \cite{BO} to show the almost surely convergence in $\mathbb{K}_2$ of the approximating sequence   on a new probability space. To this end, we need to show $\mathbb{K}_2$ is a standard Borel space with respect to an appropriate topology (see Remark \ref{k2t} for details).
On the other hand, due to the dependence of laws and the fact that $\mathbb{K}_2$ is not Polish, one cannot  get the convergence of the  sequence of laws $\{\nu^n:=\mathscr{L}_{X^n}\}_{n\in\mathbb{N}}$ in $\mathscr{P}(\mathbb{K}_2)$. Therefore, we will consider the convergence of $\nu^n$ in a larger space $\mathscr{P}(\mathbb{K}_1)$. Here,
$\mathbb{K}_1:=C([0,T];\mathbb{X}^*)\cap L^{q}([0,T];\mathbb{Y})$
which is a Polish space under the natural norm. Employing the martingale approach, we are able to characterize the weak limit of $\{X^n\}_{n\in\mathbb{N}}$, which is proved to be a weak solution to  MVSPDEs (\ref{eq00}).

In order to investigate the pathwise uniqueness of MVSPDEs (\ref{eq00}), we impose certain local monotonicity conditions on the coefficients, see the assumption $\mathbf{H4}$ in Subsection \ref{strongsolution}.
It is worth emphasizing that there is  a counterexample from Scheutzow \cite{SCHEUTZOW} which shows that if the drift is merely local Lipschitz continuous, the uniqueness of solutions to the McKean-Vlasov equations  might not hold in general. Therefore, it is crucial to balance the influence between the nonlinear operators and  the distribution dependent coefficients.

We shall prove the pathwise uniqueness of (\ref{eq00}) under the  local monotonicity condition $\mathbf{H4}$ and then derive the existence and uniqueness of strong solutions by using a {\it decoupled} argument.
Such a local monotonicity condition is more general than the original one introduced in \cite{LR2}. Thus our results on strong solutions can be applied to a larger class of SPDE models in the distribution dependent case, which are not covered by previous works. More importantly, this result is also very useful in characterizing the mean field limit and studying the propagation of chaos for weakly interacting nonlinear SPDEs, which is another main purpose of the present paper.
%Some examples are given in Subsection \ref{example} below.

\subsection*{Propagation of chaos}
The propagation of chaos for weakly interacting SDEs has been intensively studied in the past several decades. We do not
intend to give a comprehensive list, but rather refer to the surveys \cite{G1,J1,S1} and the monographs \cite{CD1,S2}.
Although the convergence of empirical measures and propagation of chaos have been intensively studied in the finite dimensional case, there are very few results on propagation of chaos in the literature for weakly interacting nonlinear SPDEs. Very recently, Shen et al.~\cite{SSZZ} studied the large $N$ limits of  $O(N)$  linear sigma model posed over $\mathbb{T}^d$ for $d=1,2$, where the limit (as $N\to\infty$) has been proved to be governed by a  mean-field singular SPDE. Such a result is generally referred to large $N$ problem in  physics and mathematics, particularly in quantum field theory and  statistical mechanics, and is also related to the mean field limit theory, see also \cite{SZZ22} for some generalization in $d=3$. We also refer to \cite{BKKX,CKS,C1,ES1,KX} on the propagation of chaos for the interacting particle systems driven by semilinear SPDEs with globally Lipschitz coefficients.

In order to illustrate our results more clearly, in this work we only focus on the propagation of chaos in a concrete model rather than in a general framework.
More precisely, we consider the following interacting stochastic 2D Navier-Stokes systems
 \begin{eqnarray}\label{ns01}
dX^{N,i}_t=\!\!\!\!\!\!\!\!&&\Big[\nu\Delta X^{N,i}_t-(X^{N,i}_t\cdot\nabla)X^{N,i}_t+\nabla p^{N,i}_t+\frac{1}{N}\sum_{k=1}^N\tilde{K}(t,X^{N,i}_t,X^{N,k}_t)\Big]dt
\nonumber\\
\!\!\!\!\!\!\!\!&&
+\frac{1}{N}\sum_{k=1}^N\tilde{\sigma}(t,X^{N,i}_t,X^{N,k}_t)dW_t^i,~~ X^{N,i}_0=\xi^i,~~ 1\leq i\leq N,
\end{eqnarray}
%with the initial values $X^{N,i}_0=\xi^i$,
% \begin{eqnarray*}
%\left\{
% \begin{aligned}
%    &{\rm div}(X^{N,i}_t)=0,~~\text{in}~\mathscr{O}, \\
 %   &X^{N,i}_t=0, ~~~~~~~~\text{on}~\partial\mathscr{O},
  %%  \\
%%    &X^{N,i}_0=\xi^i, ~~~~~~~\text{in}~\mathscr{O},
 % \end{aligned}
%\right.
%\end{eqnarray*} $\mathscr{O}\subset\mathbb{R}^2$ is a bounded open domain with smooth boundary,
where the interaction external force term $\tilde{K}$ as well as $\tilde{\sigma}$ satisfies locally Lipschitz and linear growth conditions (see $\mathbf{A1}$-$\mathbf{A2}$ below), and a typical example for  the interaction term  is the {\it{Stokes drag force}} in  fluids, namely
$$\tilde{K}(t,u,v):=c_0(u-v),$$
where $c_0$ is a positive constant. The {\it{Stokes drag force}} (also called fluid resistance) is appropriate for objects or particles moving through a fluid at relatively slow speed (i.e.~low Reynolds number), which is approximately proportional to  the relative velocity (i.e.~$u-v$). This interaction force is very  classical and important also in the theory of propagation of chaos for finite-dimensional stochastic differential equations.

The second main aim of this paper is to identify the mean field limit of the interacting stochastic 2D Navier-Stokes systems \eref{ns01}. More precisely, we show (in Theorem \ref{th4}) that if $$\mathscr{S}^N_0:=\frac{1}{N}\sum_{j=1}^N\delta_{\xi^j}\to \mu_0~\text{in probability as}~ N\to\infty,$$
then there exists an accumulation point of  the empirical laws $\mathscr{S}^N_t:=\frac{1}{N}\sum_{j=1}^N\delta_{X^{N,j}_t}$ to the interacting system \eref{ns01}, which is a martingale solution of 2D McKean-Vlasov stochastic Navier-Stokes equation (2D MVSNSE) with initial law $\mu_0$. Furthermore, under additional assumptions we can also show the following convergence
\begin{equation}\label{eq.5}
\lim_{N\to\infty}\mathbb{E}\Big[\mathbb{W}_{2,T,\mathbb{H}}(\mathscr{S}^N,\Gamma)^2\Big]=0.
\end{equation}
Here, $\mathbb{H}$ is the space of Lebesgue square
integrable vector field with zero divergence, $\mathbb{W}_{2,T,\mathbb{H}}$ denotes the $L^2$-Wasserstein distance of probability measures on $C([0,T],\mathbb{H})$, and $\Gamma$ is the unique solution of the martingale problem to the following 2D MVSNSE
 \begin{eqnarray}\label{eq.6}
\left\{
 \begin{aligned}
    &dX_t=\Big[\nu\Delta X_t-(X_t\cdot\nabla)X_t+\nabla p_t+ K(t,X_t,\mathscr{L}_{X_t})\Big]dt
+{\sigma}(t,X_t,\mathscr{L}_{X_t})dW_t,\\
    &X_0\sim\mu_0,
  \end{aligned}
\right.
\end{eqnarray}
where
$$K(t,X_t,\mathscr{L}_{X_t}):=\int\tilde K(t,X_t,y)\mathscr{L}_{X_t}(dy),~~{\sigma}(t,X_t,\mathscr{L}_{X_t}):=\int \tilde{\sigma}(t,X_t,y)\mathscr{L}_{X_t}(dy).$$
Note that if the uniqueness in law holds for the McKean-Vlasov equation \eref{eq.6}, by the classical work of Sznitman (cf.~Proposition 2.2 in \cite{S1}), this result says that
the (asymptotic) independence of the law of any $k$-particles at time zero propagates to future time
instants, and for this reason this result is referred as the propagation of chaos.

The approximations  of stochastic interacting particle systems to the deterministic 2D NSEs and other PDEs have already been investigated a lot in the literature, see e.g.~\cite{CHJZ,FOS,GL} and references therein. In particular,
 Jabin and Wang \cite{JW}  showed that a mean field approximation
converges to the solution of the NSEs in vorticity form, and
they are able to obtain quantitative optimal convergence rates for all finite marginal
distributions of the particles. However, the system \eref{ns01} considered here is given by a family of interacting stochastic Navier-Stokes systems, which can be used to describe the dynamics of interaction between a large number of particles in fluids. In particular,  it allows to consider micro-organisms like bacteria ``swiming" in the fluid and to study  its microscopic limit. A related work is the paper by Flandoli et al.~\cite{FLR21}, which investigates a coupling between the fluid and a particle system, where the interaction between particles and the fluid is modelled by the Stokes drag force. In \cite{FLR21}, the empirical measure of particles is proved to converge to the Vlasov-Fokker-Planck component of the system and the velocities of the fluid coupled with the particles converge in the uniform norm to the Navier-Stokes component. See also \cite{Al91,DGR08,JO04}  which also treat links between particles and fluid.

Now, we outline the main ideas in the proof of the propagation of chaos.  The main proof relying on the uniform in $N$ energy estimates on $\mathbb V:=W_{{div}}^{1,2}(\mathscr{O})$ before a suitable stopping time and the stochastic compactness argument.
We point out that it is quite non-trivial to establish the convergence of empirical measures in a Wasserstein space with respect to the state space $\mathbb{H}$, due to the non-linear term appeared in the interacting system.
The first difficulty
lies in establishing the tightness of the sequence $\{X^{N,1}\}_{N\in\mathbb{N}}$
both in $C([0,T];\mathbb{H})$ and $L^2([0,T];\mathbb{V})$,
and of the empirical measures
$\{\mathscr{S}^N\}_{N\in\mathbb{N}}$ in
\begin{equation}\label{ti1}
\mathscr{P}_2(C([0,T];\mathbb{H})\cap L^2([0,T];\mathbb{V})).
\end{equation}
To overcome this difficulty, we assume that
the initial value has higher regularity, so that we can derive some energy estimates to the interacting system \eref{ns01} in a more regular space before a stopping time uniformly in the number of particles. With these uniform in $N$ energy estimates and the help of Lemma \ref{lem00} presented in Appendix, which plays a crucial role in finding a relatively compact subset in the Wasserstein space, we are able to prove the tightness the empirical measures
$\{\mathscr{S}^N\}_{N\in\mathbb{N}}$ in Wasserstein space (\ref{ti1}).
Then through the localization procedure and the martingale characterization, we  show that as an  accumulation point of $\{\mathscr{S}^N\}_{N\in\mathbb{N}}$ it is a solution of the martingale problem to 2D MVSNSE \eref{eq.6}.
By the weak uniqueness of solutions, we are able to show the convergence (\ref{eq.5}).

\vspace{2mm}
The main contributions of the present work can be summarized as follows.
\begin{enumerate}[(i)]
  \item The first main challenge in this work is to prove the existence of weak solutions in $\mathbb{R}^d$ to the distribution dependent stochastic equations whose coefficients only satisfy continuity, coercivity and polynomial growth conditions. To this end, we construct an effective cut-off function (cf.~(\ref{cutoff}) below) to truncate both the solution and its laws. This is more general than what was done in the existing works and should be of independent interest.
  \item The second main novel point is that we build a general framework to study the existence of weak solutions as well as existence and uniqueness of strong solutions to MVSPDEs  that can applied to various nonlinear SPDE models with distribution dependent coefficients.

  \item The third main novel point concerns the propagation of chaos for weakly interacting nonlinear SPDEs. To the best of our knowledge, this is the first result in the literature on the propagation of chaos for weakly interacting stochastic Navier-Stokes systems, and we believe that our method is also applicable to many other weakly interacting SPDE models.
     % which might shed a new light on this subject.
\end{enumerate}

This paper is organized as follows. In Sect.~\ref{wellposed}, we  introduce the detailed framework and state the main results about weak and strong solutions in Theorems \ref{th1} and \ref{th2} respectively. Then we apply our general framework to some concrete examples to illustrate the  wide applicability of the main results. In Sect.~\ref{proof1} we give the proofs of Theorems \ref{th1} and \ref{th2}. Sect.~\ref{Poc.4} is devoted to giving the main result (Theorem \ref{th4}) and its proof, which concerns the propagation of chaos problem for weakly interacting stochastic 2D Navier-Stokes systems. We also recall some useful lemmas and postpone their proofs to  Appendix (see Sect.~\ref{appendix}).
Throughout this paper $C_{p}$  denotes some positive constant which may change from line to line, where the subscript $p$ is used to emphasize that the constant depends on certain parameter $p$.

\section{Weak and strong solutions}\label{wellposed}
\setcounter{equation}{0}
 \setcounter{definition}{0}
Let $(U,\langle\cdot,\cdot\rangle_U)$, $(\mathbb{X},\langle\cdot,\cdot\rangle_{\mathbb{X}})$ and $({\mathbb{H}}, \langle\cdot,\cdot\rangle_{\mathbb{H}}) $ be  separable Hilbert spaces and  ${\mathbb{V}}$, $\mathbb{Y}$ be reflexive Banach spaces equipped with norms $\|\cdot\|_{\mathbb{V}}$ and $\|\cdot\|_{\mathbb{Y}}$ respectively such that the following embeddings
$$\mathbb{X}\subset {\mathbb{V}}\subset\mathbb{Y}\subset {\mathbb{H}}$$
are continuous and dense, which then implies that  $\mathbb{V}$ and $\mathbb{Y}$ are also separable.
Identifying ${\mathbb{H}}$ with its dual space by the Riesz isomorphism, then we have the following Gelfand triple
\begin{equation*}
\mathbb{X}\subset {\mathbb{H}}(\simeq {\mathbb{H}}^*)\subset \mathbb{X}^*.
\end{equation*}
Since all respective spaces are separable, according to Kuratowski's theorem, it is easy to see that
$$\mathbb{X}\in\mathscr{B}({\mathbb{V}}),~{\mathbb{V}}\in\mathscr{B}(\mathbb{Y}),~\mathbb{Y}\in\mathscr{B}({\mathbb{H}}),~\mathscr{B}(\mathbb{Y})=\mathscr{B}({\mathbb{H}})\cap\mathbb{Y},~
\mathscr{B}({\mathbb{V}})=\mathscr{B}(\mathbb{Y})\cap {\mathbb{V}},~\mathscr{B}(\mathbb{X})=\mathscr{B}({\mathbb{V}})\cap\mathbb{X}.$$
The dualization between $\mathbb{X}$ and $\mathbb{X}^*$ is denoted by $_{\mathbb{X}^*}\langle\cdot,\cdot\rangle_\mathbb{X}$, then it is easy to see that $$_{\mathbb{X}^*}\langle\cdot,\cdot\rangle_\mathbb{X}|_{{\mathbb{H}}\times \mathbb{X}}=\langle\cdot,\cdot\rangle_{\mathbb{H}}.$$
Moreover, we use $L_2(U;{\mathbb{H}})$ to denote the space of all Hilbert-Schmidt operators from $U$ to ${\mathbb{H}}$, which is equipped with the Hilbert-Schmidt norm $\|\cdot\|_{L_2(U;{\mathbb{H}})}$.

For a Banach space $\mathbb{B}$ equipped with norm $\|\cdot\|_{\mathbb{B}}$,  we use $\mathscr{P}(\mathbb{B})$ to denote the space of all probability measures on $\mathbb{B}$ equipped with the weak topology. For any $p>0$, we define
$$\mathscr{P}_p(\mathbb{B})=\Big\{\mu\in\mathscr{P}(\mathbb{B}):\int_\mathbb{B}\|x\|_{\mathbb{B}}^p\mu(dx)<\infty\Big\},$$
which is a Polish space under the $p$-Wasserstein distance
$$\mathbb{W}_{p,\mathbb{B}}(\mu,\nu):=\inf_{\pi\in\mathscr{C}(\mu,\nu)}\Big(\int_{\mathbb{B}\times \mathbb{B}}\|x-y\|_{\mathbb{B}}^p\pi(dx,dy)\Big)^{\frac{1}{p\vee1}},~\mu,\nu\in\mathscr{P}_p(\mathbb{B}),$$
where $\mathscr{C}(\mu,\nu)$ stands for the set of all couplings for  $\mu$ and $\nu$.

Let $C([0,T];\mathbb{B})$ be the space of all continuous functions from $[0,T]$ to $\mathbb{B}$, which is a Banach space equipped with  the uniform norm given by
$$\|u\|_{T,\mathbb{B}}:=\sup_{t\in[0,T]}\|u_t\|_{\mathbb{B}},~u\in C([0,T];\mathbb{B}).$$
Then the corresponding $p$-Wasserstein distance on $\mathscr{P}_p(C([0,T];\mathbb{B}))$ is denoted by $\mathbb{W}_{p,T,\mathbb{B}}(\mu,\nu)$, i.e.
$$\mathbb{W}_{p,T,\mathbb{B}}(\mu,\nu)=\inf_{\pi\in\mathscr{C}(\mu,\nu)}\Big(\int_{C([0,T];\mathbb{B})\times C([0,T];\mathbb{B})}\|x-y\|_{T,\mathbb{B}}^p\pi(dx,dy)\Big)^{\frac{1}{p\vee1}},~\mu,\nu\in\mathscr{P}_p(C([0,T];\mathbb{B})).$$

Now we introduce a function class $\mathfrak{ N}^q,q\geq 1$: A lower semi-continuous function $\mathscr{N}_0:\mathbb{Y}\to[0,\infty]$ belongs to $\mathfrak{ N}^q$ if $\mathscr{ N}_0(x)=0$ implies $x=0$,
$$\mathscr{N}_0(cx)\leq c^q\mathscr{N}_0(x)~\text{for any}~c\geq0,~x\in \mathbb{Y},$$
and
$$\Big\{x\in\mathbb{Y}:\mathscr{N}_0(x)\leq 1\Big\}~\text{is compact in}~\mathbb{Y}.$$
%We remark that one can extend $\mathscr{N}_0$ to a $\mathscr{B}(\mathbb{X}^*)/\mathscr{B}([0,\infty))$-measurable function on $\mathbb{X}^*$ by setting $\mathscr{N}_0(x)=\infty$, $x\in \mathbb{X}^*\backslash\mathbb{Y}$ such that $\int_0^t\mathscr{N}_0(u_s)ds$ is well-defined for each $u\in C([0,\infty),\mathbb{X}^*)$.

Let $\mathscr{E}:=\big\{l_1,l_2,\ldots\big\}\subset\mathbb{X}$
 be an orthonormal basis (ONB)  of ${\mathbb{H}}$. Furthermore, we say a function $\mathscr{N}_1\in\mathfrak{M}^q,q\geq 1$ if $\mathscr{N}_1\in\mathfrak{ N}^q$ and  there exists a constant $C>0$  such that
\begin{equation}\label{es00}
\|x\|_{\mathbb{V}}^q\leq C\mathscr{N}_1(x)~~~\text{for any}~x\in\mathbb{V},
\end{equation}
 and
\begin{equation}\label{es01}
\mathscr{N}_1(x)\leq C_n\|x\|_{{\mathbb{H}}^n}^q,~x\in {\mathbb{H}}^n,~n\in\mathbb{N},
\end{equation}
where $C_n$ is a constant depending on $n$ and
$${\mathbb{H}}^n:=\text{span}\Big\{l_1,l_2,\cdots\,l_n\Big\}.$$

 For any $t\in[0,T],~R>0,~u\in C([0,T];\mathbb{H})$, we define the path $u^t_{\cdot}:[0,T]\to\mathbb{H}$ before time $t$ by
$$u^t_s:=u_{t\wedge s},~s\in[0,T],$$
and map $\pi_t(u):=u^t$. Then the marginal distribution before time $t$ of a probability measure $\mu\in\mathscr{P}(C([0,T];\mathbb{H}))$ is denoted by
$$\mu^t:=\mu\circ \pi_t^{-1}.$$
Define a time
\begin{equation*}
\tau_R^u:=\inf\Big\{t\geq 0:\|u_t\|_{\mathbb{H}}+\int_0^t\mathscr{N}_1(u_s)ds\geq R\Big\}\wedge T.
\end{equation*}
Then for any $\mu,\nu\in\mathscr{P}(C([0,T];\mathbb{H}))$, we can define the following ``local'' $L^2$-Wasserstein distance
\begin{equation}\label{localw}
\mathcal{W}_{2,R,\mathbb{H}}(\mu,\nu):=\inf_{\pi\in\mathscr{C}(\mu,\nu)}\Bigg(\int\|u^{\tau_R^u\wedge\tau_R^{v}}-v^{\tau_R^u\wedge\tau_R^{v}}\|_{T,\mathbb{H}}^2\pi(du,dv)\Bigg)^{\frac{1}{2}}.
\end{equation}

\subsection{Weak solutions}\label{weaksolution}

For the measurable maps
$$
A:[0,T]\times {\mathbb{V}}\times\mathscr{P}({\mathbb{V}})\rightarrow \mathbb{X}^*,~~\sigma:[0,T]\times {\mathbb{V}}\times\mathscr{P}({\mathbb{V}})\rightarrow L_2(U;{\mathbb{H}}),
$$
we consider the following MVSPDE
\begin{equation}\label{eqSPDE}
dX_t=A(t,X_t,\mathscr{L}_{X_t})dt+\sigma(t,X_t,\mathscr{L}_{X_t})dW_t,
\end{equation}
where $W_t$ is an $U$-valued cylindrical Wiener process defined on a complete filtered probability space $\left(\Omega,\mathscr{F},\{\mathscr{F}_t\}_{t\in[0,T]},\mathbb{P}\right)$ with right continuous filtration $\{\mathscr{F}_t\}$.
%which admits the following composition
%$$W_t=\sum_{k=1}^{\infty}\beta_k_te_k,$$
%where $\beta_k_t,k\in\mathbb{N}_+$, are independent standard one-dimensional Wiener processes.

Let $T>0$. We assume that $\mathbb{X}\subset\mathbb{H}$ is compact and there exist some constants $C>0$, $\gamma_2\geq\gamma_1>1$, $\theta_1\geq 2$  such that the following conditions hold for all $t\in[0,T]$.

 \begin{conditionH}\label{H1}
 $($Demicontinuity$)$ For any $v\in \mathbb{X}$, if $u_n,u\in\mathbb{V}$ with $u_n\to u$ in $\mathbb{Y}$ and $\mu_n$ converges  to $\mu$ in $\mathscr{P}_2(\mathbb{X}^*)$, then
\begin{equation*}
\lim_{n\to\infty}{}_{\mathbb{X}^*}\langle A(t,u_n,\mu_n),v\rangle_{\mathbb{X}}={}_{\mathbb{X}^*}\langle A(t,u,\mu),v\rangle_{\mathbb{X}}
\end{equation*}
and
\begin{equation*}
\lim_{n\to\infty}\| \sigma(t,u_n,\mu_n)^*v-\sigma(t,u,\mu)^*v\|_{U}=0.
\end{equation*}

\end{conditionH}

\begin{conditionH}\label{H3}
 $($Coercivity$)$ There exists a function $\mathscr{N}_1\in\mathfrak{M}^q$ for some $q\geq 2$ such that for any $u\in \mathbb{X}$, $\mu\in\mathscr{P}_{2}({\mathbb{H}})$,
\begin{eqnarray*}
2_{\mathbb{X}^*}\langle A(t,u,\mu),u\rangle_{\mathbb{X}}+\|\sigma(t,u,\mu)\|_{L_2(U;{\mathbb{H}})}^2\leq -\mathscr{N}_1(u)+C\big(1+\|u\|_{\mathbb{H}}^2+\mu(\|\cdot\|_{\mathbb{H}}^2)\big).
\end{eqnarray*}
\end{conditionH}

\begin{conditionH}\label{H4}
 $($Growth$)$ For any $u\in \mathbb{X}$, $\mu\in\mathscr{P}_{\theta_1}({\mathbb{H}})$,
\begin{equation}\label{c1}
\|A(t,u,\mu)\|_{\mathbb{X}^*}^{\gamma_1}\leq C\big(1+\mathscr{N}_1(u)+\mu(\|\cdot\|_{\mathbb{H}}^{\theta_1})\big)\big(1+\|u\|_{\mathbb{H}}^{\gamma_2}+\mu(\|\cdot\|_{\mathbb{H}}^{\theta_1})\big)
\end{equation}
and
\begin{equation}\label{c2}
\|\sigma(t,u,\mu)\|_{L_2(U;{\mathbb{H}})}^2\leq C\big(1+\|u\|_{\mathbb{H}}^2+\mu(\|\cdot\|_{\mathbb{H}}^2)\big),
\end{equation}
where $\mathscr{N}_1$ is the same as in $\mathbf{H3}$.
\end{conditionH}

%Let $\mathscr{E}:=\Big\{l_1,l_2,\ldots\Big\}\subset\mathbb{X}$
%be an orthonormal basis (ONB for short) of ${\mathbb{H}}$.
%Let $\Omega_0:=C([0,T];\mathbb{X}^*)$ which is a Polish space with respect to (w.r.t.~for short) the uniform norm. Then we use $\omega$ to denote
%a path in $\Omega_0$ and $w_t(\omega):=\omega_t$ to denote the coordinate process.
%  Define the $\sigma$-algebra by
%$$\mathscr{F}_t:=\sigma\Big\{w_s:~s\leq t\Big\}.$$

Now we recall the notion of (probabilistically) weak solutions to MVSPDE (\ref{eqSPDE}).
\begin{definition}\label{de1} $($Weak solution$)$ A pair $(X,W)$ is called a (probabilistically) weak solution to (\ref{eqSPDE}) with initial law $\mu_0\in\mathscr{P}(\mathbb{H})$, if there exists a stochastic basis $(\Omega,\mathscr{F},\{\mathscr{F}_t\}_{t\in[0,T]},\mathbb{P})$ such that $X$ is an $\{\mathscr{F}_t\}$-adapted process with $\mathbb{P} \circ X_0^{-1}=\mu_0$,   $W$ is an $U$-valued cylindrical Wiener process on $(\Omega,\mathscr{F},\{\mathscr{F}_t\}_{t\in[0,T]},\mathbb{P})$ and the following holds:

\vspace{1mm}
(i) $X\in C([0,T];\mathbb{X}^*)\cap L^q([0,T];{\mathbb{V}})$ $\mathbb{P}$-a.s., $\mathscr{L}_{X}\in\mathscr{P}(C([0,T];\mathbb{X}^*)\cap L^{q}([0,T];\mathbb{V}))$;

\vspace{2mm}
(ii) $\int_0^T\|A(s,X_s,\mathscr{L}_{X_s})\|_{\mathbb{X}^*}ds+\int_0^T\|\sigma(s,X_s,\mathscr{L}_{X_s})\|_{L_2(U;{\mathbb{H}})}^2ds<\infty$ $\mathbb{P}$-a.s.;

\vspace{2mm}
(iii) For any $t\in[0,T]$, $l\in\mathbb{X}$, $\mathbb{P}$-a.s.
$${}_{\mathbb{X}^*}\langle X_t,l\rangle_{\mathbb{X}}={}_{\mathbb{X}^*}\langle X_0,l\rangle_{\mathbb{X}}-\int_0^t{}_{\mathbb{X}^*}\langle A(s,X_s,\mathscr{L}_{X_s}),l\rangle_{\mathbb{X}}ds+
\langle\int_0^t \sigma(s,X_s,\mathscr{L}_{X_s})dW_s,l\rangle_{{\mathbb{H}}}.$$
\end{definition}
\begin{definition}\label{de4} $($Weak uniqueness$)$
We say MVSPDE (\ref{eqSPDE}) is weakly unique, if $(\tilde{X},\tilde{W})$ on the stochastic basis $(\tilde{\Omega},\tilde{\mathscr{F}},\{\tilde{\mathscr{F}}_t\}_{t\in[0,T]},\tilde{\mathbb{P}})$ and $(\bar{X},\bar{W})$ on $(\bar{\Omega},\bar{\mathscr{F}},\{\bar{\mathscr{F}}_t\}_{t\in[0,T]},\bar{\mathbb{P}})$ are two weak solutions to (\ref{eqSPDE}) in the sense of Definition \ref{de1}, then $\mathscr{L}_{\tilde{X}_0}=\mathscr{L}_{\bar{X}_0}$ implies that $\mathscr{L}_{\tilde{X}_t}=\mathscr{L}_{\bar{X}_t}$.
\end{definition}

The existence of weak solutions to MVSPDE (\ref{eqSPDE}) is given in the following theorem.
\begin{theorem}\label{th1}
Suppose that $\mathbf{H1}$-$\mathbf{H3}$ hold.
Then for any $X_0\sim\mu_0\in \mathscr{P}_p(\mathbb{H})$ with $p>\eta_0:=\theta_1\vee \frac{2(q+\gamma_2)}{\gamma_1}\vee 2(2+\gamma_2)$,
%where
%$$\eta_0:=\theta_1\vee \frac{2(q+\gamma_2)}{\gamma_1}\vee 2(2+\gamma_2),$$
MVSPDE (\ref{eqSPDE}) has a weak solution in the sense of Definition \ref{de1}. In addition,
\begin{equation}\label{es48}
\mathbb{E}\Big[\sup_{t\in[0,T]}\|X_t\|_{{\mathbb{H}}}^p\Big]+\mathbb{E}\int_0^T\mathscr{N}_1(X_t)dt+\mathbb{E}\int_0^T\|X_t\|_{{\mathbb{H}}}^{p-2}\mathscr{N}_1(X_t)dt<\infty.
\end{equation}

\end{theorem}

\begin{remark}
 Since we only consider the existence of weak solutions to MVSPDE (\ref{eqSPDE}) in Theorem \ref{th1}, we do not need to assume any monotone or locally monotone conditions on the coefficients, as is usually done in the existing works for classical SPDEs or McKean-Vlasov SPDEs (cf.~e.g.~\cite{HHL,LR2,RSZ}). It  can be applied to a considerable larger class of SPDEs, in particular,  stochastic 3D Navier-Stokes equations with distribution dependent coefficients (see Subsection \ref{example} below).

\end{remark}

\subsection{Strong solutions}\label{strongsolution}
In this subsection, we intend to show the pathwise uniqueness of solutions to (\ref{eqSPDE}), which combining with Theorem \ref{th1} and the modified Yamada-Watanabe theorem  implies the existence and uniqueness of strong and weak solutions.

Recall that $$\mathbb{X}\subset {\mathbb{V}}\subset\mathbb{Y}\subset {\mathbb{H}}$$
are the spaces defined as in Subsection \ref{weaksolution}.
%where the inclusions on the right are just given by the restriction of the respective linear functions.
In order to investigate (probabilistically) strong solutions to  (\ref{eqSPDE}), we consider the following Gelfand triple
\begin{equation}\label{es0}
{\mathbb{V}}\subset\mathbb{H}(\simeq\mathbb{H}^*)\subset\mathbb{V}^*.
\end{equation}
Note that $\mathbb{X}\subset {\mathbb{V}}$ is continuous and dense, we have
$${\mathbb{V}}^*\subset \mathbb{X}^*,$$
and
\begin{equation}\label{es53}
{}_{{\mathbb{V}}^*}\langle u, v\rangle_{{\mathbb{V}}}=\langle u,v\rangle_{\mathbb{H}}={}_{\mathbb{X}^*}\langle u,v\rangle_{\mathbb{X}},~u\in {\mathbb{H}},v\in\mathbb{X}.
\end{equation}

Recall the constant $q$ given in $\mathbf{H2}$. We assume that for the maps
$$
A:[0,T]\times {\mathbb{V}}\times\mathscr{P}(\mathbb{V})\rightarrow {\mathbb{V}}^*,~~\sigma:[0,T]\times {\mathbb{V}}\times\mathscr{P}(\mathbb{V})\rightarrow L_2(U;\mathbb{H})
$$
there exist $C>0$, $\beta>1,\theta_2\geq 2$ such that the following conditions hold.
\begin{conditionH}\label{H2}
 $($Local Monotonicity$)$ There exists $C,R_0>0$ such that for any $R>R_0$,  $u,v\in \mathbb{V}$, $\mu,\nu\in\mathscr{P}_{\theta_2}(C([0,T];\mathbb{H}))$ and $t\in[0,T]$,
\begin{eqnarray*}
&&_{{\mathbb{V}}^*}\langle A(t,u,\mu_t)-A(t,v,\nu_t),u-v\rangle_{\mathbb{V}}
\nonumber\\
&&\leq
(C+\rho(u,\mu_t)+\eta(v,\nu_t))(\|u-v\|_{\mathbb{H}}^2+\mathcal{W}_{2,R,\mathbb{H}}(\mu^t,\nu^t)^2)
\end{eqnarray*}
and
\begin{eqnarray*}
&&\|\sigma(t,u,\mu_t)-\sigma(t,v,\nu_t)\|_{L_2(U;{\mathbb{H}})}^2
\nonumber\\
&&\leq
(C+\rho(u,\mu_t)+\eta(v,\nu_t))(\|u-v\|_{\mathbb{H}}^2+\mathcal{W}_{2,R,\mathbb{H}}(\mu^t,\nu^t)^2),
\end{eqnarray*}
where $\rho,\eta:{\mathbb{V}}\times \mathscr{P}_{\theta_2}(\mathbb{H})\to [0,\infty)$ are measurable functions that satisfy
\begin{equation}\label{es22}
\rho(u,\mu) + \eta(u,\mu)\leq C\big(1+\mathscr{N}_1(u)+\mu(\|\cdot\|_{\mathbb{H}}^{\theta_2})\big)\big(1+\|u\|_{\mathbb{H}}^{\beta}+\mu(\|\cdot\|_{\mathbb{H}}^{\theta_2})\big),~u\in {\mathbb{V}},\mu\in\mathscr{P}_{\theta_2}(\mathbb{H}).
\end{equation}
%\begin{equation}\label{es23}
%\eta(v)\leq C\big(1+\mathscr{N}_1(v)\big)\big(1+\|v\|_{\mathbb{H}}^{\beta}\big),~v\in {\mathbb{V}},
%\end{equation}

\end{conditionH}

\begin{conditionH}\label{H5}
$($Growth$)$ For any $t\in[0,T]$, $u\in {\mathbb{V}}$ and  $\mu\in\mathscr{P}_{ \theta_2}({\mathbb{H}})$,
\begin{equation}\label{es45}
\|A(t,u,\mu)\|_{{\mathbb{V}}^*}^{\frac{q}{q-1}}\leq C\big(1+\mathscr{N}_1(u)+\mu(\|\cdot\|_{\mathbb{H}}^{\theta_2})\big)\big(1+\|u\|_{\mathbb{H}}^{\beta}+\mu(\|\cdot\|_{\mathbb{H}}^{\theta_2})\big).
\end{equation}

\end{conditionH}

%\begin{remark}
%It should be noted that compared to Section \ref{sec4}, the maps $A$ considered in this section are restricted to a smaller range ${\mathbb{V}}^*$ instead of  $\mathbb{X}^*$. Moreover, the new Hilbert space ${\tilde{\mathbb{H}}}$ might be different from ${\mathbb{H}}$, which is mainly used to cover some quasilinear SPDEs  such as the stochastic porous media equation, see subsection \ref{sec5.3} for details.
%\end{remark}
\vspace{2mm}
Now we present the definition of (probabilistically) strong solutions to  (\ref{eqSPDE}).
\begin{definition}\label{de2}$($Strong solution$)$
We say there exists a (probabilistically) strong solution to (\ref{eqSPDE}) if for every probability space $(\Omega,\mathscr{F},\{\mathscr{F}_t\}_{t\in[0,T]},\mathbb{P})$ with an $U$-valued cylindrical $\{\mathscr{F}_t\}$-Wiener process $W$, there exists
an $\{\mathscr{F}_t\}$-adapted process $X$ such that

\vspace{2mm}
(i) $X\in C([0,T];\mathbb{H})\cap L^q([0,T];\mathbb{V})$ $\mathbb{P}$-a.s., $\mathscr{L}_{X}\in\mathscr{P}(C([0,T];\mathbb{H})\cap L^{q}([0,T];\mathbb{V}))$;

\vspace{2mm}
(ii) $\int_0^T\|A(s,X_s,\mathscr{L}_{X_s})\|_{\mathbb{V}^*}ds+\int_0^T\|\sigma(s,X_s,\mathscr{L}_{X_s})\|_{L_2(U;\mathbb{H})}^2ds<\infty$ $\mathbb{P}$-a.s.;

\vspace{2mm}
(iii)
$X_t=X_0+\int_0^t A(s,X_s,\mathscr{L}_{X_s})ds+\int_0^t \sigma(s,X_s,\mathscr{L}_{X_s})dW_s,~t\in[0,T],~\mathbb{P}\text{-a.s.}$
holds in ${\mathbb{V}}^*$.
\end{definition}

The existence and uniqueness of  strong solutions to (\ref{eqSPDE}) is stated in the following theorem.
\begin{theorem}\label{th2}
Suppose that $\mathbf{H1}$-$\mathbf{H5}$ hold.
Then for any $X_0\in  L^p(\Omega,\mathscr{F}_0,\mathbb{P};{\mathbb{H}})$ with $p>\eta_1:=\eta_0\vee (\beta+2)\vee \theta_2$,
MVSPDE (\ref{eqSPDE}) has a unique  strong solution in the  sense of Definition \ref{de2}. In addition,
\begin{equation}\label{es47}
\mathbb{E}\Big[\sup_{t\in[0,T]}\|X_t\|_{{\mathbb{H}}}^p\Big]+\mathbb{E}\int_0^T\mathscr{N}_1(X_t)dt+\mathbb{E}\int_0^T\|X_t\|_{{\mathbb{H}}}^{p-2}\mathscr{N}_1(X_t)dt<\infty.
\end{equation}
%and
%\begin{equation}\label{es49}
%\mathbb{E}\Big[\sup_{t\in[0,T]}\|X_t\|_{{\tilde{\mathbb{H}}}}^2\Big]<\infty.
%\end{equation}
\end{theorem}

\begin{remark}\label{re1}
By the modified Yamada-Watanabe theorem (cf.~\cite[Lemma 2.1]{HW1}), it turns out that Theorem \ref{th2} implies the uniqueness of weak solutions (also martingale solutions) to MVSPDE  (\ref{eqSPDE}) in the sense of Definition \ref{de4}.
\end{remark}

\subsection{Examples/Applications}\label{example}

 The main results of Theorem \ref{th1} and \ref{th2} can be used to get the existence of weak solutions and the existence of unique strong solutions for a large class of distribution dependent SPDEs. In particular, compared to the existing works for SDEs,  our results are also new in the finite-dimensional case. Thus we first apply our results to finite-dimensional SDEs, then we illustrate our main results by various concrete Mckean-Vlasov SPDEs such as  stochastic NSEs, stochastic Cahn-Hilliard equations and stochastic Kuramoto-Sivashinsky equations.

In addition to the examples mentioned above, our main results can also be applied to many other McKean-Vlasov type SPDEs where the coefficients satisfy local monotonicity and polynomial growth conditions.
For instance, the stochastic Burgers equation, stochastic $p$-Laplace equation, stochastic magneto-hydrodynamic equation, stochastic Boussinesq equation, stochastic magnetic B\'{e}nard problem, stochastic Leray-$\alpha$ model, stochastic Allen-Cahn equation, stochastic power law fluid equation, stochastic Ladyzhenskaya model, stochastic Swift-Hohenberg equation and stochastic Liquid crystal model,  for those models one could refer to e.g.~\cite{AV,LR2,LR13,LR1,RSZ,RSZ1}. Here we omit the details to keep down the length of the paper.

\subsubsection*{Example 1: MVSDEs in finite dimensions}\label{sec5.1}
Now we apply our results established in Theorems \ref{th1}, \ref{th2}  to the case of finite dimensions, which  seem to be not covered by the existing works for  distribution dependent SDEs with coefficients depending continuously on the measure variable.

Let $U={\mathbb{H}}=\mathbb{X}=\mathbb{Y}=\mathbb{V}=\mathbb{X}^*=\RR^d$, we use the notations $|\cdot|$ (resp. $\langle\cdot,\cdot\rangle$) to denote the norm (resp. scalar product) on $\RR^d$ and $\|\cdot\|$ to denote the Hilbert-Schmidt norm for linear maps from $\RR^d$ to $\RR^d$. Consider the following MVSDEs
\begin{equation}\label{eq4}
dX_t=b(t,X_t,\mathscr{L}_{X_t})dt+\sigma(t,X_t,\mathscr{L}_{X_t})dW_t,
\end{equation}
where $W_t$ is an $\RR^d$-valued standard Wiener process defined on  $\left(\Omega,\mathscr{F},\{\mathscr{F}_t\}_{t\in[0,T]},\mathbb{P}\right)$.

For the coefficients $b,\sigma$ we assume that
there are some constants $C>0$, $\kappa\geq 2$  such that the following conditions hold.
 \begin{conditionC}\label{C1}
For any $t\in[0,T]$, $b(t,\cdot,\cdot),\sigma(t,\cdot,\cdot)$ are continuous on $\RR^d\times\mathscr{P}_2(\RR^d)$.

\end{conditionC}

\begin{conditionC}\label{C2}
For any $t\in[0,T]$, $u\in \RR^d$ and $\mu\in\mathscr{P}_{2}(\RR^d)$,
\begin{eqnarray*}
2\langle b(t,u,\mu),u\rangle+\|\sigma(t,u,\mu)\|^2\leq C\big(1+|u|^2+\mu(|\cdot|^2)\big).
\end{eqnarray*}
\end{conditionC}

\begin{conditionC}\label{C3}
  For any $t\in[0,T]$, $u\in \RR^d$ and $\mu\in\mathscr{P}_{\kappa}(\RR^d)$,
\begin{equation*}
|b(t,u,\mu)|^2\leq C\big(1+|u|^{\kappa}+\mu(|\cdot|^{\kappa})\big),
\end{equation*}
\begin{equation*}
\|\sigma(t,u,\mu)\|^2\leq C\big(1+|u|^2+\mu(|\cdot|^2)\big).
\end{equation*}
\end{conditionC}

%We show the existence and uniqueness of solutions to the MVSDES (\ref{eq4}) by the following theorem.
\begin{theorem}\label{th55}
Suppose that $\mathbf{C1}$-$\mathbf{C3}$ hold.
Then for any $X_0\in L^p(\Omega,\mathscr{F}_0,\mathbb{P};\RR^d)$ with $p>\eta_1:=4+2\kappa$,
(\ref{eq4}) has a weak solution satisfying
\begin{equation*}
\mathbb{E}\Big[\sup_{t\in[0,T]}|X_t|^p\Big]<\infty.
\end{equation*}
%Furthermore, if $\mathbf{C4}$ also holds, then (\ref{eq4}) admits a unique strong solution.
\end{theorem}
\begin{proof}
The results follow directly from Theorem \ref{th1}  by taking $\mathscr{N}_1(u)=|u|^2$ with $q=2$.
\end{proof}
\begin{remark}
The existence  of strong solutions to (\ref{eq4}) under $\mathbf{C1}$-$\mathbf{C3}$ and certain local monotonicity conditions have been studied in the recent work \cite{HHL}. However,   Theorem \ref{th55} is new for the existence of weak solutions to (\ref{eq4}) under $\mathbf{C1}$-$\mathbf{C3}$, which is not covered by  \cite{HHL} or other papers in the literature and has its own interest.
\end{remark}

\vspace{2mm}
In the sequel, we aim to illustrate our main results by applications to a number of SPDE models. To this end, we  first introduce some notations and spaces which are commonly used in analysis.

Let $\mathscr{O}\subset\mathbb{R}^d$, $d\geq 1$,  be a bounded open domain with smooth boundary.
Let $(L^p(\mathscr{O};\mathbb{R}^d),\|\cdot\|_{L^p})$ be the space of $L^p$-integrable functions on $\mathscr{O}$, and for any $m\geq0$, we denote by $(W^{m,p}(\mathscr{O};\mathbb{R}^d),\|\cdot\|_{m,p})$ the Sobolev space defined on $\mathscr{O}$ taking values in $\mathbb{R}^d$ with the norm
$$ \|f\|_{m,p} := \left( \sum_{ 1\leq|\alpha|\leq m} \int_{\mathscr{O}} |D^\alpha f|^pd x \right)^\frac{1}{p}.$$
For the sake of simplicity,  we use $\|\cdot\|_{m}$ to denote the norm on $W^{m,2}(\mathscr{O};\mathbb{R}^d)$. As usual, we also use $C_{c}^{\infty}(\mathscr{O}; \mathbb{R}^{d})$ to denote the space of all infinitely differentiable $d$-dimensional vector fields with compact support in the domain $\mathscr{O}$.

The following is the Sobolev inequality, which will be frequently used in the proof.
\begin{lemma}\label{lem7}
Let $q>1$, $p\in[q, \infty)$ and
$$\frac{1}{q}+\frac{m}{d}=\frac{1}{p}.$$
Suppose that $f\in W^{m,p}(\mathscr{O};\mathbb{R}^d)$, then $f\in L^q(\mathscr{O};\mathbb{R}^d)$ and there exists $C>0$ such that the following inequality holds
$$\|f\|_{L^q}\leq C\|f\|_{m,p}.$$
\end{lemma}

We also recall the Gagliardo-Nirenberg interpolation inequality (cf. \cite[Theorem 2.1.5]{Taira}) for later use.
\begin{lemma}\label{lem8}
If $m,n\in\mathbb{N}$ and $q\in[1,\infty]$ satisfying
$$
\frac{1}{q}=\frac{1}{2}+\frac{n}{d}-\frac{m \theta}{d},\ \frac{n}{m}\le\theta\le1,
$$
then there is a constant $C>0$ such that
\begin{equation}
\|f\|_{n,q}\le C\|f\|_{m}^{\theta}\|f\|_{L^2}^{1-\theta},\ \ f\in W^{m,2}(\mathscr{O}; \mathbb{R}^d).\label{GN_inequality}
\end{equation}
\end{lemma}

%\subsection{Stochastic Burgers type equations}
\subsubsection*{Example 2: Stochastic Navier-Stokes equations}\label{SecNS}
We first apply the main results to the stochastic Navier-Stokes equations.
Let $d=2$~or~$3$, $\mathscr{O}\subset\mathbb{R}^d$ is a bounded open domain with smooth boundary.  We consider the following McKean-Vlasov type stochastic Navier-Stokes equations
 \begin{eqnarray}\label{eqns}
\left\{
 \begin{aligned}
&dX_t=\big[\nu \Delta X_t-(X_t\cdot \nabla) X_t+\nabla p_t+K(t,X_t,\mathscr{L}_{X_t})\big]dt+\sigma(t,X_t,\mathscr{L}_{X_t})dW_t,\\
   &{\rm div}(X_t)=0, \\
    &X_t=0, ~~\text{on}~\partial\mathscr{O},
  \end{aligned}
\right.
\end{eqnarray}
where $\nu$ is the viscosity constant and $W_t$ is an  $U$-valued cylindrical Wiener process defined on $\left(\Omega,\mathscr{F},\{\mathscr{F}_t\}_{t\in[0,T]},\mathbb{P}\right)$.  Without loss of generality, we assume $\nu\equiv1$ in the following.

%We consider the stochastic Navier-Stokes equations of McKean-Vlasov type,
% \begin{eqnarray}\label{Eq2}
%\left\{
% \begin{aligned}
%&dX_t=\Big[\nu\Delta X_t-(X_t \cdot \nabla)X_t+\nabla p_t+ K(t,X_t,\mathscr{L}_{X_t})\Big]dt+\sigma(t,X_t,\mathscr{L}_{X_t})dW_t,\\
%&{\rm div}(X_t)=0 , \\
%    &X_t=0, ~~\text{on}~\mathscr{O}.
%  \end{aligned}
%\right.
%\end{eqnarray}
\vspace{1mm}
Let
\begin{equation}\label{cinfty}
\mathscr{V}:=\{u\in C_{c}^{\infty}(\mathscr{O}; \mathbb{R}^{d}), {\rm div} (u)=0\}.
\end{equation}
For $m\geq0$, we define
$$W_{{div}}^{m,2}(\mathscr{O}):=\text{the completion of}~\mathscr{V}~\text{in}~ W^{m,2}(\mathscr{O}, \mathbb{R}^{d})$$
with the norm denoted by $\|\cdot\|_{m}$.
We set in particular,
$$\mathbb{Y}=\mathbb{H}:=W_{{div}}^{0,2}(\mathscr{O}),~~\mathbb V:=W_{{div}}^{1,2}(\mathscr{O}),~~\mathbb{X}:=W_{div}^{2+d,2}(\mathscr{O}).$$
In space $\mathbb{H}$, the inner product and the norm  are denoted by $(\cdot,\cdot)$ and $\|\cdot\|_{L^2}$ respectively. By the Poincar\'{e} inequality,  we consider the following inner product and (equivalent) norm in space $\mathbb V$
$$ ((u,v)):=( \nabla u,\nabla v), ~\|u\|_{1}:=\|\nabla u\|_{L^2},~u,v\in \mathbb V.$$

Identifying $\mathbb H$ with its dual space by the Riesz isomorphism, then we have
\begin{equation*}
W_{div}^{m,2}(\mathscr{O})\subset \mathbb H(\simeq \mathbb H^*)\subset {W_{div}^{m,2}(\mathscr{O})}^*.
\end{equation*}
In particular, the dual pair between ${\mathbb V}^*$ and ${\mathbb V}$ is denoted by $\langle\cdot,\cdot\rangle$, the norm of ${\mathbb V}^*$ is denoted by $\|\cdot\|_{-1}$.

Let $\mathcal{P}_{\mathbb{H}}:L^{2}(\mathscr{O};\mathbb{R}^{d})\to \mathbb H$ be the orthogonal projection operator on $L^{2}(\mathscr{O};\mathbb{R}^{d})$ onto $\mathbb H$, which is called the Leray-Helmholtz
projection. We define the Stokes operator by
$$Au:=\mathcal{P}_{\mathbb{H}}\big[\Delta u\big],~~\text{for}~u\in \mathscr{D}(A):=W_{div}^{2,2}(\mathscr{O}),$$
and define the bilinear operator $B(\cdot,\cdot): \mathbb V\times \mathbb V\rightarrow \mathbb V^*$  as
$$B(u,v):=\mathcal{P}_{\mathbb{H}}\big[(u\cdot\nabla)v\big],~~\text{for}~u,v\in\mathbb V.$$
Let
$$B(u):=B(u,u),~~\text{for}~u\in\mathbb V.$$
It is well known that for any $u,v,z\in \mathbb V$
\begin{eqnarray}
\langle B(u,v),z\rangle=-\langle B(u,z),v\rangle,\quad \langle B(u,v),v\rangle=0.\label{P1}
\end{eqnarray}

In this section, we suppose that the measurable maps
$$
K:[0,T]\times\mathbb V\times \mathscr{P}(\mathbb{V})\rightarrow \mathbb H,~~\sigma:[0,T]\times \mathbb V\times \mathscr{P}(\mathbb{V})\rightarrow L_2(U;\mathbb H),
$$
satisfy the following conditions.
\begin{conditionC}\label{H7}
For any $t\in[0,T]$ and $v\in \mathbb{X}$, if $u_n,u\in\mathbb{V}$ with $u_n\to u$ in $\mathbb{Y}$ and $\mu_n$ converges  to $\mu$ in $\mathscr{P}_2(\mathbb{X}^*)$, then
\begin{equation*}
\lim_{n\to\infty}\langle K(t,u_n,\mu_n),v\rangle_{\mathbb{H}}=\langle K(t,u,\mu),v\rangle_{\mathbb{H}}
\end{equation*}
and
\begin{equation*}
\lim_{n\to\infty}\|\sigma(t,u_n,\mu_n)^*v\|_{U}=\| \sigma(t,u,\mu)^*v\|_{U}.
\end{equation*}
\end{conditionC}

\begin{conditionC}\label{H8}
There exists a constant $C>0$ such that for any $t\in[0,T]$, $u\in \mathbb V$, $\mu\in \mathscr{P}_2(\mathbb H)$,
\begin{eqnarray*}
\|K(t,u,v)\|_{\mathbb H}+\|\sigma(t,u,\mu)\|_{L_2(U;\mathbb H)}\leq C\big(1+\|u\|_{\mathbb H}+\mu(\|\cdot\|_{\mathbb H})\big).
\end{eqnarray*}
\end{conditionC}

We now rewrite (\ref{eqns}) as
 \begin{eqnarray}\label{eqns222}
\left\{
 \begin{aligned}
&dX_t=\big[ A X_t-B(X_t)+K(t,X_t,\mathscr{L}_{X_t})\big]dt+\sigma(t,X_t,\mathscr{L}_{X_t})dW_t,\\
   &{\rm div}(X_t)=0, \\
    &X_t=0, ~~\text{on}~\partial\mathscr{O}.
  \end{aligned}
\right.
\end{eqnarray}

Moreover, if $m>\frac{d}{2}+1$, the following Sobolev embeddings hold
$$W_{{div}}^{m-1,2}(\mathscr{O})\hookrightarrow\mathscr{C}_{b}(\mathscr{O}, \mathbb{R}^{d})\hookrightarrow L^\infty(\mathscr{O}, \mathbb{R}^{d}).$$
Here, $\mathscr{C}_{b}(\mathscr{O}; \mathbb{R}^{d})$ denotes the space of continuous and bounded $\mathbb R^d$-valued functions defined on $\mathscr{O}$. If $u,v\in\mathbb{V}$ and $z\in W_{{div}}^{m,2}(\mathscr{O})$ with $m>\frac{d}{2}+1$, there exists a constant $C>0$ such that
\begin{eqnarray*}
|\langle B(u,v),z\rangle|\leq \|u\|_{L^2}\|v\|_{L^2}\|\nabla z\|_{L^\infty}\leq C\|u\|_{L^2}\|v\|_{L^2}\|z\|_{m}.
\end{eqnarray*}
Moreover, we have the following results (cf.~Lemma 6.1 in \cite{GRZ}).
\begin{lemma} \label{Property B2}
For any $u,v,z\in \mathbb V$, there exists some generic constant $C$ such that\\
$$\|A(u)-A(v)\|_{\mathbb{X}^*}\leq C\|u-v\|_{L^2}$$
and
\begin{eqnarray*}
\|B(u)-B(v)\|_{\mathbb{X}^*}\leq C(\|u\|_{L^2}+\|v\|_{L^2})\|u-v\|_{L^2}.
\end{eqnarray*}
In particular, we can extend the operators $A$ and $B$ to $\mathbb{H}$ such that for $u\in\mathbb{H}$, $A(u)\in{\mathbb{X}^*}, B(u)\in{\mathbb{X}^*}$.
\end{lemma}

%\begin{conditionC}\label{H9}
% There exists constant $C>0$ such that for any $u,u_1,u_2\in \mathbb V$, $\mu,\mu_1,\mu_2\in \mathscr{P}_2(\mathbb H)$ and $t\in[0,T]$,
%\begin{eqnarray}\label{c9}
%\!\!\!\!\!\!\!\!&&\|\sigma(t,u_1,\mu_1)-\sigma(t,u_2,\mu_2)\|_{L_2(U;\mathbb H)}
%\nonumber \\
%\leq \!\!\!\!\!\!\!\!&&C
%\big(1+\|u_1\|_{\mathbb H}+\|u_2\|_{\mathbb H}+\mu_1(\|\cdot\|_{\mathbb H})+\mu_2(\|\cdot\|_{\mathbb H})\big)\|u_1-u_2\|_{\mathbb H}
%\nonumber \\
%\!\!\!\!\!\!\!\!&&
%+C\big(1+\mu_1(\|\cdot\|_{\mathbb H})+\mu_2(\|\cdot\|_{\mathbb H})\big)\mathbb{W}_{2,\mathbb{H}}(\mu_1,\mu_2)
%\end{eqnarray}
%and
%\begin{equation*}
%\|\sigma(t,u,\mu)\|_{L_2(U;\mathbb H)}\leq C\big(1+\|u\|_{\mathbb H}+\mu(\|\cdot\|_{\mathbb H})\big).
%\end{equation*}
%\end{conditionC}

In order to state the main results of this subsection, we define the function $\mathscr{N}_1$ on $\mathbb{Y}$ as follows,
$$
\mathscr{N}_1(u):=\begin{cases} \| u\|_{1}^2,~~~~\text{if}~u\in W_{{div}}^{1,2}(\mathscr{O})&\quad\\
+\infty,~~~~~\text{otherwise}.&\quad\end{cases}$$
Then we can easily deduce that $\mathscr{N}_1\in\mathfrak{M}^2$.

\vspace{1mm}
We obtain the existence and uniqueness of solutions to (\ref{eqns}) by the following theorem.
\begin{theorem}\label{thn1}
Suppose that  $\mathbf{C4}$-$\mathbf{C5}$ hold and  $\mathbf{H4}$ holds for $K,\sigma$ with $q=\beta=\theta_2=2$.
Then for any $X_0\in L^p(\Omega,\mathscr{F}_0,\mathbb{P};{\mathbb{H}})$ with $p>8$,
MVSNSE (\ref{eqns222}) has a weak solution in the sense of Definition \ref{de1}. Furthermore, the following estimates hold
\begin{equation*}
\mathbb{E}\Big[\sup_{t\in[0,T]}\|X_t\|_{L^2}^p\Big]+\mathbb{E}\int_0^T\|X_t\|_{1}^2dt<\infty.
\end{equation*}
 In particular, in the case $d=2$, (\ref{eqns222}) has a unique (probabilistically) strong solution in the  sense of Definition \ref{de2}.
\end{theorem}

\begin{proof}
For the existence of weak solutions to MVSNSE (\ref{eqns222}), it suffices to check that assumptions $\mathbf{H1}$-$\mathbf{H3}$ hold.
By Lemma \ref{Property B2} and assumptions $\mathbf{C4}$-$\mathbf{C5}$,
it is easy to see that the demicontinuity condition $\mathbf{H1}$ holds.

For any $u\in\mathbb X$,
$$_{\mathbb{X}^*}\langle B(u),u\rangle_{\mathbb{X}}= \langle B(u),u\rangle=0,$$
by assumption $\mathbf{C5}$, for any $\mu\in\mathscr{P}_{2}({\mathbb{H}})$, we have
\begin{eqnarray*}
2_{\mathbb{X}^*}\langle A(u)-B(u)+K(t,u,\mu),u\rangle_{\mathbb{X}}=\!\!\!\!\!\!\!\!&&-2\|u\|_{1}^2+2_{\mathbb{X}^*}\langle K(t,u,\mu),u\rangle_{\mathbb{X}}
\nonumber \\
\leq\!\!\!\!\!\!\!\!&& -\mathscr{N}_1(u)+C\big(1+\|u\|_{L^2}^2+\mu(\|\cdot\|_{L^2}^2)\big),
\end{eqnarray*}
then the coercivity condition $\mathbf{H2}$ holds.

In view of Lemma \ref{Property B2} and $\mathbf{C5}$, we have
\begin{eqnarray*}
\|A(u)-B(u)+K(t,u,\mu)\|_{\mathbb{X}^*}^2\leq C\big(1+\|u\|_{L^2}^2+\mu(\|\cdot\|_{L^2}^2)\big).
\end{eqnarray*}
Thus $\mathbf{H3}$ holds for $\gamma_1=\gamma_2=\theta_1=2$.
Then according to Theorem \ref{th1}, MVSNSE (\ref{eqns222}) has a weak solution in the sense of Definition \ref{de1}.

Now we turn to prove the local monotonicity condition $\mathbf{H4}$ in the case of $d=2$.
Recall the following estimate
\begin{eqnarray}
|\langle B(u,v),z\rangle|\leq C\|u\|_{L^4}\|v\|_{L^4}\|\nabla z\|_{L^2}\leq C\|u\|_{L^2}^{\frac{1}{2}}\|u\|_{1}^{\frac{1}{2}}\|v\|_{L^2}^{\frac{1}{2}}
\|v\|_{1}^{\frac{1}{2}}\|z\|_{1},~u,v,z\in \mathbb V. \label{P21}
\end{eqnarray}
In particular,
\begin{eqnarray}\label{unies}
|\langle B(u)-B(v),u-v\rangle|=\!\!\!\!\!\!\!\!&&|\langle B(u-v),v\rangle|
\nonumber \\
\leq\!\!\!\!\!\!\!\!&&
 C\|u-v\|_{L^2}\|u-v\|_1\|v\|_1
 \nonumber \\
\leq\!\!\!\!\!\!\!\!&&\frac{1}{2}\|u-v\|_1^2
+ C\|v\|_1^2\|u-v\|_{L^2}^2.
\end{eqnarray}
Then for any $u,v\in {\mathbb{V}}$,
\begin{equation*}
\langle A(u)-B(u)-A(v)+B(v)
,u-v\rangle
\leq
C\big(1+\|u\|_1^2+\|u\|_{L^2}^2+\|v\|_{L^2}^2\big)\|u-v\|_{L^2}^2,
\end{equation*}
which yields that the local monotonicity condition $\mathbf{H4}$ holds with $\rho(u)=\|u\|_1^2+\|u\|_{L^2}^2$ and $\eta(v)=\|v\|_{L^2}^{2}$. It is also easy to see that \eref{es22} holds for $\beta=2$.

By \eref{P21}, we get
\begin{eqnarray*}
\|B(u)\|_{-1}^2\leq C\|u\|_{L^2}^2\|u\|_{1}^2,
\end{eqnarray*}
and thus
\begin{eqnarray*}
\|A(u)-B(u)+K(t,u,\mu)\|_{-1}^2\leq C\|u\|_{L^2}^2\|u\|_{1}^2+C\big(1+\|u\|_{L^2}^2+\mu(\|\cdot\|_{L^2}^2)\big),
\end{eqnarray*}
which implies that the growth condition $\mathbf{H5}$ holds for $q=\theta_2=2$.
Therefore, by Theorem \ref{th2}, there exists a unique (probabilistically) strong solution to (\ref{eqns222}).
 The proof is completed.
\end{proof}
\begin{remark}
The existence of weak solutions for stochastic 3D NSE has been studied a lot in the literature, see e.g. \cite{BM1,GRZ,M14}.  By applying our framework, one can also get the existence of weak solutions for many other McKean-Vlasov stochastic hydrodynamical type evolution equations, such as stochastic magneto-hydrodynamic equation, stochastic Boussinesq equation and stochastic magnetic B\'{e}nard problem when $d=3$, and get the unique strong solution when $d=2$ (cf. \cite{M14}).

In Section \ref{Poc.4}, we will show that the empirical laws of the weakly interacting stochastic Navier-Stokes systems converge in the Wasserstein distance to the law of the unique solution of MVSNSE when $d=2$.
\end{remark}

\subsubsection*{Example 3: Stochastic Cahn-Hilliard equations}\label{sche}

Our next example is the McKean-Vlasov stochastic Cahn-Hilliard system.  Let $\mathscr{O}\subset\mathbb{R}^d$, $d\leq 3$,  be a bounded open domain with smooth boundary. The Cahn-Hilliard equation is a classical model to describe phase separation in a binary alloy, and it is a fundamental phase field model in material science. The deterministic Cahn-Hilliard equation has been studied in the following form
 \begin{eqnarray}\label{E41}
\left\{
 \begin{aligned}
    &\partial_t u=-\Delta^2 u +
\Delta\varphi(u),\;u(0)=u_0,\\[2pt]
    &\nabla u\cdot n = \nabla(\Delta u) \cdot n =
0,\;\text{on}~\partial\mathscr{O},
\\[2pt]&\int_{\mathscr{O}}u(x)dx=0,
  \end{aligned}
\right.
\end{eqnarray}
 where $n$ is the outer unit normal vector on the boundary $\partial\mathscr{O}$.
The reader is referred to \cite{DD96,NC} and the references therein for the study of the Cahn-Hilliard equation
in the deterministic and stochastic case.

 Suppose that $\varphi\in
C^1(\mathbb{R}),$ and that there exists
some positive constants $C>0$, $p\in[2,\frac{d+4}{d}]$ such that
\begin{equation}\label{che}
	\begin{split}
 & \varphi'(x)\geq-C,\ |\varphi(x)|\leq C(1+|x|^{p}),\ x\in\mathbb{R};\\
 & |\varphi(x)-\varphi(y)|\leq C(1+|x|^{p-1}+|y|^{p-1})|x-y|,\ x,y\in\mathbb{R}.
 	\end{split}
\end{equation}
We want to study the following McKean-Vlasov stochastic Cahn-Hilliard
equations (MVSCHEs)
 \begin{eqnarray}\label{CHE}
\left\{
 \begin{aligned}
&dX_t=\big[-\Delta^2X_t+\Delta\varphi(X_t)+ K(t,X_t,\mathscr{L}_{X_t})\big]dt+\sigma(t,X_t,\mathscr{L}_{X_t})dW_t,\\
    &\nabla X_t\cdot n = \nabla(\Delta X_t) \cdot n =
0, ~~\text{on}~\partial\mathscr{O},
\\[2pt]&\int_{\mathscr{O}}X_t(x)dx=0,
  \end{aligned}
\right.
\end{eqnarray}
where $W_t$ is an $U$-valued cylindrical Wiener process defined on $\left(\Omega,\mathscr{F},\{\mathscr{F}_t\}_{t\in[0,T]},\mathbb{P}\right)$.

 Let
$$\mathscr{V}_0:=\Big\{u\in C^\infty(\mathscr{O}):\nabla u\cdot n=\nabla(\Delta u)\cdot n=0 ~\text{on}~\partial\mathscr{O},~\int_{\mathscr{O}}u(x)dx=0\Big\}.$$
For $m\geq0$, we define
$$W_{\sigma}^{m,2}(\mathscr{O}):=\text{the completion of}~\mathscr{V}_0~\text{in}~ W^{m,2}(\mathscr{O}; \mathbb{R}^{d}).$$

We choose $$\mathbb{H}=:W_{\sigma}^{0,2}(\mathscr{O}),~~\mathbb{Y}:=L^p(\mathscr{O}),~~{\mathbb{V}}:=W_{\sigma}^{2,2}(\mathscr{O}),
~~\mathbb{X}:=W_{\sigma}^{4,2}(\mathscr{O}).$$
In particular, we use the following (equivalent) Sobolev norm on $W_{\sigma}^{2,2}(\mathscr{O})$,
 \begin{eqnarray}\label{norm2}
 \|u\|_{2}:=\|\Delta u\|_{L^2}=\left(\int_\mathscr{O}|\Delta u|^2 dx\right)^\frac{1}{2},
 \end{eqnarray}
and define the function $\mathscr{N}_1$ on $\mathbb{Y}$ as follows,

$$
\mathscr{N}_1(u):=\begin{cases} \|u\|_{2}^2,~~~~\text{if}~u\in W_{\sigma}^{2,2}(\mathscr{O})&\quad\\
+\infty,~~~~~\text{otherwise}.&\quad\end{cases}$$
Then it is easy to deduce that $\mathscr{N}_1\in\mathfrak{M}^2$.

\vspace{1mm}
Then we get the following result concerning the existence and uniqueness of solutions to MVSCHE \eref{CHE}.

\begin{theorem}\label{thch1}
Suppose that \eref{che} and $\mathbf{C4}$-$\mathbf{C5}$ hold and  $\mathbf{H4}$ holds for $K,\sigma$ with $\beta=2p-2,q=\theta_2=2$.
Then for any $X_0\in L^r(\Omega,\mathscr{F}_0,\mathbb{P};{\mathbb{H}})$ with $r>4p$,
MVSCHE (\ref{CHE})  has a unique (probabilistically) strong solution in the sense of Definition \ref{de2}. Furthermore, the following estimates hold
\begin{equation*}
\mathbb{E}\Big[\sup_{t\in[0,T]}\|X_t\|_{L^2}^p\Big]+\mathbb{E}\int_0^T\|X_t\|_{2}^2dt<\infty.
\end{equation*}
\end{theorem}

\begin{proof}
Set
$$A(u):=A_1(u)+K(t,u,\mu),$$
where
$$A_1(u):= -\Delta^2 u+\Delta\varphi(u).$$
By Sobolev's inequality (note that the embedding $W_{\sigma}^{2,2}(\mathscr{O})\subset L^\infty(\mathscr{O})$ holds by our
assumption on the dimension $d$), for $u\in\mathscr{V}_0$ we have
\begin{eqnarray*}
|_{\mathbb{V}^*}\langle A_1(u), v\rangle_\mathbb{V}|=\!\!\!\!\!\!\!\!&&|\langle -\Delta u+\varphi(u),\Delta v\rangle_{L^2}|
\nonumber \\
\leq\!\!\!\!\!\!\!\!&& \|v\|_{2}(\|u\|_{2}+\|\varphi(u)\|_{L^2})
\nonumber \\
\leq\!\!\!\!\!\!\!\!&&
  C\|v\|_{2}(1+\|u\|_{2}+\|u\|_{L^\infty}^p)
\nonumber \\
\leq\!\!\!\!\!\!\!\!&& C\|v\|_{2}(1+\|u\|_{2}+\|u\|_{2}^p),~~v\in\mathbb{V}.
\end{eqnarray*}
Therefore, by continuity $A_1$ can be extended to a map from $\mathbb{V}$ to $\mathbb{V}^*$.

By Theorem \ref{th2} and assumptions $\mathbf{C4}$-$\mathbf{C5}$, it suffices to check that $\mathbf{H1}$-$\mathbf{H5}$ hold for $A_1$.
For $v\in\mathbb{X}$ and $u_n,u\in\mathbb{V}$ with $u_n\to u$ in $\mathbb{Y}$, we have
\begin{eqnarray*}
_{\mathbb{X}^*}\langle A_1(u_n)-A_1(u), v\rangle_\mathbb{X}=\!\!\!\!\!\!\!\!&& _{\mathbb{X}^*}\langle\Delta^2(u-u_n)+\Delta(\varphi(u_n)-\varphi(u)),v\rangle_\mathbb{X}
\nonumber \\
\leq\!\!\!\!\!\!\!\!
 && \langle u-u_n,\Delta^2v\rangle_\mathbb{H}
\nonumber\\
\!\!\!\!\!\!\!\!&&+C\|\Delta v\|_{L^\infty}\int_{\mathscr{O}}|\varphi(u(x))-\varphi(u_n(x))|dx\nonumber \\
=:\!\!\!\!\!\!\!\!
 && I_1^n+I_2^n.\nonumber
\end{eqnarray*}
For $I_1^n$, we can deduce that
$$I_1^n\leq C\|v\|_{4}\|u-u_n\|_{L^2}\leq C\|v\|_{4}\|u-u_n\|_{L^p}.$$
For $I_2^n$, by \eref{che}, Poincar\'{e}'s inequality and H\"{o}lder's inequality, we obtain
\begin{eqnarray*}
I_2^n\leq\!\!\!\!\!\!\!\!&& C\|v\|_{4}\int_{\mathscr{O}}(1+|u(x)|^{p-1}+|u_n(x)|^{p-1})|u(x)-u_n(x)|dx\nonumber \\
\leq\!\!\!\!\!\!\!\! &&C\|v\|_{4}(1+\|u\|_{L^p}^{p-1}+\|u_n\|_{L^p}^{p-1})\|u_n-u\|_{L^p},
\end{eqnarray*}
where we have used the Sobolev embedding $W_{\sigma}^{4,2}(\mathscr{O}) \subseteq W_{\sigma}^{2,\infty}(\mathscr{O}) $ for $d\leq3$.
Thus, $I_1^n+I_2^n\rightarrow0$ as $n\rightarrow\infty$, which implies that $\mathbf{H1}$ holds.

Now let us prove $\mathbf{H2}$. Since $\varphi'$ is lower bounded by the interpolation inequality \eref{GN_inequality} we have for any $u\in\mathbb{X}$,
\begin{eqnarray*}
_{\mathbb{X}^*}\langle A_1(u), u\rangle_\mathbb{X}=\!\!\!\!\!\!\!\!&& _{\mathbb{X}^*}\langle-\Delta^2u+\Delta\varphi(u),u\rangle_\mathbb{X}
\nonumber \\=\!\!\!\!\!\!\!\!&&-\|u\|_{2}^2-\int_{\mathscr{O}}\varphi'(u)|\nabla u|^2dx\nonumber \\\leq\!\!\!\!\!\!\!\!&&-\|u\|_{2}^2+C\|u\|_{1}^2
\nonumber \\\leq\!\!\!\!\!\!\!\!&&-\frac{1}{2}\|u\|_{2}^2+C\|u\|_{L^2}^2
\nonumber \\=\!\!\!\!\!\!\!\!&&-\frac{1}{2}\mathscr{N}_1(u)+C\|u\|_{L^2}^2.
\end{eqnarray*}
By the interpolation inequality \eref{GN_inequality} with $n=0,q=2p,m=2$ and $\theta:=\frac{(p-1)d}{4p}$, for any $u\in\mathbb{V}$, we have
\begin{eqnarray*}
\| A_1(u)\|_{\mathbb{V}^*}=\!\!\!\!\!\!\!\!&& \|-\Delta^2u+\Delta\varphi(u)\|_{\mathbb{V}^*}
\nonumber \\\leq\!\!\!\!\!\!\!\!&&C\|u\|_{2}+\|\varphi(u)\|_{L^2}
\nonumber \\\leq\!\!\!\!\!\!\!\!&&C\|u\|_{2}+C(1+\|u\|_{L^{2p}}^p)
\nonumber \\\leq\!\!\!\!\!\!\!\!&&C\|u\|_{2}+C(1+\|u\|_{2}^{\theta p}\|u\|_{L^2}^{(1-\theta)p})
\nonumber\\=\!\!\!\!\!\!\!\!&&C\|u\|_{2}+C(1+\|u\|_{2}^{\theta p}\|u\|_{L^2}^{1-\theta p}\|u\|_{L^2}^{p-1}).
\end{eqnarray*}
 Since $p\leq\frac{4}{d}+1$ ($\theta p=\frac{(p-1)d}{4}\leq
1$) and $\|v\|_{L^2}\leq C\|v\|_2$, this implies that
\begin{eqnarray*}
\| A_1(u)\|_{\mathbb{V}^*}^2\leq C(1+\|u\|_2^2)(1+\|u\|_{L^2}^{2p-2})\leq C(1+\mathscr{N}_1(u))(1+\|u\|_{L^2}^{2p-2}),~\text{for}~ u\in\mathbb{V}.
\end{eqnarray*}
Therefore by assumptions $\mathbf{C5}$, we can see that the growth conditions
$\mathbf{H3}$ and $\mathbf{H5}$ hold with $\gamma_1=2,\gamma_2=\beta=2p-2,q=\theta_1=\theta_2=2.$

Finally, for any $u,v\in {\mathbb{V}}$, we have
\begin{eqnarray*}
_{{\mathbb{V}}^*}\langle A_1(u)-A_1(v),u-v\rangle_{\mathbb{V}}=
\!\!\!\!\!\!\!\!&&{_{{\mathbb{V}}^*}\langle}-\Delta^2(u-v)+\Delta\varphi(u)-\Delta\varphi(v),u-v\rangle_{\mathbb{V}}
\nonumber \\\leq\!\!\!\!\!\!\!\!&&
-\|u-v\|_{2}+\|u-v\|_{2}\|\varphi(u)-\varphi(v)\|_{L^2}
\nonumber \\\leq\!\!\!\!\!\!\!\!&&
-\frac{1}{2}\|u-v\|_{2}+C(1+\|u\|_{L^\infty}^{2p-2}+\|v\|_{L^\infty}^{2p-2})\|u-v\|_{L^2}^2,
\end{eqnarray*}
which yields that the local monotonicity condition $\mathbf{H2}$ holds with $\eta(u)=\rho(u)=C\|u\|_{L^\infty}^{2p-2}$. Moreover, by the interpolation inequality \eref{GN_inequality} with $n=0, q=\infty, m=2,
\theta=\frac{d}{4}$, and $2\leq p\leq \frac{d+4}{d}$, for any $u\in {\mathbb{V}}$, we have
$$\eta(u)=\rho(u)=C\|u\|_{L^\infty}^{2p-2}\leq C\|u\|_{2}^{2\theta(p-1)}\|u\|_{L^2}^{2(1-\theta(p-1))}\|u\|_{L^2}^{2(p-1)-2}\leq C\|u\|_{2}^2\|u\|_{L^2}^{2p-4},$$
which shows that \eref{es22}  holds.

Therefore, the assertion follows by applying Theorem \ref{th2}.
\end{proof}
\begin{remark}
Note that when (\ref{CHE}) is driven by additive noise and without distribution dependence, the well-posedness is studied in \cite{LR13} using the theory of pseudo-monotone operator. However, such method is quite difficult to apply in the case of multiplicative noise. Our method is essentially different from \cite{LR13}, which is not only applicable in case the noise is multiplicative, but also cover the case of the distribution dependent coefficients.
\end{remark}

\subsubsection*{Example 4: Stochastic Kuramoto-Sivashinsky equations}

The Kuramoto-Sivashinsky equation was  initially introduced in the papers by Kuramoto \cite{YK78} and Sivashinsky \cite{SG80} as a model to describe flame propagation. The
equation in the one dimension case has the following form
\begin{equation}\label{KS1}
		\partial_{t} u = -\partial_{x}^{4}u - \partial_{x}^{2} u - u \partial_{x} u.
	\end{equation}
We can see that the first two terms on the right-hand side of \eref{KS1} are of Cahn-Hilliard type (with $\varphi(x)=-x$ in (\ref{E41})), and the last term is of Burgers type. The stochastic Kuramoto-Sivashinsky equation has been studied a lot in the literature, see e.g. \cite{GLS,Yang12}.

Let $\mathscr{O} = (-L,L)$, $L > 0$. We consider the following McKean-Vlasov stochastic Kuramoto-Sivashinsky equations (MVSKSEs) with periodic boundary conditions,
 \begin{eqnarray}\label{KSE}
\left\{
 \begin{aligned}
&d X_t = \big[ -\partial_{x}^{4}X_t - \partial_{x}^{2} \varphi(X_t) - X_t \partial_{x} X_t+K(t,X_t,\mathscr{L}_{X_t})\big] dt + \sigma(t,X_t,\mathscr{L}_{X_t})dW_{t},\\
    & X_t(x-L) = X_t(x+L),~\int_{\mathscr{O}}X_t(x)dx=0,
  \end{aligned}
\right.
\end{eqnarray}
where $W_t$ is an $U$-valued cylindrical Wiener process defined on $\left(\Omega,\mathscr{F},\{\mathscr{F}_t\}_{t\in[0,T]},\mathbb{P}\right)$, and $\varphi$ satisfies the condition \eref{che} with $2\leq p\leq5$.

Let
$$\mathscr{V}_0:=\Big\{u\in C^\infty(\mathscr{O}): u(x-L,t) =  u(x+L,t), \int_{\mathscr{O}}u(x)dx=0 \Big\}.$$
For $m\geq0$, we define
$$W_{\sigma}^{m,2}(\mathscr{O}):=\text{the completion of}~\mathscr{V}_0~\text{in}~ W^{m,2}(\mathscr{O}; \mathbb{R}).$$

To formulate the main results, we choose $$\mathbb{H}:=W_{\sigma}^{0,2}(\mathscr{O}),~~\mathbb{Y}:=L^p(\mathscr{O}),~~{\mathbb{V}}:=
W_{\sigma}^{2,2}(\mathscr{O}),
~~\mathbb{X}:=W_{\sigma}^{4,2}(\mathscr{O}).$$
Define the function $\mathscr{N}_1$ on $\mathbb{Y}$ as follows,

$$
\mathscr{N}_1(u):=\begin{cases} \|u\|_{2}^2,~~~~\text{if}~u\in W_{\sigma}^{2,2}(\mathscr{O})&\quad\\
+\infty,~~~~~\text{otherwise},&\quad\end{cases}$$
where $\|\cdot\|_{2}^2$ denotes the  (equivalent) Sobolev norm on $W_{\sigma}^{2,2}(\mathscr{O})$ (see \eref{norm2}),
then it is easy to deduce that $\mathscr{N}_1\in\mathfrak{M}^2$.

\vspace{1mm}
The following gives the existence and uniqueness of strong solutions to MVSKSEs \eref{KSE}.

\begin{theorem}\label{thks1}
Suppose that $\mathbf{C4}$-$\mathbf{C5}$ hold and  $\mathbf{H4}$ holds for $K,\sigma$ with $\beta=2p-2,q=\theta_2=2$.
Then for any $X_0\in L^r(\Omega,\mathscr{F}_0,\mathbb{P};{\mathbb{H}})$ with $r>4p$,
MVSKSE (\ref{KSE})  has a unique (probabilistically) strong solution in the sense of Definition \ref{de2}. Furthermore, the following estimates hold
\begin{equation*}
\mathbb{E}\Big[\sup_{t\in[0,T]}\|X_t\|_{L^2}^p\Big]+\mathbb{E}\int_0^T\|X_t\|_{2}^2dt<\infty.
\end{equation*}
\end{theorem}

\begin{proof}
We write
$$A(u):=A_1(u)+A_2(u)+K(t,u,\mu),$$
where
$$A_1(u):=- \partial_{x}^{4} u + \partial_{x}^{2} \varphi (u), ~~A_2(u):=- u \partial_{x}u.$$
Note that $A_1$ has been extended to an operator from $\mathbb{V}$ to $\mathbb{V}^*$ as in the previous example, where the conditions $\mathbf{H1}$-$\mathbf{H5}$ hold for $A_1$. Therefore, we only need to show that $A_2$ is well defined and satisfies the conditions in Theorem \ref{th2}.
First,  for $u \in \mathbb{X}$ by the interpolation inequality
 \begin{align}
	 	\left|~_{\mathbb V^{*}}\langle u \partial_{x} u, v \rangle_{\mathbb V} \right| = \frac{1}{2} \left| \langle \partial_{x} (u^{2}), v \rangle_{{\mathbb H} } \right| \leq \frac{1}{2} \| u \|_{L^{4}}^{2} \| \partial_{x} v \|_{L^{2}} \leq C \| u \|_{2} \| u \|_{L^2} \| v \|_{2}.\label{ks3}
	 \end{align}
Thus $A_2$ can be extended to an operator from $\mathbb{V}$ to $\mathbb{V}^*$.
For $v\in\mathbb{X}$ and $u_n,u\in\mathbb{V}$ with $u_n\to u$ in $\mathbb{Y}$, we have
\begin{eqnarray*}
_{\mathbb{X}^*}\langle A_2(u_n)-A_2(u), v\rangle_\mathbb{X}=\!\!\!\!\!\!\!\!&&-_{\mathbb{X}^*}\langle u_n \partial_{x}u_n-u \partial_{x}u,v\rangle_\mathbb{X}
\nonumber \\
=\!\!\!\!\!\!\!\!
 && _{\mathbb{X}^*}\langle u_n \partial_{x}v,u_n-u\rangle_\mathbb{X}+{_{\mathbb{X}^*}}\langle (u_n-u) \partial_{x}v,u\rangle_\mathbb{X}
\nonumber \\\leq\!\!\!\!\!\!\!\!
 && \|\partial_{x}v\|_{L^\infty}\|u_n\|_{L^{2}}\|u_n-u\|_{L^{2}}+ \|\partial_{x}v\|_{L^\infty}\|u\|_{L^{2}}\|u_n-u\|_{L^{2}}
  \nonumber \\\leq\!\!\!\!\!\!\!\!
 && C\|v\|_{4}(\|u_n\|_{L^2}+\|u\|_{L^2})\|u_n-u\|_{L^p},
  \nonumber
\end{eqnarray*}
where we used the embedding $W_{\sigma}^{4,2}(\mathscr{O}) \subseteq W_{\sigma}^{2,2}(\mathscr{O}) \subseteq W_{\sigma}^{1,\infty}(\mathscr{O}) $. Thus $\mathbf{H1}$ holds. For $\mathbf{H2}$  we
note that $_{\mathbb{X}^*}\langle A_2(u),u\rangle_\mathbb{X}=0$. Moreover, by \eref{ks3} we know that the growth conditions
$\mathbf{H3}$ and $\mathbf{H5}$ hold with $\gamma_1=\gamma_2=q=2$ and $\beta=2p-2.$
Finally, for any $u,v\in {\mathbb{V}}$, we have
\begin{eqnarray*}
	_{\mathbb{V}^{*}}\langle A_2(u) - A_2(v), u - v \rangle_{\mathbb{V}}
=\!\!\!\!\!\!\!\!&&_{\mathbb{V}^{*}}\langle u \partial_{x} u - v \partial_{x} v, u - v \rangle_{\mathbb{V}}
\nonumber \\
=\!\!\!\!\!\!\!\!
 &&\frac{1}{2} \int_{\mathscr{O}} (u-v)^{2} \partial_{x} v d x\leq C \| \partial_{x} v \|_{L^{\infty}} \| u - v \|_{L^2}^{2}.
\end{eqnarray*}
Thus the local monotonicity condition $\mathbf{H4}$ holds with another term $\rho_{A_2}(v) = \|\partial_{x} v \|_{L^{\infty}}$, which is bounded on balls in $W_{\sigma}^{2,2}(\mathscr{O})$. By the embedding $W_{\sigma}^{2,2}(\mathscr{O}) \subseteq W_{\sigma}^{1,\infty}(\mathscr{O})$, it is easy to see that \eref{es22} holds.

Therefore, the assertion follows by applying Theorem \ref{th2}.
\end{proof}
\begin{remark}
Yang \cite{Yang12} studied the  stochastic Kuramoto-Sivashinsky equation with $\varphi(x)=-x$, which was extended by Gess et al. \cite{GLS} to more general $\varphi$ with growth order $p\leq2$. Here, by applying our framework, we can consider the stochastic Kuramoto-Sivashinsky equation with growth order $p\in [2,5]$, which is new even in the distribution independent case. In particular, the function $\varphi$ can be taken as the typical example $\varphi(x)=x^3-x$, which is
the derivative of the double well potential $F(x)=\frac{1}{4}(x^2-1)^2$.
\end{remark}

\section{Proofs of Theorem \ref{th1} and \ref{th2}}\label{proof1}
\setcounter{equation}{0}
 \setcounter{definition}{0}

In this part, we shall prove Theorems \ref{th1} and \ref{th2}. In Subsection \ref{sec2.2}, we will prove that the result about the existence of weak solutions  holds in finite dimensions by constructing a suitable cut-off function. In Subsections \ref{sec2.3} and \ref{sec2.4}, we will give the proof of the existence of weak solutions to MVSPDE (\ref{eqSPDE}). In Subsection \ref{proof2}, we will prove the existence and uniqueness of (probabilistically) strong solutions to MVSPDE (\ref{eqSPDE}).

\subsection{Weak existence in finite dimensions}\label{sec2.2}
In this subsection, we intend to prove Theorem \ref{th1} in the case $U={\mathbb{H}}=\mathbb{X}=\mathbb{Y}=\mathbb{V}=\mathbb{X}^*=\RR^d$. For simplicity, we use $|\cdot|$ (resp. $\langle\cdot,\cdot\rangle$) to denote the norm (resp. scalar product) on $\RR^d$ and $\|\cdot\|$ to denote the Hilbert-Schmidt norm from $\RR^d$ to $\RR^d$.

To construct a weak solution of (\ref{eqSPDE}) in $\RR^d$, we first introduce a localizing  approximation through a suitable cut-off function. More precisely, for any $n\in\mathbb{N}$ we define
\begin{equation}\label{cutoff}
\psi_n(u)=\frac{nu}{n\vee|u|},~~u\in\RR^d,
\end{equation}
and
$$A^n(t,u,\mu)=A(t, \psi_n(u),\mu\circ\psi_n^{-1}),~\sigma^n(t,u,\mu)=\sigma(t, \psi_n(u),\mu\circ\psi_n^{-1}),~t\in[0,T].$$
Note that the weak convergence is equivalent to the strong convergence in $\RR^d$. Then by $\mathbf{H1}$ and $\mathbf{H3}$, it is clear that
\begin{eqnarray}
\!\!\!\!\!\!\!\!&&(i)~~ A^n(t,\cdot,\cdot), \sigma^n(t,\cdot,\cdot)~ \text{are continuous for any $t\in[0,T]$ }.\label{well-pose1}
 \\
\!\!\!\!\!\!\!\!&&(ii)~~ A^n, \sigma^n~ \text{are bounded (the bounds depend on $n$)}.\label{well-pose2}
\end{eqnarray}
On the other hand, by (\ref{es01}) the coefficients $A^n,\sigma^n$ satisfy  $\mathbf{H3}$ with $|u|^q$ replacing $\mathscr{N}_1(u)$ and with $C>0$ which is independent of $n$. In addition, the following lemma holds.

\begin{lemma}\label{cut-off}
There exists a constant $C>0$ (independent of $n$) such that for any $t\in[0,T]$, $u\in \RR^d$ and $\mu\in\mathscr{P}_{2}(\RR^d)$,
\begin{equation}\label{es59}
\langle A^n(t,u,\mu),u\rangle+\|\sigma^n(t,u,\mu)\|^2\leq C\big(1+|u|^2+\mu(|\cdot|^2)\big).
\end{equation}
\end{lemma}
\begin{proof}
Firstly, by $\mathbf{H2}$, there exists a constant $C>0$ independent of $n$ such that
\begin{eqnarray*}
\langle A^n(t,u,\mu),u\rangle=\!\!\!\!\!\!\!\!&&\frac{n\vee|u|}{n}\langle A(t, \psi_n(u),\mu\circ\psi_n^{-1}),\psi_n(u)\rangle
\nonumber \\
 \leq\!\!\!\!\!\!\!\!&&C\frac{n\vee|u|}{n}\big(1+| \psi_n(u)|^2+\mu\circ\psi_n^{-1}(|\cdot|^2)\big).
\end{eqnarray*}
Note that
\begin{eqnarray}\label{cut1}
\mu\circ\psi_n^{-1}(|\cdot|^2)=\!\!\!\!\!\!\!\!&&\int_{\mathbb{R}^d}|\psi_n(y)|^2\mu(dy)
\nonumber \\
 =\!\!\!\!\!\!\!\!&&\int_{\mathbb{R}^d}\frac{n^2|y|^2}{(n\vee|y|)^2}\mu(dy)
\nonumber \\
\leq\!\!\!\!\!\!\!\!&& \int_{\{n\leq |y|\}}n^2\mu(dy)+\int_{\{|y|\leq n\}}|y|^2\mu(dy)
\nonumber \\
\leq\!\!\!\!\!\!\!\!&&\mu(|\cdot|^2).
\end{eqnarray}
On the other hand, it is easy to see that
\begin{eqnarray}\label{cut2}
\mu\circ\psi_n^{-1}(|\cdot|^2)\leq\!\!\!\!\!\!\!\!&&n^2.
\end{eqnarray}

If $n\geq|u|$, by (\ref{cut1}), it is obvious that each $A^n$ satisfies (\ref{es59}) with $C>0$ which is independent of $n$. If $n\leq |u|$, by (\ref{cut2}) we deduce that
\begin{eqnarray*}
\langle A^n(t,u,\mu),u\rangle\leq C\frac{|u|}{n}\big(1+2n^2 \big)
 \leq C\big(1+|u|^2\big),
\end{eqnarray*}
where the constant $C>0$ is independent of $n$. Therefore, we obtain that for any $t\in[0,T]$, $u\in \RR^d$ and $\mu\in\mathscr{P}_{2}(\RR^d)$, there exists a constant $C>0$ which is independent of $n$ such that
\begin{eqnarray*}
\langle A^n(t,u,\mu),u\rangle\leq C\big(1+|u|^2+\mu(|\cdot|^2)\big).
\end{eqnarray*}
By similar arguments, we obtain
$$\|\sigma^n(t,u,\mu)\|^2\leq C\big(1+|u|^2+\mu(|\cdot|^2)\big).$$
This completes the proof of the lemma.
\end{proof}

\vspace{1mm}
Let $X^0_t\equiv X_0$. For any $n\in\mathbb{N}$, we introduce the following approximating equation
\begin{equation}\label{eq2}
X^n_t=X^n_0+\int_0^tA^n(s,X^n_s,\mu^n_s)ds+\int_0^t\sigma^n(s,X^n_s,\mu^n_s)dW_s,
\end{equation}
where $\mu^n_s:=\mathscr{L}_{X^n_s}$.
Following from \cite{HSS}, by claims (\ref{well-pose1}) and (\ref{well-pose2}) we know that there exists a weak solution
$$(\Omega^n,\mathscr{F}^n,\{\mathscr{F}^n_t\}_{t\in[0,T]},\PP^n; X^n,W^n)~~\text{to (\ref{eq2})}$$
with $\PP^n \circ (X^n_0)^{-1}=\mu_0$.

In what follows, we denote by $ \mathbb{E}^n$ the expectation with respect to the measure $\PP^n$.
The following lemma contains some  uniform  estimates and time H\"{o}lder continuity estimates for the approximating solution $X^{n}_t$ to (\ref{eq2}), which are crucial in the proof of tightness.
\begin{lemma}\label{pro1}
Suppose that the assumptions in Theorem \ref{th1}  hold. Then for any  $p>\eta_0$,  there is $C_{p,T}>0$ independent of $n$ such that
\begin{equation}\label{es6}
\mathbb{E}^n\Big[\sup_{t\in[0,T]}|X^{n}_t|^p\Big]+\mathbb{E}^n\int_0^T\mathscr{N}_1(X^{n}_t)dt\leq C_{p,T}.
\end{equation}
Moreover, there exist $r>2$, $C_{p,T}>0$ independent of $n$  such that for any $t,s\in[0,T]$ we have
\begin{equation}\label{es1}
\mathbb{E}^n|X^{n}_t-X^{n}_s|^{r}\leq C_{p,T}|t-s|^{\frac{r}{2}}.
\end{equation}
\end{lemma}
\begin{proof}
By (\ref{es59}) and It\^{o}'s formula, we deduce that for any $p\in(\eta_0,\infty)$,
\begin{eqnarray*}
|X^{n}_t|^p=\!\!\!\!\!\!\!\!&&|X_0^n|^p+\frac{p(p-2)}{2}\int_0^t|X^{n}_s|^{p-4}|\sigma^n(s,X^{n}_s,\mu^{n}_s)^*X^{n}_s|^2ds
\nonumber \\
 \!\!\!\!\!\!\!\!&&+\frac{p}{2}\int_0^t|X^{n}_s|^{p-2}\big(2\langle A^n(s,X^{n}_s,\mu^{n}_s),X^{n}_s\rangle
+\|\sigma^n(s,X^{n}_s,\mu^{n}_s)\|^2\big)ds
\nonumber \\
 \!\!\!\!\!\!\!\!&&+p\int_0^t|X^{n}_s|^{p-2}\langle X^{n}_s,\sigma^n(s,X^{n}_s,\mu^{n}_s)dW^n_s\rangle
\nonumber \\
 \leq\!\!\!\!\!\!\!\!&&|X_0^n|^p+C_p\int_0^t|X^{n}_s|^{p-2}(1+|X^{n}_s|^2+\mathbb{E}^n|X^{n}_s|^2)ds
\nonumber \\
 \!\!\!\!\!\!\!\!&&
+p\int_0^t|X^{n}_s|^{p-2}\langle X^{n}_s,\sigma^n(s,X^{n}_s,\mu^{n}_s)dW^n_s\rangle
\nonumber \\
 \leq\!\!\!\!\!\!\!\!&&|X_0^n|^p+C_p\int_0^t(1+|X^{n}_s|^p+\mathbb{E}^n|X^{n}_s|^p)ds
 \nonumber \\
 \!\!\!\!\!\!\!\!&&
 +p\int_0^t|X^{n}_s|^{p-2}\langle X^{n}_s,\sigma^n(s,X^{n}_s,\mu^{n}_s)dW^n_s\rangle.
\end{eqnarray*}
By BDG's inequality, we obtain
\begin{eqnarray*}
&&p\mathbb{E}^n\Bigg\{\sup_{t\in[0,T]}\Big|\int_0^t|X^{n}_s|^{p-2}\langle X^{n}_s,\sigma^n(s,X^{n}_s,\mu^{n}_s)dW^n_s\rangle\Big|\Bigg\}
 \nonumber \\
 \leq\!\!\!\!\!\!\!\!&& C_p\mathbb{E}^n\Bigg\{\int_0^{T}|X^{n}_t|^{2p-2}\|\sigma^n(t,X^{n}_t,\mu^{n}_t)\|^2dt\Bigg\}^{\frac{1}{2}}
 \nonumber \\
 \leq\!\!\!\!\!\!\!\!&&\ C_p\mathbb{E}^n\Bigg\{\int_0^{T}|X^{n}_t|^{2p-2}(1+|X^{n}_t|^2+\mathbb{E}^n|X^{n}_t|^2)dt\Bigg\}^{\frac{1}{2}}
  \nonumber \\
 \leq\!\!\!\!\!\!\!\!&&\ C_p\mathbb{E}^n\Bigg\{\sup_{t\in[0,T]}|X^{n}_t|^{p}\cdot \int_0^{T}|X^{n}_t|^{p-2}(1+|X^{n}_t|^2+\mathbb{E}^n|X^{n}_t|^2)dt\Bigg\}^{\frac{1}{2}}
 \nonumber \\
 \leq\!\!\!\!\!\!\!\!&&\ \mathbb{E}^n\Bigg\{\frac{1}{2}\sup_{t\in[0,T]}|X^{n}_t|^{p}+C_p \int_0^{T}\big(1+|X^{n}_t|^p+(\mathbb{E}^n|X^{n}_t|^2)^{\frac{p}{2}}\big)dt\Bigg\}
  \nonumber \\
 \leq\!\!\!\!\!\!\!\!&&\frac{1}{2}\mathbb{E}^n\Big[\sup_{t\in[0,T]}|X^{n}_t|^p\Big]
+ C_p\mathbb{E}^n \int_0^{T}(1+|X^{n}_t|^p+\mathbb{E}^n|X^{n}_t|^p)dt,
\end{eqnarray*}
where we have used Young's inequality in the fourth inequality.

This leads to
\begin{eqnarray*}
\!\!\!\!\!\!\!\!&&\mathbb{E}^n\Big[\sup_{t\in[0,T]}|X^{n}_t|^p\Big]
 \nonumber \\
\leq\!\!\!\!\!\!\!\!&&\mathbb{E}^n|X_0^n|^p
+ C_p\mathbb{E}^n \int_0^{T}\big(1+|X^{n}_t|^p+\mathbb{E}^n|X^{n}_t|^p\big)dt+\frac{1}{2}\mathbb{E}^n\Big[\sup_{t\in[0,T]}|X^{n}_t|^p\Big]
\nonumber \\  \leq\!\!\!\!\!\!\!\!&&C_{p,T}\big(1+\mathbb{E}^n|X_0^n|^p\big)
+ C_p\int_0^T\mathbb{E}^n\Big[\sup_{s\in[0,t]}|X^{n}_s|^p\Big]dt+\frac{1}{2}\mathbb{E}^n\Big[\sup_{t\in[0,T]}|X^{n}_t|^p\Big].
\end{eqnarray*}
By  Gronwall's inequality we have
$$ \mathbb{E}^n\Big[\sup_{t\in[0,T]}|X^{n}_t|^p\Big]\leq C_{p,T}\big(1+\mathbb{E}^n|X_0^n|^p\big)=C_{p,T}\big(1+\mu_0(|\cdot|^p)\big)\leq C_{p,T}.$$
In view of the assumption on $\mathscr{N}_1$ (i.e.~(\ref{es01})), we also obtain
$$\mathbb{E}^n\int_0^T\mathscr{N}_1(X^{n}_t)dt\leq C_{p,T}.$$

Concerning (\ref{es1}),  for some $r>2$ and any  $t,s\in[0,T]$, by $\mathbf{H3}$ and (\ref{es01}) we have
\begin{eqnarray*}
\!\!\!\!\!\!\!\!&&\mathbb{E}^n|X^{n}_t-X^{n}_s|^{r}\nonumber \\
\leq\!\!\!\!\!\!\!\!&& C|t-s|^{\frac{r(\gamma_1-1)}{\gamma_1}}\mathbb{E}^n\Big(\int_{s}^{t}|A^n(u,X^{n}_u,\mu^{n}_u)|^{\gamma_1} du\Big)^{\frac{r}{\gamma_1}}
+C\mathbb{E}^n\Big|\int_{s}^{t}\sigma^n(u,X^{n}_u,\mu^{n}_u)dW^n_u\Big|^r
\nonumber \\
\leq  \!\!\!\!\!\!\!\!&&C|t-s|^{\frac{r(\gamma_1-1)}{\gamma_1}}\mathbb{E}^n\Big(\int_{s}^{t}\big(1+|X^{n}_u|^q+|X^{n}_u|^{q+\gamma_2}+| X^{n}_u|^{\gamma_2}+(\mathbb{E}^n|X^{n}_u|^{\theta_1})^2 \big)du\Big)^{\frac{r}{\gamma_1}}
\nonumber \\
  \!\!\!\!\!\!\!\!&&
+C|t-s|^{\frac{r-2}{2}}\int_{s}^{t}\Big(1+\mathbb{E}^n
|X^{n}_u|^{r}\Big)du
\nonumber \\
\leq\!\!\!\!\!\!\!\!&&C_{p,T}|t-s|^r+C_{p,T}|t-s|^{\frac{r}{2}}
+C|t-s|^{r-1}\Big(\int_{s}^{t}\mathbb{E}^n|X^{n}_u|^{\frac{(q+\gamma_2)r}{\gamma_1}} du\Big)
\nonumber \\
\leq\!\!\!\!\!\!\!\!&&C_{p,T}|t-s|^{\frac{r}{2}},
\end{eqnarray*}
where we used (\ref{es6}) in the last step. The proof is completed.
\end{proof}

\vspace{1mm}
The existence of weak solutions to (\ref{eqSPDE}) in $\RR^d$ is the content of  the following theorem.

  \begin{theorem}\label{pro3}
Suppose that the assumptions in Theorem \ref{th1}  hold. Then for any  $X_0\sim\mu_0\in \mathscr{P}_p(\mathbb{R}^d)$ with $p>\eta_0$, there exists a weak solution $(\Omega,\mathscr{F},\{\mathscr{F}_t\}_{t\in[0,T]},\PP; X,W)$ to (\ref{eqSPDE}). Moreover, the following estimates hold
\begin{equation}\label{es19}
\mathbb{E}\Big[\sup_{t\in[0,T]}|X_t|^{p}\Big]+\mathbb{E}\int_0^T\mathscr{N}_1(X_t)dt<\infty.
\end{equation}
\end{theorem}

  \begin{proof}
The proof will be divided into the following two steps.

\vspace{1mm}
\textbf{Step 1:} By Lemma \ref{pro1}, it is easy to see that the sequence $\{\mu^n=\mathscr{L}_{X^{n}}\}_{n\in\mathbb{N}}$ is tight in $\mathscr{P}(\mathbb{C})$, where $\mathbb{C}:=C([0,T];\mathbb{R}^d)$, then so is $\{\Lambda^n:=\mathscr{L}_{(X^{n},W^n)}\}_{n\in\mathbb{N}}$. By the Prokhorov theorem, there exists a subsequence of $\{n\}$, we keep denoting it by $\{n\}$, such that $\mu^n$ weakly converges to $\mu$ in $\mathscr{P}(\mathbb{C})$ and $\Lambda^n\to \Lambda$ weakly in $\mathscr{P}(\mathbb{C}\times\mathbb{C})$.

The Skorokhod representation theorem (cf.~e.g.~\cite{MV}) yields that there exists a probability space $(\tilde{\Omega},\tilde{\mathscr{F}},\tilde{\mathbb{P}})$, and, on this space, $\mathbb{C}\times \mathbb{C}$-valued random variables $(\tilde{X}^{n},\tilde{W}^{n})$ whose law is equal to $\Lambda^n$. Moreover, we  obtain
\begin{equation}\label{es81}
\tilde{X}^{n}\to \tilde{X}\sim \mu~\text{in}~\mathbb{C},~\tilde{\mathbb{P}}\text{-a.s.},~\text{as}~n\to\infty,
\end{equation}
\begin{equation*}
\tilde{W}^{n}\to \tilde{W}~\text{in}~\mathbb{C},~\tilde{\mathbb{P}}\text{-a.s.},~\text{as}~n\to\infty.
\end{equation*}

Let $\tilde{\mathscr{F}}^{n}_t$ be the filtration satisfying the usual conditions and generated by $\{\tilde{X}^{n}_s,\tilde{W}^{n}_s:s\leq t\}$, then we know that $\tilde{X}^{n}_t$ is $\tilde{\mathscr{F}}^{n}_t$-adapted and continuous, and $\tilde{W}^{n}_t$ is a standard $d$-dimensional Wiener process on $(\tilde{\Omega},\tilde{\mathscr{F}},\{\tilde{\mathscr{F}}^{n}_t\}_{t\in[0,T]},\tilde{\mathbb{P}})$.

Notice that, due to (\ref{es6}) and the equality of the laws of $\tilde{X}^{n}$ and $X^{n}$, we can easily infer that for any $T>0$,  there is a constant $C_{p,T}>0$ such that
\begin{equation}\label{es70}
\sup_{n\in\mathbb{N}}\tilde{\mathbb{E}}\Big[\sup_{t\in[0,T]}|\tilde{X}^{n}_t|^p\Big]\leq C_{p,T},
\end{equation}
where $p>\eta_0$. From the convergence (\ref{es81}) and the estimate (\ref{es70}), Vitali's convergence theorem implies that
\begin{equation}\label{es9}
\tilde{\mathbb{E}}\Big[\sup_{t\in[0,T]}|\tilde{X}^{n}_t-\tilde{X}_t|^{p'}\Big]\to 0,~\text{as}~n\to\infty,
\end{equation}
where $\eta_0<p'<p$.
Moreover, by  (\ref{es81}), (\ref{es70}) and using the lower semicontinuity, we obtain  there is a constant $C_{p,T}>0$ such that
 \begin{equation}\label{es10}
\tilde{\mathbb{E}}\Big[\sup_{t\in[0,T]}|\tilde{X}_t|^{p}\Big]+\tilde{\mathbb{E}}\int_0^T\mathscr{N}_1(\tilde{X}_t)dt\leq C_{p,T}.
\end{equation}

For any $l\in\RR^d$, let us consider the process $\tilde{\mathscr{M}}^{n}_l(t)$ in $\RR$ defined by
\begin{eqnarray}
\tilde{\mathscr{M}}^{n}_l(t)=\!\!\!\!\!\!\!\!&&\langle\tilde{X}^{n}_t,l\rangle-\langle\tilde{X}^{n}_0,l\rangle-\int_0^t\langle A^n(s,\tilde{X}^{n}_s,\tilde{\mu}^{n}_s),l\rangle ds
\nonumber \\
=\!\!\!\!\!\!\!\!&&\langle\tilde{X}^{n}_t,l\rangle-\langle\tilde{X}^{n}_0,l\rangle-\int_0^t\langle A(s,\psi_n(\tilde{X}^{n}_s),\tilde{\mu}^{n}_s\circ\psi_n^{-1}),l\rangle ds,
\end{eqnarray}
where $\tilde{\mu}^{n}_t:=\mathscr{L}_{\tilde{X}^{n}_t}$, which is a square integrable martingale w.r.t.~the filtration $\tilde{\mathscr{F}}^{n}_t$ with quadratic variation
$$\langle \tilde{\mathscr{M}}^{n}_l\rangle(t)=\int_0^t |\sigma\big(s,\psi_n(\tilde{X}^{n}_s),\tilde{\mu}^{n}_s\circ\psi_n^{-1}\big)^*l|^2ds.$$
Indeed, for all $0\leq s<t\leq T$ and all bounded continuous functions $\Phi(\cdot)$ on $\mathbb{C}\times\mathbb{C}$, due to the fact that $\tilde{\mu}^{n}=\mu^{n}$, $n\in\mathbb{N}$, we get
\begin{eqnarray}\label{mar1}
\tilde{\mathbb{E}}\Big[\big(\tilde{\mathscr{M}}^{n}_l(t)-\tilde{\mathscr{M}}^{n}_l(s)\big)\Phi((\tilde{X}^{n},\tilde{W}^{n})|_{[0,s]})\Big]=0
\end{eqnarray}
and
\begin{eqnarray}\label{mar2}
\tilde{\mathbb{E}}\Big[\Big(\tilde{\mathscr{M}}^{n}_l(t)^2-\tilde{\mathscr{M}}^{n}_l(s)^2
-\int_s^t |\sigma\big(u,\psi_n(\tilde{X}^{n}_u),\tilde{\mu}^{n}_u\circ\psi_n^{-1}\big)^*l|^2du\Big)\Phi((\tilde{X}^{n},\tilde{W}^{n})|_{[0,s]})\Big]=0.
\end{eqnarray}
Then it suffices to show that for all $s\leq t\in [0,T]$ and all bounded continuous functions $\Phi(\cdot)$ on $\mathbb{C}\times \mathbb{C}$,
\begin{eqnarray}\label{mar3}
\tilde{\mathbb{E}}\Big[\big(\tilde{\mathscr{M}}_l(t)-\tilde{\mathscr{M}}_l(s)\big)\Phi((\tilde{X},\tilde{W})|_{[0,s]})\Big]=0
\end{eqnarray}
and
\begin{eqnarray}\label{mar4}
\!\!\!\!\!\!\!\!&&\tilde{\mathbb{E}}\Big[\Big(\tilde{\mathscr{M}}_l(t)^2-\tilde{\mathscr{M}}_l(s)^2
-\int_s^t |\sigma\big(u,\tilde{X}_u,\tilde{\mu}_u\big)^*l|^2du\Big)\Phi((\tilde{X},\tilde{W})|_{[0,s]})\Big]=0,
\end{eqnarray}
where $\tilde{\mathscr{M}}_l(t)$ is defined by
$$\tilde{\mathscr{M}}_l(t)=\langle\tilde{X}_t,l\rangle-\langle\tilde{X}_0,l\rangle-\int_0^t\langle A(s,\tilde{X}_s,\tilde{\mu}_s),l\rangle ds,~\tilde{\mu}_s=\mathscr{L}_{\tilde{X}_s}.$$
In fact, if (\ref{mar3}) and (\ref{mar4}) hold, it follows that the process $\tilde{\mathscr{M}}_l(t)$ is a square integrable martingale in $\mathbb{R}$ w.r.t.~the filtration $\tilde{\mathscr{F}}_t$ satisfying the usual conditions and generated by $\{\tilde{X}_s,\tilde{W}_s:s\leq t\}$,
whose quadratic variation is
$$\langle \tilde{\mathscr{M}}_l\rangle(t)=\int_0^t |\sigma\big(s,\tilde{X}_s,\tilde{\mu}_s\big)^*l|^2ds.$$
Consequently,  (\ref{eqSPDE}) has a weak solution in $\RR^d$ by the martingale representation theorem (cf.~\cite{Daz1}).

\textbf{Step 2:} Now we shall prove (\ref{mar3}) and (\ref{mar4}).
To this end, we first show that for any $t\in[0,T]$, $\tilde{\mu}^{n}_t\circ\psi_n^{-1}\to\tilde{\mu}_t$ in $\mathbb{W}_2$-sense as $n\to \infty$. %In fact, this will be completed by the following two statements:
%
%\vspace{1mm}
%(i) For any $t\in[0,T]$, $\tilde{\mu}^{n}_t\circ\psi_n^{-1}\to\tilde{\mu}_t$ weakly in $\mathscr{P}_2(\RR^d)$.
%
%\vspace{1mm}
%(ii) For any $t\in[0,T]$, $\{\psi_n(\tilde{X}^{n}_t)\}_{n\in\mathbb{N}}\subset L^2(\Omega;\RR^d)$ is uniformly integrable.
Note that
\begin{eqnarray}\label{es300}
\mathbb{W}_2(\tilde{\mu}^{n}_t\circ\psi_n^{-1},\tilde{\mu}_t)^2\leq\!\!\!\!\!\!\!\!&&\tilde{\mathbb{E}}|\psi_n(\tilde{X}^{n}_t)-\tilde{X}_t|^2
\nonumber \\
=\!\!\!\!\!\!\!\!&&\tilde{\mathbb{E}}\Big[|\psi_n(\tilde{X}^{n}_t)-\tilde{X}_t|^2\mathbf{1}_{\{|\tilde{X}^{n}_t|\geq n\}}\Big]
+\tilde{\mathbb{E}}\Big[|\psi_n(\tilde{X}^{n}_t)-\tilde{X}_t|^2\mathbf{1}_{\{|\tilde{X}^{n}_t|< n\}}\Big]
\nonumber \\
\leq\!\!\!\!\!\!\!\!&&\tilde{\mathbb{E}}\Bigg[\Big|\frac{n\tilde{X}^{n}_t}{|\tilde{X}^{n}_t|}-\tilde{X}_t\Big|^2\mathbf{1}_{\{|\tilde{X}^{n}_t|\geq n\}}\Bigg]+\tilde{\mathbb{E}}|\tilde{X}^{n}_t-\tilde{X}_t|^2
\nonumber \\
\leq\!\!\!\!\!\!\!\!&&\tilde{\mathbb{E}}\Big[\big(|\tilde{X}^{n}_t|^2+|\tilde{X}_t|^2\big)\mathbf{1}_{\{|\tilde{X}^{n}_t|\geq n\}}\Big]+\tilde{\mathbb{E}}|\tilde{X}^{n}_t-\tilde{X}_t|^2
\nonumber \\
\leq\!\!\!\!\!\!\!\!&&\Big(\big(\tilde{\mathbb{E}}|\tilde{X}^{n}_t|^{p}\big)^{\frac{2}{p}}+\big(\tilde{\mathbb{E}}|\tilde{X}_t|^{p}\big)^{\frac{2}{p
}}\Big)\frac{\big(\tilde{\mathbb{E}}|\tilde{X}^{n}_t|\big)^{\frac{p-2}{p}}}{n^{\frac{p-2}{p}}}
+\tilde{\mathbb{E}}|\tilde{X}^{n}_t-\tilde{X}_t|^2.
\end{eqnarray}
Taking (\ref{es70})-(\ref{es10}) into account and letting $n\to\infty$ on both sides of (\ref{es300}), we have for any $t\in[0,T]$,
$$\lim_{n\to\infty}\mathbb{W}_2(\tilde{\mu}^{n}_t\circ\psi_n^{-1},\tilde{\mu}_t)^2=0.$$
By the continuity of $A,\sigma$ and the dominated convergence theorem, it is easy to show that for any $t\in[0,T]$,
\begin{equation}\label{es66}
\int_0^t|A(s,\psi_n(\tilde{X}^{n}_s),\tilde{\mu}^{n}_s\circ\psi_n^{-1})-A(s,\tilde{X}_s,\tilde{\mu}_s)|^2ds\to 0,~\tilde{\mathbb{P}}\text{-a.s.},~\text{as}~n\to\infty,
\end{equation}
and
\begin{equation}\label{es170}
\int_0^t\|\sigma(s,\psi_n(\tilde{X}^{n}_s),\tilde{\mu}^{n}_s\circ\psi_n^{-1})-\sigma(s,\tilde{X}_s,\tilde{\mu}_s)\|^2ds\to 0,~\tilde{\mathbb{P}}\text{-a.s.},~\text{as}~n\to\infty.
\end{equation}
In addition, by $\mathbf{H3}$, (\ref{es70}) and (\ref{es10}), we deduce that
\begin{eqnarray}\label{es18}
\!\!\!\!\!\!\!\!&&\tilde{\mathbb{E}}\Big|\int_0^t\|
\sigma(s,\psi_n(\tilde{X}^{n}_s),\tilde{\mu}^{n}_s\circ\psi_n^{-1})- \sigma(s,\tilde{X}_s,\tilde{\mu}_s)\|^2ds\Big|^{p'}
\nonumber \\
\leq\!\!\!\!\!\!\!\!&& C_T\tilde{\mathbb{E}}\Big[\int_0^t\Big(1+|\psi_n(\tilde{X}^{n}_s)|^{2p'}+\tilde{\mathbb{E}}|\psi_n(\tilde{X}^{n}_s)|^{2p'}
+|\tilde{X}_s|^{2p'}+\tilde{\mathbb{E}}|\tilde{X}_s|^{2p'}\Big)ds\Big]
\nonumber \\
\leq\!\!\!\!\!\!\!\!&& C_{p,T}
\end{eqnarray}
and
\begin{eqnarray}\label{es65}
\!\!\!\!\!\!\!\!&&\tilde{\mathbb{E}}\Big|\int_0^t|
A(s,\psi_n(\tilde{X}^{n}_s),\tilde{\mu}^{n}_s\circ\psi_n^{-1})- A(s,\tilde{X}_s,\tilde{\mu}_s)|^2ds\Big|^{p'}
\nonumber \\
\leq\!\!\!\!\!\!\!\!&& C_T\tilde{\mathbb{E}}\Big[\int_0^t\Big(1+|\psi_n(\tilde{X}^{n}_s)|^{\frac{2qp'}{\gamma_1}}+|\psi_n(\tilde{X}^{n}_s)|^{\frac{2\gamma_2p'}{\gamma_1}}
+(\tilde{\mathbb{E}}|\psi_n(\tilde{X}^{n}_s)|^{\theta_1})^{\frac{2 p'}{\gamma_1}}
\nonumber \\
\!\!\!\!\!\!\!\!&&
+|\tilde{X}_s|^{\frac{2qp'}{\gamma_1}}+|\tilde{X}_s|^{\frac{2\gamma_2p'}{\gamma_1}}  +(\tilde{\mathbb{E}}|\tilde{X}_s|^{\theta_1})^{\frac{2 p'}{\gamma_1}}\Big)ds\Big]
\nonumber \\
\leq\!\!\!\!\!\!\!\!&& C_{p,T},
\end{eqnarray}
for some $p'>1$.

Collecting (\ref{es66})-(\ref{es65}) and applying the Vitali convergence theorem, it leads to
\begin{eqnarray*}
\lim_{n\to\infty}\tilde{\mathbb{E}}\Big[\int_0^t\|
\sigma(s,\psi_n(\tilde{X}^{n}_s),\tilde{\mu}^{n}_s\circ\psi_n^{-1})- \sigma(s,\tilde{X}_s,\tilde{\mu}_s)\|^2ds\Big]
=0
\end{eqnarray*}
and
\begin{eqnarray*}
\lim_{n\to\infty}\tilde{\mathbb{E}}\Big[\int_0^t|
A(s,\psi_n(\tilde{X}^{n}_s),\tilde{\mu}^{n}_s\circ\psi_n^{-1})- A(s,\tilde{X}_s,\tilde{\mu}_s)|^2ds\Big]
=0.
\end{eqnarray*}
Now taking limit in (\ref{mar1}) and (\ref{mar2}), it follows that (\ref{mar3}) and (\ref{mar4}) hold.

Finally, the  estimate (\ref{es19})  follows directly from (\ref{es10}).  The proof is completed.
\end{proof}

\subsection{Faedo-Galerkin approximation}\label{sec2.3}
First we introduce the Faedo-Galerkin approximation of MVSPDE (\ref{eqSPDE}).  Define the maps
$$P_n:\mathbb{X}^{*}\rightarrow {\mathbb{H}}^n=\text{span}\big\{l_1,l_2,\cdots\,l_n\big\},~n\in\mathbb{N},$$
by
$$P_n x:=\sum\limits_{i=1}^{n}{}_{\mathbb{X}^*}\langle x,l_i\rangle_{\mathbb{X}}l_i,~x\in \mathbb{X}^*.$$
%Note that $\mathbb{X}\subset {\mathbb{H}}$ is compact and so is ${\mathbb{H}}\subset \mathbb{X}^*$, one could choose $\Big\{l_n,n\in\mathbb{N}\Big\}\subset\mathbb{X}$ such that
%$$\|P_n u\|_{\mathbb{X}^*}\leq \|u\|_{\mathbb{X}^*},~n\in\mathbb{N},~u\in\mathbb{X}^*,$$
%whose proof can refer to \cite[Lemma 4.4]{GRZ}.

It is straightforward that the restriction of $P_n$ to ${\mathbb{H}}$, denoted by $P_n|_{{\mathbb{H}}}$,  is an orthogonal projection from $\mathbb{H}$ onto ${\mathbb{H}}^n$. Let $\{g_1,g_2,\cdots\}$ denote an ONB of $U$. Let  $\tilde{P}_n$ be an orthogonal projection from $U$ onto $U^n:=\text{span}\{g_1,g_2,\cdots,g_n\}$, which is given by
\begin{equation*}
\tilde{P}_n W_t=\sum\limits_{i=1}^{n}\langle W_t,g_i\rangle_{U}g_i,~n\in\mathbb{N}.
\end{equation*}

For any $n\in\mathbb{N}$, we consider the following approximating equation on ${\mathbb{H}}^n$:
\begin{eqnarray}\label{eqf}
dX^{n}_t=P_n A(t,X^{n}_t,\mathscr{L}_{X^{n}_t})dt
+P_n\sigma(t,X^{n}_t,\mathscr{L}_{X^{n}_t})\tilde{P}_ndW_t,
\end{eqnarray}
with initial value $X^{n}_0=P_n X_0$.

\vspace{1mm}
According to  Theorem \ref{pro3}, we have the following result.
\begin{lemma}
Under the assumptions in Theorem \ref{th1}, for any $X_0\sim\mu_0\in \mathscr{P}_p(\mathbb{H})$ with $p>\eta_0$, (\ref{eqf}) has a weak solution
$(\Omega^n,\mathscr{F}^n,\{\mathscr{F}^n_t\}_{t\in[0,T]},\PP^n; X^n,W^n)$.

\end{lemma}

In the sequel, we establish some uniform moment bounds for  solutions of the approximating equation (\ref{eqf}).
\begin{lemma}\label{lem13}
Suppose that the assumptions in Theorem \ref{th1} hold. For any $p> \eta_0$, there exists a constant $C_{p,T}>0$ independent of $n$ such that
\begin{equation}\label{es7}
\mathbb{E}^n\Big[\sup_{t\in[0,T]}\|X^{n}_t\|_{\mathbb{H}}^p\Big]+\mathbb{E}^n\Big[\int_0^{T}\|X^{n}_t\|_{\mathbb{H}}^{p-2}\mathscr{N}_1(X^{n}_t)dt\Big]+\mathbb{E}^n\Big[\int_0^{T}\mathscr{N}_1(X^{n}_t)dt\Big]\leq C_{p,T}.
\end{equation}
Moreover, there exist $C_{p,T}>0$ independent of $n$  and $r_0>2$ such that
\begin{equation}\label{es12}
\mathbb{E}^n\Big[\int_0^T\mathscr{N}_1(X^{n}_t) dt\Big]^{r_0}+\mathbb{E}^n\Big[\int_0^T\mathscr{N}_1(X^{n}_t)\|X^{n}_t\|_{\mathbb{H}}^{\gamma_2} dt\Big]^{r_0}\leq C_{p,T}.
\end{equation}

\end{lemma}

\begin{proof}
Applying It\^{o}'s formula twice, for any $t\in[0,T]$, we obtain
\begin{eqnarray}\label{es5}
\|X^{n}_t\|_{\mathbb{H}}^2
=\!\!\!\!\!\!\!\!&&\|X^{n}_0\|_{\mathbb{H}}^2+\int_0^t\Big(2{}_{{\mathbb{X}}^*}\langle A(s,X^{n}_s,\mathscr{L}_{X^{n}_s}),X^{n}_s\rangle_{\mathbb{X}}
+\|P_n\sigma(s,X^{n}_s,\mathscr{L}_{X^{n}_s})\tilde{P}_n\|_{L_2(U;{\mathbb{H}})}^2\Big)ds
\nonumber \\
\!\!\!\!\!\!\!\!&&+2\int_0^t\langle X^{n}_s,P_n\sigma(s,X^{n}_s,\mathscr{L}_{X^{n}_s})\tilde{P}_ndW^{n}_s\rangle_{\mathbb{H}}
\end{eqnarray}
and
\begin{eqnarray*}
\|X^{n}_t\|_{\mathbb{H}}^p
=\!\!\!\!\!\!\!\!&&\|X^{n}_0\|_{\mathbb{H}}^p+\frac{p(p-2)}{2}\int_0^t\|X^{n}_s\|_{\mathbb{H}}^{p-4}
\|\big(P_n\sigma(s,X^{n}_s,\mathscr{L}_{X^{n}_s})\tilde{P}_n\big)^*X^{n}_s\|_U^2ds
\nonumber \\
\!\!\!\!\!\!\!\!&&+\frac{p}{2}\int_0^t\|X^{n}_s\|_{\mathbb{H}}^{p-2}\Big(2{}_{{\mathbb{X}}^*}\langle A(s,X^{n}_s,\mathscr{L}_{X^{n}_s}),X^{n}_s\rangle_{\mathbb{X}}
+\|P_n\sigma(s,X^{n}_s,\mathscr{L}_{X^{n}_s})\tilde{P}_n\|_{L_2(U;{\mathbb{H}})}^2\Big)ds
\nonumber \\
\!\!\!\!\!\!\!\!&&+p\int_0^t\|X^{n}_s\|_{\mathbb{H}}^{p-2}\langle X^{n}_s,P_n\sigma(s,X^{n}_s,\mathscr{L}_{X^{n}_s})\tilde{P}_ndW^{n}_s\rangle_{\mathbb{H}}.
\end{eqnarray*}
By condition $\mathbf{H2}$ we have
\begin{eqnarray}\label{es2}
\|X^{n}_t\|_{\mathbb{H}}^p
\leq\!\!\!\!\!\!\!\!&&\|X^{n}_0\|_{\mathbb{H}}^p-\frac{p}{2}\int_0^t\|X^{n}_s\|_{\mathbb{H}}^{p-2}\mathscr{N}_1(X^{n}_s)ds
+C_p\int_0^t\Big(\|X^{n}_s\|_{\mathbb{H}}^p+\mathbb{E}^n\|X^{n}_s\|_{\mathbb{H}}^p+1\Big)ds
\nonumber \\
\!\!\!\!\!\!\!\!&&+p\int_0^t\|X^{n}_s\|_{\mathbb{H}}^{p-2}\langle X^{n}_s,P_n\sigma(s,X^{n}_s,\mathscr{L}_{X^{n}_s})\tilde{P}_ndW^{n}_s\rangle_{\mathbb{H}}.~
\end{eqnarray}
%For any $n\in\mathbb{N}$, we set the stopping time
%$$\tau_R^{n}=\inf\Big\{t\in[0,T]:\|X^{n}_t\|_{\mathbb{H}}\geq R\Big\}\wedge T,~R>0,$$
%where we take $\inf{\emptyset}=\infty$.  It's easy to see that
%$$\lim_{R\to\infty}\tau_R^{n}=T,~\mathbb{P}\text{-a.s.},~n\in\mathbb{N}.$$
Using  BDG's inequality, we obtain
\begin{eqnarray}\label{es3}
\!\!\!\!\!\!\!\!&&\mathbb{E}^n\Big[\sup_{t\in[0,T]}\Big|\int_0^t\|X^{n}_s\|_{\mathbb{H}}^{p-2}\langle X^{n}_s,P_n\sigma(s,X^{n}_s,\mathscr{L}_{X^{n}_s})\tilde{P}_ndW^{n}_s\rangle_{\mathbb{H}}\Big|\Big]
\nonumber \\
\leq\!\!\!\!\!\!\!\!&& C_p\mathbb{E}^n\Big[\int_0^{T}\|X^{n}_s\|_{\mathbb{H}}^{2p-2}\|\sigma(s,X^{n}_s,\mathscr{L}_{X^{n}_s})\|_{L_2(U;{\mathbb{H}})}^2ds\Big]^{\frac{1}{2}}
\nonumber \\
\leq\!\!\!\!\!\!\!\!&& C_p\mathbb{E}^n\Big[\int_0^{T}\big(\|X^{n}_t\|_{\mathbb{H}}^p+\mathbb{E}^n\|X^{n}_t\|_{\mathbb{H}}^p+1\big)dt\Big]
+\frac{1}{2}\mathbb{E}^n\Big[\sup_{t\in[0,T]} \|X^{n}_t\|_{\mathbb{H}}^p\Big].
\end{eqnarray}
Combining (\ref{es2}) with (\ref{es3})  we have
\begin{eqnarray}\label{es4}
\!\!\!\!\!\!\!\!&&\mathbb{E}^n\Big[\sup_{t\in[0,T]}\|X^{n}_t\|_{\mathbb{H}}^p\Big]+p\mathbb{E}^n\int_0^{T}\|X^{n}_t\|_{\mathbb{H}}^{p-2}\mathscr{N}_1(X^{n}_t)dt
\nonumber \\
\leq\!\!\!\!\!\!\!\!&&  \mathbb{E}^n\|X^{n}_0\|_{\mathbb{H}}^p+C_p\mathbb{E}^n\int_0^{T}\|X^{n}_t\|_{\mathbb{H}}^pdt+C_{p,T}.
\end{eqnarray}
Applying  Gronwall's lemma, we find
\begin{eqnarray*}
\!\!\!\!\!\!\!\!&&\mathbb{E}^n\Big[\sup_{t\in[0,T]}\|X^{n}_t\|_{\mathbb{H}}^p\Big]+\mathbb{E}^n\int_0^{T}\|X^{n}_t\|_{\mathbb{H}}^{p-2}\mathscr{N}_1(X^{n}_t)dt
\nonumber \\
\leq\!\!\!\!\!\!\!\!&&
 C_{p,T}\big(1+\mathbb{E}^n\|X_0^n\|_{\mathbb{H}}^p\big)\leq C_{p,T}\big(1+\mu_0(\|\cdot\|_{\mathbb{H}}^p)\big)\leq C_{p,T}.
\end{eqnarray*}
Furthermore, due to (\ref{es5}) we also have
$$\mathbb{E}^n\int_0^{T}\mathscr{N}_1(X^{n}_t)dt\leq C_{T}.$$

By BDG's inequality and (\ref{es7}), this leads to
\begin{eqnarray*}
\mathbb{E}^n\Big[\int_0^T\mathscr{N}_1(X^{n}_t) dt\Big]^{r_0}
\leq\!\!\!\!\!\!\!\!&&C_T(1+\mathbb{E}^n\|X^{n}_0\|_{\mathbb{H}}^{2r_0})
+C_T\mathbb{E}^n\int_0^T\| X^{n}_t\|_{\mathbb{H}}^{2r_0}dt
\nonumber \\
\!\!\!\!\!\!\!\!&&
+ C_T\mathbb{E}^n\Big|\int_0^T \langle X^{n}_t,P_n\sigma(t,X^{n}_t,\mathscr{L}_{X^{n}_t})\tilde{P}_ndW^{n}_t\rangle_{\mathbb{H}}\Big|^{r_0}
\nonumber \\
\leq\!\!\!\!\!\!\!\!&&C_{p,T}
\end{eqnarray*}
and
\begin{eqnarray*}
\!\!\!\!\!\!\!\!&&\mathbb{E}^n\Big[\int_0^T\mathscr{N}_1(X^{n}_t)\|X^{n}_t\|_{\mathbb{H}}^{\gamma_2} dt\Big]^{r_0}
\nonumber \\
\leq\!\!\!\!\!\!\!\!&&C_{p,T}(1+\mathbb{E}^n\|X^{n}_0\|_{\mathbb{H}}^{(2+\gamma_2)r_0})
+C_{p,T}\mathbb{E}^n\int_0^T\| X^{n}_t\|_{\mathbb{H}}^{(2+\gamma_2)r_0}dt
\nonumber \\
\!\!\!\!\!\!\!\!&&
+ C_{p,T}\mathbb{E}^n\Big|\int_0^T \|X^{n}_t\|_{\mathbb{H}}^{\gamma_2}\langle X^{n}_t,P_n\sigma(t,X^{n}_t,\mathscr{L}_{X^{n}_t})\tilde{P}_ndW^{n}_t\rangle_{\mathbb{H}}\Big|^{r_0}
\nonumber \\
\leq\!\!\!\!\!\!\!\!&&C_{p,T}.
\end{eqnarray*}
Thus (\ref{es12}) holds and  the proof is completed.
\end{proof}

\vspace{1mm}
The following lemma contains the time H\"{o}lder continuity of $X^{n}_t$, which plays an important role in the proof of tightness for laws of the sequence $\{X^{n}\}_{n\in\mathbb{N}}$.
\begin{lemma}\label{lem2}
Suppose that the assumptions in Theorem \ref{th1} hold. If $1<\gamma_1< 2$, there exist $r_1>\frac{\gamma_1}{\gamma_1-1}$ and $C_{p,T}>0$ such that for any $\gamma\in(0,\frac{\gamma_1-1}{\gamma_1}-\frac{1}{r_1})$,
\begin{equation}\label{es13}
\mathbb{E}^n\|X^{n}_t-X^{n}_s\|_{\mathbb{X}^*}^{r_1}\leq C_{p,T}|t-s|^{\frac{r_1(\gamma_1-1)}{\gamma_1}},
\end{equation}
and
\begin{equation}\label{es8}
\sup_{n\in\mathbb{N}}\mathbb{E}^n\Big[\sup_{s\neq t\in[0,T]}\frac{\|X^{n}_t-X^{n}_s\|_{\mathbb{X}^*}^{r_1}}{|t-s|^{r_1\gamma}}\Big]<\infty.
\end{equation}
If $\gamma_1\geq2$, there exist $r_1>2$ and $C_{p,T}>0$ such that for any $\gamma\in(0,\frac{1}{2}-\frac{1}{r_1})$,
\begin{equation}\label{es14}
\mathbb{E}^n\|X^{n}_t-X^{n}_s\|_{\mathbb{X}^*}^{r_1}\leq C_{p,T}|t-s|^{\frac{r_1}{2}},
\end{equation}
and
\begin{equation}\label{es11}
\sup_{n\in\mathbb{N}}\mathbb{E}^n\Big[\sup_{s\neq t\in[0,T]}\frac{\|X^{n}_t-X^{n}_s\|_{\mathbb{X}^*}^{r_1}}{|t-s|^{r_1\gamma}}\Big]<\infty.
\end{equation}
\end{lemma}
\begin{proof}
Note that the following equality holds in $\mathbb{X}^*$
$$X^{n}_t-X^{n}_s=\int_s^tP_nA(r,X^{n}_r,\mathscr{L}_{X^{n}_r})dr+\int_s^tP_n\sigma(r,X^{n}_r,\mathscr{L}_{X^{n}_r})\tilde{P}_ndW^n_r.$$
Thus for some $r_1>1$
\begin{equation}\label{es15}
\mathbb{E}^n\|X^{n}_t-X^{n}_s\|_{\mathbb{X}^*}^{r_1}\leq C(\text{I}+\text{II}),
\end{equation}
where
$$\text{I}:=\mathbb{E}^n\Big(\int_s^t\|A(r,X^{n}_r,\mathscr{L}_{X^{n}_r})\|_{\mathbb{X}^*}dr\Big)^{r_1},$$
$$\text{II}:=\mathbb{E}^n\Big\|\int_s^tP_n\sigma(r,X^{n}_r,\mathscr{L}_{X^{n}_r})\tilde{P}_ndW^n_r\Big\|_{\mathbb{X}^*}^{r_1}.$$
According to (\ref{c1}) and H\"{o}lder's inequality,
\begin{eqnarray}\label{es16}
\text{I}\leq\!\!\!\!\!\!\!\!&&C\mathbb{E}^n\Big[\int_s^t\Big(1+\mathscr{N}_1(X^{n}_r)^{\frac{1}{\gamma_1}}+\mathscr{N}_1(X^{n}_r)^{\frac{1}{\gamma_1}}\|X^{n}_r\|_{\mathbb{H}}^{\frac{\gamma_2}{\gamma_1}}+\|X^{n}_r\|_{\mathbb{H}}^{\frac{\gamma_2}{\gamma_1}}+\big(\mathbb{E}^n\|X^{n}_r\|_{\mathbb{H}}^{\theta_1}\big)^{\frac{2}{\gamma_1}}\Big)dr\Big]^{r_1}
\nonumber \\
\leq\!\!\!\!\!\!\!\!&&C|t-s|^{r_1}+C\mathbb{E}^n\Big(\int_0^T\mathscr{N}_1(X^{n}_r)dr\Big)^{\frac{r_1}{\gamma_1}}|t-s|^{\frac{r_1(\gamma_1-1)}{\gamma_1}}
\nonumber \\
\!\!\!\!\!\!\!\!&&
+C\mathbb{E}^n\Big(\int_0^T\mathscr{N}_1(X^{n}_r)\|X^{n}_r\|_{\mathbb{H}}^{\gamma_2}dr\Big)^{\frac{r_1}{\gamma_1}}|t-s|^{\frac{r_1(\gamma_1-1)}{\gamma_1}}
\nonumber \\
\!\!\!\!\!\!\!\!&&
+C\mathbb{E}^n\Big(\sup_{r\in[0,T]}\|X^{n}_r\|_{\mathbb{H}}^{\frac{ r_1\gamma_2}{\gamma_1}}\Big)|t-s|^{r_1}
\nonumber \\
\!\!\!\!\!\!\!\!&&
+C\Big[\mathbb{E}^n\Big(\sup_{r\in[0,T]}\|X^{n}_r\|_{\mathbb{H}}^{{\theta_1}}\Big)\Big]^{\frac{2r_1}{\gamma_1}}|t-s|^{r_1}
\nonumber \\
\leq\!\!\!\!\!\!\!\!&&C_{p,T}|t-s|^{r_1}+C_{p,T}|t-s|^{\frac{r_1(\gamma_1-1)}{\gamma_1}}.
\end{eqnarray}
By BDG's inequality and (\ref{c2}), we deduce
\begin{eqnarray}\label{es17}
\text{II}\leq\!\!\!\!\!\!\!\!&&C\mathbb{E}^n\Big(\int_s^t\|\sigma(r,X^{n}_r,\mathscr{L}_{X^{n}_r})\|_{L_2(U;{\mathbb{H}})}^2dr\Big)^{\frac{r_1}{2}}
\nonumber \\
\leq\!\!\!\!\!\!\!\!&&C|t-s|^{\frac{r_1-2}{2}}\int_s^t\Big(1+\mathbb{E}^n\|X^{n}_r\|_{\mathbb{H}}^{r_1}\Big)dr
\nonumber \\
\leq\!\!\!\!\!\!\!\!&&C_{p,T}|t-s|^{\frac{r_1}{2}}.
\end{eqnarray}

Therefore, (\ref{es13}) and (\ref{es14}) follow by combining (\ref{es15})-(\ref{es17}). Then it is clear that the estimates (\ref{es8}) and (\ref{es11}) are  direct consequences of Kolmogorov's criterion. The proof is completed.
\end{proof}

\subsection{Proof of existence of weak solutions}\label{sec2.4}
In this subsection, we first investigate the tightness of laws of $\{X^{n}\}_{n\in\mathbb{N}}$.  To this end, we recall
\begin{equation}\label{es73}
\mathbb{K}_1:=C([0,T];\mathbb{X}^*)\cap L^{q}([0,T];\mathbb{Y}),
\end{equation}
which is a Polish space under the metric
$$d_T(u,v):=\sup_{t\in[0,T]}\|u_t-v_t\|_{\mathbb{X}^*}+\Big(\int_0^T\|u_t-v_t\|_{\mathbb{Y}}^q dt\Big)^{\frac{1}{q}}.$$
We further set
$$\mathbb{K}_2:=C([0,T];\mathbb{X}^*)\cap L^{q}([0,T];\mathbb{Y})\cap L^{q}_w([0,T];{\mathbb{V}}),$$
where $ L^{q}_w([0,T];{\mathbb{V}})$ denotes the space $L^{q}([0,T];{\mathbb{V}})$ endowed with the weak topology.
\begin{remark}\label{k2t}
Here the intersection space $\mathbb{K}_2$ takes the following intersection topology denoted by $\tau_{\mathbb{K}_2}$: the class of
open sets of $\mathbb{K}_2$  are generated by the sets of the form $\mathscr{O}_1\cap\mathscr{O}_2\cap\mathscr{O}_3$, where $\mathscr{O}_1$, $\mathscr{O}_2$, $\mathscr{O}_3$ are the open sets in the above three spaces, respectively. Let $\mathscr{B}(\tau_{\mathbb{K}_2})$ be the corresponding Borel $\sigma$-algebra. By similar argument as \cite[Theorem B.5]{Liang}, we have
$$\mathbb{K}_2\in \mathscr{B}(\mathbb{K}_1),~~\mathscr{B}(\tau_{\mathbb{K}_2})=\mathscr{B}(\mathbb{K}_1)\cap \mathbb{K}_2.$$

\end{remark}

Then we have the following lemma.
\begin{lemma}\label{lem6}
The set of measures $\{\mathscr{L}_{X^{n}}\}_{n\in\mathbb{N}}$ is tight  on  $\mathbb{K}_2$ $($hence on $\mathbb{K}_1$$)$.
\end{lemma}
\begin{proof}
We postpone the proof  to the Appendix below.
\end{proof}

\vspace{2mm}
In the following, we shall give the detailed proof of Theorem \ref{th1}.

\medskip
Since $\{\nu^n:=\mathscr{L}_{X^{n}}\}_{n\in\mathbb{N}}$ is tight in  $\mathbb{K}_2$ (also in $\mathbb{K}_1$), then $\{\Gamma^n:=\mathscr{L}_{(X^{n},W^n)}\}_{n\in\mathbb{N}}$ is also tight. In order to apply Jakubowski's version of the Skorokhod theorem (cf.~Lemma \ref{sko1} in Appendix), we have to check that $\mathbb{K}_2$ satisfies the conditions of Lemma \ref{sko1}. Specifically, it is necessary to prove that on each space appearing in the definition of $\mathbb{K}_2$ there exists a
countable set of continuous real-valued functions separating points. Firstly, since $C([0,T];\mathbb{X}^*)$ and  $L^{q}([0,T];\mathbb{Y})$ are separable Banach spaces, it is easy to see the conditions in Lemma \ref{sko1} hold. Secondly, for the space $L^{q}_w([0,T];{\mathbb{V}})$, it suffices to put
$$f_m(u):=\int_0^T{}_{\mathbb{V}^*}\langle v_m(t),u(t)\rangle_{\mathbb{V}}dt\in\mathbb{R},~u\in L^{q}_w([0,T];{\mathbb{V}}),~m\in\mathbb{N},$$
where $\{v_m\}_{m\geq 1}$ is a dense subset of $L^{\frac{q}{q-1}}([0,T];{\mathbb{V}^*})$. Since $\{v_m\}_{m\geq 1}$ is dense in $L^{\frac{q}{q-1}}([0,T];{\mathbb{V}^*})$, it separates points of $L^{q}([0,T];{\mathbb{V}})$. Thus all the conditions of Lemma \ref{sko1} are satisfied. Moreover, the $\sigma$-algebra generated by the sequence of the above continuous functions
which separate the points in $\mathbb{K}_2$ is exactly $\mathscr{B}(\tau_{\mathbb{K}_2})$, since $(\mathbb{K}_2,\mathscr{B}(\tau_{\mathbb{K}_2}))$  is a standard Borel space (see Definition \ref{de5} in Appendix) whose  proof can refer to \cite[Theorem B.5,Corollary B.9]{Liang}.
%

%By the Prokhorov theorem, there exists a subsequence of $\{n\}$, we keep denoting by $\{n\}$, $\nu^n$ weakly converges to $\nu$ in $\mathscr{P}(\mathbb{K}_2)$ and $\Gamma^n\to \Gamma$ weakly in $\mathscr{P}(\mathbb{K}_2\times C([0,T];U))$.
Now by Lemma \ref{sko1} there exists a probability space $(\tilde{\Omega},\tilde{\mathscr{F}},\tilde{\PP})$, and on this space, $\mathbb{K}_2\times C([0,T];U_1)$-valued random variables $(\tilde{X}^{n},\tilde{W}^{n})$ (here $U_1$ is a Hilbert space such that the embedding $U\subset U_1$ is Hilbert-Schmidt) whose law is equal to $\Gamma^n$.
Moreover,
\begin{equation}\label{es80}
\tilde{X}^{n}\to \tilde{X}\sim \nu~\text{in}~\mathbb{K}_2,~\tilde{\mathbb{P}}\text{-a.s.},~\text{as}~n\to\infty,
\end{equation}
\begin{equation}\label{es42}
\tilde{W}^{n}\to \tilde{W}~\text{in}~C([0,T];U_1),~\tilde{\mathbb{P}}\text{-a.s.},~\text{as}~n\to\infty.
\end{equation}
Let $\tilde{\mathscr{F}}^{n}_t$ (resp.~$\tilde{\mathscr{F}}_t$) be the filtration satisfying the usual conditions and generated by $\{\tilde{X}^{n}_s,\tilde{W}^{n}_s:s\leq t\}$ (resp.~$\{\tilde{X}_s,\tilde{W}_s:s\leq t\}$). It is clear that $\tilde{X}_t$ is $\{\tilde{\mathscr{F}}_t\}_{t\in[0,T]}$-adapted and continuous in $\mathbb{X}^*$, and $\tilde{W}_t$ is an $\{\tilde{\mathscr{F}}_t\}_{t\in[0,T]}$-cylindrical Wiener process on $U$. Moreover, by Lemma \ref{lem13} we have the same bounds for $\tilde{X}^{n}_t$ under $\tilde{\mathbb{P}}$, i.e.
\begin{equation}\label{es32}
\tilde{\mathbb{E}}\Big[\sup_{t\in[0,T]}\|\tilde{X}^{n}_t\|_{\mathbb{H}}^p\Big]+\tilde{\mathbb{E}}\Big[\int_0^{T}\mathscr{N}_1(\tilde{X}^{n}_t)dt\Big]^{r_0}+\tilde{\mathbb{E}}\Big[\int_0^{T}\|\tilde{X}^{n}_t\|_{\mathbb{H}}^{p-2}\mathscr{N}_1(\tilde{X}^{n}_t)dt\Big]\leq C_{p,T},
\end{equation}
for some $r_0>2$.
%and for any $r_0\in(2,\frac{p}{2+\beta}]$,
%\begin{equation}\label{es33}
%\tilde{\mathbb{E}}\Big[\int_0^T\mathscr{N}_1(\tilde{X}^{n}_t) dt\Big]^{r_0}\leq C_{p,T}.
%\end{equation}

Note that $\nu^n\to\nu$ weakly in $\mathscr{P}(\mathbb{K}_1)$ (here choosing a subsequence if necessary), then by \cite[Theorem 5.5]{CD1} and (\ref{es32}), we further obtain
\begin{equation}\label{es38}
\nu^n\to\nu~\text{in}~\mathscr{P}_2(\mathbb{K}_1),
\end{equation}
which turns out that
$$\mathbb{W}_{2,T,\mathbb{X}^*}(\nu^n,\nu)\to 0~\text{as}~n\to\infty.$$

Now we define
$$\Upsilon(T,w):=\sup_{t\in[0,T]}\|w_t\|_{\mathbb{H}}^p+\int_0^{T}\mathscr{N}_1(w_t)dt+\int_0^{T}\|w_t\|_{\mathbb{H}}^{p-2}\mathscr{N}_1(w_t)dt.$$
Then by (\ref{es80}), (\ref{es32}), the lower semicontinuity of $\Upsilon(T,\cdot)$ and Fatou's lemma, we obtain
\begin{eqnarray*}
\tilde{\mathbb{E}}\Upsilon(T,\tilde{X})\leq\tilde{\mathbb{E}}\big[\liminf_{n\to\infty}\Upsilon(T,\tilde{X}^n)]\leq\liminf_{n\to\infty}\tilde{\mathbb{E}}\Upsilon(T,\tilde{X}^n)\leq C_{p,T},
\end{eqnarray*}
which means
\begin{eqnarray}\label{es119}
\tilde{\mathbb{E}}\Big[\sup_{t\in[0,T]}\|\tilde{X}_t\|_{\mathbb{H}}^p\Big]+\tilde{\mathbb{E}}\Big[\int_0^{T}\mathscr{N}_1(\tilde{X}_t)dt\Big]+\tilde{\mathbb{E}}\Big[\int_0^{T}\|\tilde{X}_t\|_{\mathbb{H}}^{p-2}\mathscr{N}_1(\tilde{X}_t)dt\Big]\leq C_{p,T}.
\end{eqnarray}

Let $l\in\mathbb{X}$. Since $\tilde{X}^{n}$ coincides in law with $X^{n}$, define
\begin{equation}\label{eqm}
\tilde{\mathscr{M}}_l^{n}(t):={}_{\mathbb{X}^*}\langle \tilde{X}^{n}_t,l\rangle_{\mathbb{X}}-{}_{\mathbb{X}^*}\langle \tilde{X}^{n}_0,l\rangle_{\mathbb{X}}-\int_0^t{}_{\mathbb{X}^*}\langle P_nA(s,\tilde{X}^{n}_s,\mathscr{L}_{\tilde{X}^{n}_s}),l\rangle_{\mathbb{X}}ds,~t\in[0,T],
\end{equation}
which is a square integrable martingale w.r.t.~the filtration $\tilde{\mathscr{F}}^{n}_t$ with quadratic variation
$$\langle \tilde{\mathscr{M}}^{n}_l\rangle(t)=\int_0^t \|\big(P_n\sigma(s,\tilde{X}^{n}_s,\mathscr{L}_{\tilde{X}^{n}_s})\tilde{P}_n\big)^*l\|_U^2ds.$$

For any $R>0$ and $(t,w,\nu)\in[0,T]\times\mathbb{K}_2\times\mathscr{P}(\mathbb{K}_1)$, we set
\begin{eqnarray}\label{es20}
\Theta_R(t,w,\nu)=\Theta^{(1)}_R(t,w)+\Theta^{(2)}_R(t,w,\nu),
\end{eqnarray}
where
$$\Theta^{(1)}_R(t,w):={}_{\mathbb{X}^*}\langle w_t,l\rangle_{\mathbb{X}}\cdot\chi_R({}_{\mathbb{X}^*}\langle w_t,l\rangle_{\mathbb{X}})-{}_{\mathbb{X}^*}\langle w_0,l\rangle_{\mathbb{X}}\cdot\chi_R({}_{\mathbb{X}^*}\langle w_0,l\rangle_{\mathbb{X}}),$$
$$\Theta^{(2)}_R(t,w,\nu):=-\int_0^t{}_{\mathbb{X}^*}\langle A(s,w_s,\nu_s),l\rangle_{\mathbb{X}}\cdot\chi_R({}_{\mathbb{X}^*}\langle A(s,w_s,\nu_s),l\rangle_{\mathbb{X}})ds.$$
Here  $\chi_R\in C^{\infty}_c(\mathbb{R})$ is a cut-off function with
$$\chi_R(r)=\begin{cases} 1,~~~~|r|\leq R&\quad\\
0,~~~~|r|>2R.&\quad\end{cases}$$
Notice that for any $t\in[0,T]$,
$$(w,\nu)\mapsto \Theta_R(t,w,\nu)~\text{is bounded continuous on}~\mathbb{K}_2\times\mathscr{P}_2(\mathbb{K}_1).$$ In fact,
if $(w^n,\nu^n)\to(w,\nu)$ in $\mathbb{K}_2\times\mathscr{P}_2(\mathbb{K}_1)$ as $n\to\infty$, it is easy to see that
\begin{eqnarray}
\!\!\!\!\!\!\!\!&&w^n\to w~~\text{in}~C([0,T];\mathbb{X}^*);
\nonumber \\
\!\!\!\!\!\!\!\!&&
w^n\to w~~\text{in}~L^q([0,T];\mathbb{Y}); \label{es34}
 \\
\!\!\!\!\!\!\!\!&&
w^n\to w~~\text{weakly in}~L^q([0,T];{\mathbb{V}}),\nonumber
\end{eqnarray}
and for any $t\in[0,T]$,
\begin{equation}\label{es35}
\nu^n_t\to \nu_t~~\text{in}~\mathscr{P}_2(\mathbb{X}^*),
\end{equation}
which follows from the fact
$$\mathbb{W}_{2,\mathbb{X}^*}(\nu^n_t,\nu_t)\leq\mathbb{W}_{2,T,\mathbb{X}^*}(\nu^n,\nu),~t\in[0,T].$$

Then it is obvious that
$$\lim_{n\to\infty}\Theta^{(1)}_R(t,w^n)=\Theta^{(1)}_R(t,w).$$
Since $(x,\mu)\mapsto{}_{\mathbb{X}^*}\langle A(t,x,\mu),l\rangle_{\mathbb{X}}\cdot\chi_R({}_{\mathbb{X}^*}\langle A(t,x,\mu),l\rangle_{\mathbb{X}})$ is a bounded continuous function on $\mathbb{Y}\times \mathscr{P}_2(\mathbb{X}^*)$, then by (\ref{es34})-(\ref{es35}), we have
$$\lim_{n\to\infty}\Theta^{(2)}_R(t,w^n,\nu^n)=\Theta^{(2)}_R(t,w,\nu).$$
Hence the assertion follows.

On the other hand,  we denote
\begin{equation}\label{es36}
\tilde{\mathscr{M}}_l(t):={}_{\mathbb{X}^*}\langle \tilde{X}_t,l\rangle_{\mathbb{X}}-{}_{\mathbb{X}^*}\langle \tilde{X}_0,l\rangle_{\mathbb{X}}-\int_0^t{}_{\mathbb{X}^*}\langle A(s,\tilde{X}_s,\mathscr{L}_{\tilde{X}_s}),l\rangle_{\mathbb{X}}ds,~t\in[0,T],
\end{equation}
and for any $(t,w,\nu)\in[0,T]\times\mathbb{K}_2\times\mathscr{P}(\mathbb{K}_1)$
\begin{equation*}
\Theta(t,w,\nu):=\Theta^{(1)}(t,w)+\Theta^{(2)}(t,w,\nu),
\end{equation*}
where
$$\Theta^{(1)}(t,w):={}_{\mathbb{X}^*}\langle w_t,l\rangle_{\mathbb{X}}-{}_{\mathbb{X}^*}\langle w_0,l\rangle_{\mathbb{X}},$$
$$\Theta^{(2)}(t,w,\nu):=-\int_0^t{}_{\mathbb{X}^*}\langle A(s,w_s,\nu_s),l\rangle_{\mathbb{X}}ds.$$
Moreover, we also denote
\begin{equation*}
\Theta^n(t,w,\nu):={}_{\mathbb{X}^*}\langle w_t,l\rangle_{\mathbb{X}}-{}_{\mathbb{X}^*}\langle w_0,l\rangle_{\mathbb{X}}-\int_0^t{}_{\mathbb{X}^*}\langle P_nA(s,w_s,\nu_s),l\rangle_{\mathbb{X}}ds.
\end{equation*}

Now we give the following two lemmas.
%In the sequel, we will claim that $\mathbb{P}$-a.s.~$\Theta(\mathscr{S})=0$. Based on this, we can show that $\mathscr{S}(\omega)$ is a martingale solution to MVSPDE (\ref{eqSPDE}).
%
%\begin{lemma}\label{lem5}
%Under the assumptions in Theorem \ref{th2}, we have
%$$\mathbb{E}|\Theta(\mathscr{S}^N)|^2\to0~~\text{as}~N\to\infty.$$
%\end{lemma}
%
%\begin{proof}
%In view of (\ref{eqi}), (\ref{es29}) and the definition of $\Theta$, we know
%$$\Theta(\mathscr{S}^N)=\frac{1}{N}\sum_{i=1}^{N}\Big(\int_s^t\langle \sigma(r,X^{n}_r,\mathscr{S}^N_r)dW^i_r,l\rangle_{\mathbb{H}}\Big)\Phi_1(X^{n}_{s_1})\cdots\Phi_m(X^{n}_{s_m}).$$
%Since $\Phi_1,\ldots,\Phi_m$ are bounded and $W^1,\ldots,W^N$ are independent, by the isometry of stochastic integrals, we have
%\begin{eqnarray*}
%\mathbb{E}|\Theta(\mathscr{S}^N)|^2\leq\!\!\!\!\!\!\!\!&&\frac{C}{N^2}\mathbb{E}\Big|\sum_{i=1}^{N}\int_s^t\langle \sigma(r,X^{n}_r,\mathscr{S}^N_r)dW^i_r,l\rangle_{\mathbb{H}}\Big|^2
%\nonumber \\
%=\!\!\!\!\!\!\!\!&&~\frac{C\|l\|_{\mathbb{V}}^2}{N^2}\sum_{i=1}^{N}\int_s^t\mathbb{E}\|\sigma(r,X^{n}_r,\mathscr{S}^N_r)\|_{L_2(U;{\mathbb{H}})}^2dr
%\nonumber \\
%\leq\!\!\!\!\!\!\!\!&&~\frac{C_T\|l\|_{\mathbb{V}}^2}{N}(t-s),
%\end{eqnarray*}
%which implies the assertion.  \hspace{\fill}$\Box$
%\end{proof}

\begin{lemma}\label{lem9} Under the assumptions in Theorem \ref{th1}, for any $0\leq s<t\leq T$,
\begin{equation}\label{es43}
\tilde{\mathbb{E}}\big[\tilde{\mathscr{M}}_l(t)\big|\tilde{\mathscr{F}}_s\big]=\tilde{\mathscr{M}}_l(s).
\end{equation}
\end{lemma}
\begin{proof}
Note that
\begin{eqnarray}\label{es24}
\!\!\!\!\!\!\!\!&&\tilde{\mathbb{E}}|\tilde{\mathscr{M}}_l^n(t)-\tilde{\mathscr{M}}_l(t)|
\nonumber \\
=\!\!\!\!\!\!\!\!&&\tilde{\mathbb{E}}|\Theta^n(t,\tilde{X}^{n},\mathscr{L}_{\tilde{X}^{n}})-\Theta(t,\tilde{X}^{n},\mathscr{L}_{\tilde{X}^{n}})|
\nonumber \\
\!\!\!\!\!\!\!\!&&+
\tilde{\mathbb{E}}|\Theta(t,\tilde{X}^{n},\mathscr{L}_{\tilde{X}^{n}})-\Theta_R(t,\tilde{X}^{n},\mathscr{L}_{\tilde{X}^{n}})|
\nonumber \\
\!\!\!\!\!\!\!\!&&
+\tilde{\mathbb{E}}|\Theta_R(t,\tilde{X}^{n},\mathscr{L}_{\tilde{X}^{n}})-\Theta_R(t,\tilde{X},\mathscr{L}_{\tilde{X}})|
\nonumber \\
\!\!\!\!\!\!\!\!&&+\tilde{\mathbb{E}}|\Theta_R(t,\tilde{X},\mathscr{L}_{\tilde{X}})-\Theta(t,\tilde{X},\mathscr{L}_{\tilde{X}})|.
\end{eqnarray}
First, by the assumption (\ref{c1}) and the estimate (\ref{es119}), we deduce that
\begin{eqnarray}\label{es21}
\!\!\!\!\!\!\!\!&&\lim_{n\to\infty}\tilde{\mathbb{E}}|\Theta^n(t,\tilde{X}^{n},\mathscr{L}_{\tilde{X}^{n}})-\Theta(t,\tilde{X}^{n},\mathscr{L}_{\tilde{X}^{n}})|
\nonumber \\
\leq\!\!\!\!\!\!\!\!&&\lim_{n\to\infty}\tilde{\mathbb{E}}\Big|\int_0^t{}_{\mathbb{X}^*}\langle A(s,\tilde{X}^{n}_s,\mathscr{L}_{\tilde{X}^{n}_s}),(P_n-I)l\rangle_{\mathbb{X}}ds\Big|
\nonumber \\
=\!\!\!\!\!\!\!\!&&0.
\end{eqnarray}
On the other hand, we know that $\Theta_R(t,\cdot,\cdot)$ is continuous and bounded, by (\ref{es80}), (\ref{es38}) and the dominated convergence theorem
\begin{equation}\label{es39}
\lim_{n\to\infty}\tilde{\mathbb{E}}|\Theta_R(t,\tilde{X}^{n},\mathscr{L}_{\tilde{X}^{n}})-\Theta_R(t,\tilde{X},\mathscr{L}_{\tilde{X}})|=0.
\end{equation}
For the second term on right hand side of (\ref{es24}),  we know
\begin{eqnarray}\label{es25}
\!\!\!\!\!\!\!\!&&\tilde{\mathbb{E}}|\Theta(t,\tilde{X}^{n},\mathscr{L}_{\tilde{X}^{n}})-\Theta_R(t,\tilde{X}^{n},\mathscr{L}_{\tilde{X}^{n}})|
\nonumber \\
\leq\!\!\!\!\!\!\!\!&&\tilde{\mathbb{E}}|\Theta^{(1)}(t,\tilde{X}^{n})-\Theta^{(1)}_R(t,\tilde{X}^{n})|
\nonumber \\
\!\!\!\!\!\!\!\!&&+\tilde{\mathbb{E}}|\Theta^{(2)}(t,\tilde{X}^{n},\mathscr{L}_{\tilde{X}^{n}})-\Theta^{(2)}_R(t,\tilde{X}^{n},\mathscr{L}_{\tilde{X}^{n}})|.
\end{eqnarray}
By assumption (\ref{c1}) and (\ref{es32}),
\begin{eqnarray*}
\!\!\!\!\!\!\!\!&&\tilde{\mathbb{E}}|\Theta^{(2)}(t,\tilde{X}^{n},\mathscr{L}_{\tilde{X}^{n}})-\Theta_R^{(2)}(t,\tilde{X}^{n},\mathscr{L}_{\tilde{X}^{n}})|
\nonumber \\
\leq\!\!\!\!\!\!\!\!&& \tilde{\mathbb{E}}\Big[\int_0^T|{}_{\mathbb{X}^*}\langle A(s,\tilde{X}^{n}_s,\mathscr{L}_{\tilde{X}^{n}_s}),l\rangle_{\mathbb{X}}|\cdot \mathbf{1}_{\big\{|{}_{\mathbb{X}^*}\langle A(s,\tilde{X}^{n}_s,\mathscr{L}_{\tilde{X}^{n}_s}),l\rangle_{\mathbb{X}}|\geq R\big\}}ds\Big]
\nonumber \\
\leq\!\!\!\!\!\!\!\!&&\|l\|_{\mathbb{X}}\Big[\Big(\int_0^T\tilde{\mathbb{E}}\|A(s,\tilde{X}^{n}_s,\mathscr{L}_{\tilde{X}^{n}_s})\|_{\mathbb{X}^*}^{\gamma_1}ds\Big)^{\frac{1}{\gamma_1}}
\cdot \Big(\int_0^T\tilde{\mathbb{P}}\big(|{}_{\mathbb{X}^*}\langle A(s,\tilde{X}^{n}_s,\mathscr{L}_{\tilde{X}^{n}_s}),l\rangle_{\mathbb{X}}|\geq R\big)ds\Big)^{\frac{\gamma_1-1}{\gamma_1}}\Big]
\nonumber \\
\leq\!\!\!\!\!\!\!\!&&\|l\|_{\mathbb{X}}\Big(\tilde{\mathbb{E}}\int_0^T\|A(s,\tilde{X}^{n}_s,\mathscr{L}_{\tilde{X}^{n}_s})\|_{{\mathbb{X}}^*}^{\gamma_1}ds\Big)\Big/R^{\gamma_1-1}
\nonumber \\
\leq\!\!\!\!\!\!\!\!&&C\|l\|_{\mathbb{X}}\Big(\tilde{\mathbb{E}}\int_0^T\big(1+\mathscr{N}_1(\tilde{X}^{n}_s)+\mathscr{N}_1(\tilde{X}^{n}_s)\|X^{n}_s\|_{\mathbb{H}}^{\gamma_2}+\|X^{n}_s\|_{\mathbb{H}}^{\gamma_2}+(\tilde{\mathbb{E}}\|X^{n}_s\|_{\mathbb{H}}^{\theta_1})^2\big)ds\Big)\Big/R^{\gamma_1-1}
\nonumber \\
\leq\!\!\!\!\!\!\!\!&&C_T\|l\|_{\mathbb{X}}\Big/R^{\gamma_1-1},
\end{eqnarray*}
where the constant $C_T$ is independent of $n$.  Thus it is easy to see
\begin{equation}\label{es27}
\lim_{R\to\infty}\sup_{n\in\mathbb{N}}\sup_{t\in[0,T]}\tilde{\mathbb{E}}|\Theta^{(2)}(t,\tilde{X}^{n},\mathscr{L}_{\tilde{X}^{n}})-\Theta_R^{(2)}(t,\tilde{X}^{n},\mathscr{L}_{\tilde{X}^{n}})|=0.
\end{equation}
Similarly,  we also get
\begin{equation}\label{es28}
\lim_{R\to\infty}\sup_{n\in\mathbb{N}}\sup_{t\in[0,T]}\tilde{\mathbb{E}}|\Theta^{(1)}(t,\tilde{X}^{n})-\Theta_R^{(1)}(t,\tilde{X}^{n})|=0.
\end{equation}
Combining (\ref{es25})-(\ref{es28}), we deduce
\begin{equation}\label{es40}
\lim_{R\to\infty}\sup_{n\in\mathbb{N}}\sup_{t\in[0,T]}\tilde{\mathbb{E}}|\Theta(t,\tilde{X}^{n},\mathscr{L}_{\tilde{X}^{n}})-\Theta_R(t,\tilde{X}^{n},\mathscr{L}_{\tilde{X}^{n}})|=0.
\end{equation}
A similar argument shows
\begin{equation}\label{es41}
\lim_{R\to\infty}\sup_{n\in\mathbb{N}}\sup_{t\in[0,T]}\tilde{\mathbb{E}}|\Theta(t,\tilde{X},\mathscr{L}_{\tilde{X}})-\Theta_R(t,\tilde{X},\mathscr{L}_{\tilde{X}})|=0.
\end{equation}
Hence in view of (\ref{es21}), (\ref{es39}), (\ref{es40}) and (\ref{es41}), it follows that for any $t\in[0,T]$,
\begin{equation}\label{es41re}
\lim_{n\to\infty}\tilde{\mathbb{E}}|\tilde{\mathscr{M}}_l^n(t)-\tilde{\mathscr{M}}_l(t)|=0.
\end{equation}
Let $\Phi(\cdot)$ be any bounded continuous function on $\mathbb{K}_2\times C([0,T];U_1)$. Due to \eref{es41re} and the martingale property of $\tilde{\mathscr{M}}_l^n$, we deduce
\begin{eqnarray*}
\!\!\!\!\!\!\!\!&&\tilde{\mathbb{E}}\Big[\Big(\tilde{\mathscr{M}}_l(t)-\tilde{\mathscr{M}}_l(s)\Big)\Phi((\tilde{X},\tilde{W})|_{[0,s]})\Big]
\nonumber \\
=\!\!\!\!\!\!\!\!&&
\lim_{n\to\infty}\tilde{\mathbb{E}}\Big[\Big(\tilde{\mathscr{M}}_l^n(t)-\tilde{\mathscr{M}}_l^n(s)\Big)\Phi((\tilde{X}^n,\tilde{W}^n)|_{[0,s]})\Big]
\nonumber \\
=\!\!\!\!\!\!\!\!&&0,
\end{eqnarray*}
which implies the assertion by the arbitrariness of $\Phi$.
\end{proof}

\begin{lemma}\label{lem10} Under the assumptions in Theorem \ref{th1}, for any $0\leq s<t\leq T$,
\begin{equation}\label{es44}
\tilde{\mathbb{E}}\Big[\tilde{\mathscr{M}}_l^2(t)-\int_0^t\|\sigma(r,\tilde{X}_r,\mathscr{L}_{\tilde{X}_r})^*l\|_U^2dr   \Big|\tilde{\mathscr{F}}_s\Big]=\tilde{\mathscr{M}}_l^2(s)-\int_0^s\|\sigma(r,\tilde{X}_r,\mathscr{L}_{\tilde{X}_r})^*l\|_U^2dr.
\end{equation}
\end{lemma}
\begin{proof}
Note that by (\ref{c2}) and (\ref{es32}),
\begin{eqnarray*}
\sup_{n\in\mathbb{N}}\tilde{\mathbb{E}}|\tilde{\mathscr{M}}_l^n(t)|^p\leq\!\!\!\!\!\!\!\!&&C\sup_{n\in\mathbb{N}}\tilde{\mathbb{E}}\Big(\int_0^t\|\sigma(r,\tilde{X}_r^n,\mathscr{L}_{\tilde{X}_r^n})\|_{L_2(U;\mathbb{H})}^2\|l\|_{\mathbb{H}}^2dr\Big)^{\frac{p}{2}}
\nonumber \\
\leq\!\!\!\!\!\!\!\!&&C_T\sup_{n\in\mathbb{N}}\tilde{\mathbb{E}}\Big(\int_0^t\|\sigma(r,\tilde{X}_r^n,\mathscr{L}_{\tilde{X}_r^n})\|_{L_2(U;\mathbb{H})}^p\|l\|_{\mathbb{H}}^pdr\Big)
\nonumber \\
<\!\!\!\!\!\!\!\!&&\infty.
\end{eqnarray*}
Since $p>2$, by Vitali's convergence theorem and the convergence (\ref{es41re}), we have
$$\lim_{n\to\infty}\tilde{\mathbb{E}}|\tilde{\mathscr{M}}_l^n(t)-\tilde{\mathscr{M}}_l(t)|^2=0,$$
and
$$\lim_{n\to\infty}\tilde{\mathbb{E}}\Big|\int_0^t\|\big(P_n\sigma(r,\tilde{X}^n_r,\mathscr{L}_{\tilde{X}^n_r})\tilde{P}_n\big)^*l-\sigma(r,\tilde{X}_r,\mathscr{L}_{\tilde{X}_r})^*l\|_U^2dr\Big|=0.$$
Thus following the similar argument for proving (\ref{es43}), we obtain (\ref{es44}).
\end{proof}

\vspace{2mm}
\textbf{Proof of Theorem \ref{th1}.}  Combining (\ref{es119}), assumption $\mathbf{H3}$, Lemmas \ref{lem9} and \ref{lem10}, by the martingale representation theorem (cf.~\cite{Daz1}), it follows that (\ref{eqSPDE}) admits a weak solution in the sense of Definition \ref{de1}. Finally, (\ref{es48}) follows from (\ref{es119}).  We finish the proof.
\hspace{\fill}$\Box$

\subsection{Proof of existence and uniqueness of strong solutions}\label{proof2}

Recall Theorem \ref{th1}, there exists a weak solution denoted by $(X,W)$ to MVSPDE (\ref{eqSPDE}) and for each $l\in\mathbb{X}$ and $t\in[0,T]$,
$${}_{\mathbb{X}^*}\langle X_t,l\rangle_{\mathbb{X}}={}_{\mathbb{X}^*}\langle X_0,l\rangle_{\mathbb{X}}-\int_0^t{}_{\mathbb{X}^*}\langle A(s,X_s,\mathscr{L}_{X_s}),l\rangle_{\mathbb{X}}ds+
\langle\int_0^t \sigma(s,X_s,\mathscr{L}_{X_s})dW_s,l\rangle_{{\mathbb{H}}},$$
which satisfies the energy estimates (\ref{es48}) with $p>\eta_1:=\eta_0\vee (\beta+2)\vee \theta_2$.
From this,  by the assumptions (\ref{c2}) and  $\mathbf{H5}$ it turns out that
\begin{equation}\label{es51}
\mathbb{E}\int_0^T\|A(s,X_s,\mathscr{L}_{X_s})\|_{{\mathbb{V}}^*}^{\frac{q}{q-1}}ds<\infty,
\end{equation}
\begin{equation}\label{es52}
\mathbb{E}\int_0^T\|\sigma(s,X_s,\mathscr{L}_{X_s})\|_{L_2(U;\mathbb{H})}^{2}ds<\infty.
\end{equation} Set
$$\hat{X}_t:=X_0+\int_0^tA(s,X_s,\mathscr{L}_{X_s})ds+\int_0^t\sigma(s,X_s,\mathscr{L}_{X_s})dW_s~~\text{in}~{\mathbb{V}}^*.$$
Then due to (\ref{es53}), we know for each $l\in\mathbb{X}$ and $t\in[0,T]$,
$${}_{\mathbb{X}^*}\langle \hat{X}_t,l\rangle_{\mathbb{X}}={}_{\mathbb{X}^*}\langle X_0,l\rangle_{\mathbb{X}}-\int_0^t{}_{\mathbb{X}^*}\langle A(s,X_s,\mathscr{L}_{X_s}),l\rangle_{\mathbb{X}}ds+
\langle\int_0^t \sigma(s,X_s,\mathscr{L}_{X_s})dW_s,l\rangle_{{\mathbb{H}}},$$
which yields that
$$X_t=\hat{X}_t,~t\in[0,T],~\PP\text{-a.s.}.$$
In addition, due to Theorem \ref{th1} we have
\begin{equation}\label{es54}
\mathbb{E}\int_0^T\|X_s\|_{\mathbb{V}}^qds<\infty.
\end{equation}

Combining (\ref{es51})-(\ref{es54})  with \cite[Theorem 4.2.5]{LR1}, we conclude that (\ref{eqSPDE}) admits a weak solution
\begin{equation}\label{es55}
X\in C([0,T];\mathbb{H})\cap L^q([0,T];\mathbb{V}).
\end{equation}
%such that (\ref{es48}) holds and
%\begin{equation}\label{es56}
%\mathbb{E}\Big[\sup_{t\in[0,T]}\|X_t\|_{{\tilde{\mathbb{H}}}}^2\Big]<\infty.
%\end{equation}

Next, we show the pathwise uniqueness of solutions to (\ref{eqSPDE}) on $\mathbb{H}$.
\begin{lemma}\label{uniqueness}
Under the conditions of Theorem \ref{th2}, MVSPDE (\ref{eqSPDE}) has pathwise uniqueness.

\end{lemma}

\begin{proof}
Let $X_t,Y_t$ be two solutions of (\ref{eqSPDE}) with initial values $X_0=Y_0=\xi\in L^p(\Omega,\mathscr{F}_0,\mathbb{P};{\mathbb{H}})$ with $p>\eta_1$, which fulfill that $\mathbb{P}$-a.s.,
\begin{eqnarray*}
&&X_t=\xi+\int_0^tA(s,X_s,\mathscr{L}_{X_s})ds+\int_0^t\sigma(s,X_s,\mathscr{L}_{X_s})dW_s,~t\in[0,T],
\nonumber\\
&&Y_t=\xi+\int_0^tA(s,Y_s,\mathscr{L}_{Y_s})ds+\int_0^t\sigma(s,Y_s,\mathscr{L}_{Y_s})dW_s,~t\in[0,T].
\end{eqnarray*}
Recalling (\ref{es48}), we have
\begin{equation}\label{es57}
\mathbb{E}\Big[\sup_{t\in[0,T]}\|X_t\|_{{\mathbb{H}}}^p\Big]+\mathbb{E}\int_0^T\mathscr{N}_1(X_t)dt+\mathbb{E}\int_0^T\|X_t\|_{{\mathbb{H}}}^{p-2}\mathscr{N}_1(X_t)dt<\infty,
\end{equation}
and
\begin{equation}\label{es58}
\mathbb{E}\Big[\sup_{t\in[0,T]}\|Y_t\|_{{\mathbb{H}}}^p\Big]+\mathbb{E}\int_0^T\mathscr{N}_1(Y_t)dt+\mathbb{E}\int_0^T\|Y_t\|_{{\mathbb{H}}}^{p-2}\mathscr{N}_1(Y_t)dt<\infty.
\end{equation}
For any $R>R_0$, we define the following stopping time
\begin{eqnarray*}
\tau_R:=\!\!\!\!\!\!\!\!&&\tau_R^X\wedge\tau_R^Y
\nonumber\\
=\!\!\!\!\!\!\!\!&&\inf\Bigg\{t\in[0,T]:\|X_t\|_{\mathbb{H}}+\int_0^t\mathscr{N}_1(X_s)ds\geq R\Bigg\}
\nonumber\\
\!\!\!\!\!\!\!\!&&\wedge
\inf\Bigg\{t\in[0,T]:\|Y_t\|_{\mathbb{H}}+\int_0^t\mathscr{N}_1(Y_s)ds\geq R\Bigg\}.
\end{eqnarray*}
Applying It\^{o}'s formula, in view of  (\ref{es57})-(\ref{es58}) and the assumption $\mathbf{H4}$ we obtain that for any $t\in[0,T]$,
\begin{eqnarray}\label{es211}
\!\!\!\!\!\!\!\!&&\|X_{t\wedge\tau_R}-Y_{t\wedge\tau_R}\|_{\mathbb{H}}^2
\nonumber\\
\leq\!\!\!\!\!\!\!\!&&\int_0^{t\wedge\tau_R}
\Big(2{}_{{\mathbb{V}}^*}\langle A(s,X_s,\mathscr{L}_{X_s})-A(s,Y_s,\mathscr{L}_{Y_s}),X_s-Y_s\rangle_{\mathbb{V}}
\nonumber\\
\!\!\!\!\!\!\!\!&&~~~~
+\|\sigma(s,X_s,\mathscr{L}_{X_s})-\sigma(s,Y_s,\mathscr{L}_{Y_s})\|_{L_2(U;\mathbb{H})}^2\Big)ds+2\mathcal{M}_{t\wedge\tau_R}
\nonumber\\
\leq\!\!\!\!\!\!\!\!&&C\int_0^{t\wedge\tau_R}\big(1+\rho(X_s,\mathscr{L}_{X_s})+\eta(Y_s,\mathscr{L}_{Y_s})\big)
\nonumber\\
\!\!\!\!\!\!\!\!&&\cdot
\big(\|X_s-Y_s\|_{\mathbb{H}}^2+\mathcal{W}_{2,R,\mathbb{H}}(\mathscr{L}_{X^s},\mathscr{L}_{Y^s})^2\big)ds
+2\mathcal{M}_{t\wedge\tau_R}
\nonumber\\
\leq\!\!\!\!\!\!\!\!&&C\int_0^{t\wedge\tau_R}\big(1+\rho(X_s,\mathscr{L}_{X_s})+\eta(Y_s,\mathscr{L}_{Y_s})\big)
\nonumber\\
\!\!\!\!\!\!\!\!&&\cdot\big(\|X_s-Y_s\|_{\mathbb{H}}^2+\mathbb{E}\Big[\sup_{r\in[0,s\wedge \tau_R]}\|X_r-Y_r\|_{\mathbb{H}}^2\Big]\big)ds
+2\mathcal{M}_{t\wedge\tau_R},
\end{eqnarray}
where the last step follows from the definition (\ref{localw}) and the fact  that
$$\mathcal{W}_{2,R,\mathbb{H}}(\mathscr{L}_{X^t},\mathscr{L}_{Y^t})^2\leq \mathbb{E}\Big[\sup_{s\in[0,t\wedge \tau_R]}\|X_s-Y_s\|_{\mathbb{H}}^2\Big] .$$
Here, $\mathcal{M}_t$ is a local martingale given by
$$\mathcal{M}_t:=\int_0^t\langle X_s-Y_s, \big(\sigma(s,X_s,\mathscr{L}_{X_s})-\sigma(s,Y_s,\mathscr{L}_{Y_s})\big)dW_s\rangle_{\mathbb{H}}.$$

\vspace{1mm}
By Burkholder-Davis-Gundy's inequality, the assumption $\mathbf{H4}$ and  (\ref{es57})-(\ref{es58}), it follows that
\begin{eqnarray}\label{es26}
\!\!\!\!\!\!\!\!&&\mathbb{E}\Big[\sup_{t\in[0,T]}|\mathcal{M}_{t\wedge\tau_R}|\Big]
\nonumber\\
\leq\!\!\!\!\!\!\!\!&&\mathbb{E}\Bigg\{\int_0^{T\wedge\tau_R}\|X_t-Y_t\|_{\mathbb{H}}^2
\|\sigma(t,X_t,\mathscr{L}_{X_t})-\sigma(t,Y_t,\mathscr{L}_{Y_t})\|_{L_2(U;\mathbb{H})}^2dt\Bigg\}^{\frac{1}{2}}
\nonumber\\
\leq\!\!\!\!\!\!\!\!&&\frac{1}{2}\mathbb{E}\Big[\sup_{t\in[0,T\wedge\tau_R]}\|X_{t}-Y_{t}\|_{\mathbb{H}}^2\Big]+C\mathbb{E}\int_0^{T\wedge\tau_R}\big(1+\rho(X_t,\mathscr{L}_{X_t})+\eta(Y_t,\mathscr{L}_{Y_t})\big)\|X_t-Y_t\|_{\mathbb{H}}^2dt
\nonumber\\
\!\!\!\!\!\!\!\!&&
+C\int_0^{T}\mathbb{E}\big(1+\rho(X_t,\mathscr{L}_{X_t})+\eta(Y_t,\mathscr{L}_{Y_t})\big)\cdot
\mathcal{W}_{2,R,\mathbb{H}}(\mathscr{L}_{X^s},\mathscr{L}_{Y^s})^2dt
\nonumber\\
\leq\!\!\!\!\!\!\!\!&&\frac{1}{2}\mathbb{E}\Big[\sup_{t\in[0,T\wedge\tau_R]}\|X_{t}-Y_{t}\|_{\mathbb{H}}^2\Big]+C\mathbb{E}\int_0^{T\wedge\tau_R}\big(1+\rho(X_t,\mathscr{L}_{X_t})+\eta(Y_t,\mathscr{L}_{Y_t})\big)\|X_t-Y_t\|_{\mathbb{H}}^2dt
\nonumber\\
\!\!\!\!\!\!\!\!&&
+C\int_0^{T}\mathbb{E}\big(1+\rho(X_t,\mathscr{L}_{X_t})+\eta(Y_t,\mathscr{L}_{Y_t})\big)\cdot\mathbb{E}\Big[\sup_{s\in[0,t\wedge \tau_R]}\|X_s-Y_s\|_{\mathbb{H}}^2\Big]dt.
\end{eqnarray}
Taking $\sup_{t\in[0,T]}$ and expectation on both sides of (\ref{es211}), by (\ref{es26}) we have
\begin{eqnarray*}
\!\!\!\!\!\!\!\!&&\mathbb{E}\Big[\sup_{t\in[0,T\wedge\tau_R]}\|X_{t}-Y_{t}\|_{\mathbb{H}}^2\Big]
\nonumber\\
\leq\!\!\!\!\!\!\!\!&&
 C\mathbb{E}\int_0^{T\wedge\tau_R}\big(1+\rho(X_t,\mathscr{L}_{X_t})+\eta(Y_t,\mathscr{L}_{Y_t})\big)\|X_t-Y_t\|_{\mathbb{H}}^2dt
\nonumber\\
\!\!\!\!\!\!\!\!&&
+C\int_0^{T}\mathbb{E}\big(1+\rho(X_t,\mathscr{L}_{X_t})+\eta(Y_t,\mathscr{L}_{Y_t})\big)\cdot\mathbb{E}\Big[\sup_{s\in[0,t\wedge \tau_R]}\|X_s-Y_s\|_{\mathbb{H}}^2\Big]dt.
\end{eqnarray*}
Utilizing stochastic Gronwall's lemma (cf.~e.g.~\cite[Lemma 5.3]{GZ1}), the estimates (\ref{es22}) and (\ref{es57})-(\ref{es58}), it follows that
\begin{eqnarray}\label{es23}
\!\!\!\!\!\!\!\!&&\mathbb{E}\Big[\sup_{t\in[0,T\wedge\tau_R]}\|X_{t}-Y_{t}\|_{\mathbb{H}}^2\Big]
\nonumber\\
\leq \!\!\!\!\!\!\!\!&& C_R\int_0^{T}\mathbb{E}\big(1+\rho(X_t,\mathscr{L}_{X_t})+\eta(Y_t,\mathscr{L}_{Y_t})\big)\cdot\mathbb{E}\Big[\sup_{s\in[0,t\wedge \tau_R]}\|X_s-Y_s\|_{\mathbb{H}}^2\Big]dt.
\end{eqnarray}
Using the estimates (\ref{es22}) and (\ref{es57})-(\ref{es58}) again and applying Gronwall's lemma to (\ref{es23}) yields that
\begin{equation*}
\mathbb{E}\Big[\sup_{t\in[0,T\wedge\tau_R]}\|X_{t}-Y_{t}\|_{\mathbb{H}}^2\Big]\leq  0.
\end{equation*}
Finally, by Fatou's lemma it leads to
\begin{equation*}
\mathbb{E}\Big[\sup_{t\in[0,T]}\|X_{t}-Y_{t}\|_{\mathbb{H}}^2\Big]\leq \liminf_{R\to\infty}\mathbb{E}\Big[\sup_{t\in[0,T\wedge\tau_R]}\|X_{t}-Y_{t}\|_{\mathbb{H}}^2\Big]\leq 0.
\end{equation*}
We complete the proof.
\end{proof}

\vspace{2mm}
\textbf{Proof of Theorem \ref{th2}.}  To prove Theorem \ref{th2}, we  let $\mu=\mathscr{L}_{X_{\cdot}}$ which can be viewed as a
deterministic measure flow.  Consider the following classical (distribution independent) SPDEs
\begin{equation}\label{eq7}
d\bar{X}_t=A^{\mu}(t,\bar{X}_t)dt+\sigma^{\mu}(t,\bar{X}_t)dW_t,~\bar{X}_0\sim\mu_0= \mathscr{L}_{X_0},
\end{equation}
where $A^{\mu}(t,u):=A(t,u,\mu_t)$, $\sigma^{\mu}(t,u):=\sigma(t,u,\mu_t)$. Under the assumptions in Theorem \ref{th2},  we claim that Eq.~(\ref{eq7}) also has pathwise uniqueness.  Indeed,
let $\bar{X}^1,\bar{X}^2$ be two solutions of (\ref{eq7}) with same initial value $\xi\in L^p(\Omega,\mathscr{F}_0,\mathbb{P};\mathbb{H})$. First,
we note that the estimate (\ref{es48}) also holds for $\bar{X}^1,\bar{X}^2$, i.e.,
\begin{eqnarray}
&&\mathbb{E}\Big[\sup_{t\in[0,T]}\|\bar{X}^1_t\|_{{\mathbb{H}}}^p\Big]+\mathbb{E}\int_0^T\mathscr{N}_1(\bar{X}^1_t)dt+\mathbb{E}\int_0^T\|\bar{X}^1_t\|_{{\mathbb{H}}}^{p-2}\mathscr{N}_1(\bar{X}^1_t)dt<\infty,\label{es30}
\\
&&\mathbb{E}\Big[\sup_{t\in[0,T]}\|\bar{X}^2_t\|_{{\mathbb{H}}}^p\Big]+\mathbb{E}\int_0^T\mathscr{N}_1(\bar{X}^2_t)dt+\mathbb{E}\int_0^T\|\bar{X}^2_t\|_{{\mathbb{H}}}^{p-2}\mathscr{N}_1(\bar{X}^2_t)dt<\infty.\label{es31}
\end{eqnarray}
We denote
$$\varphi^\mu_t:=C+\rho(\bar{X}^1_t,\mu_t)+\eta(\bar{X}^2_t,\mu_t).$$
By It\^{o}'s formula and the product rule we have
\begin{eqnarray*}
\!\!\!\!\!\!\!\!&&e^{-\int_0^t\varphi^\mu_sds}\|\bar{X}^1_t-\bar{X}^2_t\|_{\mathbb{H}}^2
\nonumber\\
\leq\!\!\!\!\!\!\!\!&&\int_0^te^{-\int_0^s\varphi^\mu_rdr}
\Big(2{}_{{\mathbb{V}}^*}\langle A^{\mu}(s,\bar{X}^1_s)-A^{\mu}(s,\bar{X}^2_s),\bar{X}^1_s-\bar{X}^2_s\rangle_{\mathbb{V}}
\nonumber\\
\!\!\!\!\!\!\!\!&&~~~~
+\|\sigma^{\mu}(s,\bar{X}^1_s)-\sigma^{\mu}(s,\bar{X}^2_s)\|_{L_2(U;\mathbb{H})}^2
-\varphi^\mu_s\|\bar{X}^1_s-\bar{X}^2_s\|_{\mathbb{H}}^2\Big)ds
\nonumber\\
\!\!\!\!\!\!\!\!&&+2\int_0^te^{-\int_0^s\varphi^\mu_rdr}
\langle \bar{X}^1_s-\bar{X}^2_s,\big(\sigma^{\mu}(s,\bar{X}^1_s)-\sigma^{\mu}(s,\bar{X}^2_s)\big)dW_s\rangle_{\mathbb{H}}
\nonumber\\
\leq\!\!\!\!\!\!\!\!&&2\int_0^te^{-\int_0^s\varphi^\mu_rdr}
\langle \bar{X}^1_s-\bar{X}^2_s,\big(\sigma^{\mu}(s,\bar{X}^1_s)-\sigma^{\mu}(s,\bar{X}^2_s)\big)dW_s\rangle_{\mathbb{H}},
\end{eqnarray*}
where we used the assumption $\mathbf{H4}$ in the last step. Taking expectation on both sides of the above inequality, by (\ref{es30})-(\ref{es31}) we have
\begin{equation}\label{es33}
\mathbb{E}\Big[e^{-\int_0^t\varphi^\mu_sds}\|\bar{X}^1_t-\bar{X}^2_t\|_{\mathbb{H}}^2\Big]\leq 0.
\end{equation}
Note that in view of (\ref{es22}) and (\ref{es30})-(\ref{es31}) it follows that
\begin{equation}\label{es37}
\int_0^T\big(\rho(\bar{X}^1_s,\mu_s)+\eta(\bar{X}^2_s,\mu_s)\big)ds<\infty~~\mathbb{P}\text{-a.s.}.
\end{equation}
Therefore, combining (\ref{es33}) and (\ref{es37}) we can deduce that
$$\bar{X}^1_t=\bar{X}^2_t~~\mathbb{P}\text{-a.s.},~t\in[0,T].$$
Thanks to  the path continuity on $\mathbb{H}$, the claim follows.

To proceed, it is clear that the solution $(X,W)$ of Eq.~(\ref{eqSPDE}) is also a weak solution of Eq.~(\ref{eq7}) with initial law $\mu_0$. Applying the classical Yamada-Watanabe theorem from \cite{RSZ},  the pathwise uniqueness of Eq.~(\ref{eq7}) implies the  existence and uniqueness of  solutions  to Eq.~(\ref{eq7}), for which we denote by $\bar{X}$,  with initial condition $\bar{X}_0\sim\mu_0$. This leads to
\begin{equation}\label{es46}
\mathscr{L}_{\bar{X}_t}=\mu_t,~~ t\in[0,T].
\end{equation}
Inserting (\ref{es46}) into Eq.~(\ref{eq7}), it follows that $\bar{X}_t$ is a strong solution of Eq.~(\ref{eqSPDE}).

Consequently, by Lemma \ref{uniqueness} we can get that Eq.~(\ref{eqSPDE}) admits a unique (probabilistically) strong solution in the sense of Definition \ref{de2}. We complete the proof. \hspace{\fill}$\Box$

\section{Propagation of chaos}\label{Poc.4}
\setcounter{equation}{0}
 \setcounter{definition}{0}
 In this section, we study the propagation of chaos problem of the weakly interacting stochastic 2D Navier-Stokes systems. Throughout this section, we use the notations and definitions introduced in Example 2 (with $d=2$) in Subsection \ref{example}.
\subsection{Main results}\label{propagation}
In this subsection, we are interested in the asymptotic behaviour of the weakly interacting particle system $(X^{N,1},\ldots,X^{N,N})$ governed by
the following stochastic 2D Navier-Stokes systems

 \begin{eqnarray}\label{eqi}
\left\{
 \begin{aligned}
    & dX^{N,i}_t=\Big[ A X^{N,i}_t-B(X^{N,i}_t)+\frac{1}{N}\sum_{j=1}^N\tilde{K}(t,X^{N,i}_t,X^{N,j}_t)\Big]dt+\frac{1}{N}\sum_{j=1}^N\tilde{\sigma}(t,X^{N,i}_t,X^{N,j}_t)dW_t^i, \\
    &{\rm div}(X^{N,i}_t)=0 , \\
    &X^{N,i}_t=0, ~~\text{on}~\partial\mathscr{O},\\
    &X^{N,i}_0=\xi^i,
  \end{aligned}
\right.
\end{eqnarray}
 where $i=1,\ldots,N$, $W_t^1,\ldots,W_t^N$ are  independent $U$-valued cylindrical Wiener processes defined on a complete filtered probability space $(\Omega,\mathscr{F},\{\mathscr{F}_t\}_{t\in[0,T]},\mathbb{P})$ with right continuous filtration $\{\mathscr{F}_t\}$. % which supports $(W^i)$ and $(\xi^i)$.
 The coefficients $\tilde{K}$ and $\tilde{\sigma}$ satisfy the following assumptions.

\begin{conditionA}\label{H33}
The measurable maps
$$
\tilde{K}:[0,T]\times\mathbb H\times \mathbb H\rightarrow \mathbb H,~~\tilde{\sigma}:[0,T]\times \mathbb H\times \mathbb H\rightarrow L_2(U;\mathbb H),
$$
are locally Lipschitz on $\mathbb H$, i.e. there exists constant $C>0$ such that for any $t\in[0,T]$, $u_1,u_2,v_1,v_2\in\mathbb H$,
\begin{eqnarray*}
\!\!\!\!\!\!\!\!&&\| \tilde{K}(t,u_1,v_1)-\tilde{K}(t,u_2,v_2)\|_{L^2}+\|\tilde{\sigma}(t,u_1,v_1)-\tilde{\sigma}(t,u_2,v_2)\|_{L_2(U;\mathbb H)}
\nonumber \\
\leq \!\!\!\!\!\!\!\!&&C\big(1+\|u_1\|_{L^2}+\|u_2\|_{L^2}+\|v_1\|_{L^2}+\|v_2\|_{L^2}\big)\big(\|u_1-u_2\|_{L^2}+\|v_1-v_2\|_{L^2}\big),
\end{eqnarray*}
and
\begin{eqnarray}\label{linear1}
\| \tilde{K}(t,u_1,v_1)\|_{L^2}+\|\tilde{\sigma}(t,u_1,v_1)\|_{L_2(U;\mathbb H)}
\leq C\big(1+\|u_1\|_{L^2}+\|v_1\|_{L^2}\big).
\end{eqnarray}
\end{conditionA}

\begin{conditionA}\label{H44}
The measurable maps
$$
\tilde{K}:[0,T]\times\mathbb V\times \mathbb V\rightarrow \mathbb V,~~\tilde{\sigma}:[0,T]\times \mathbb V\times \mathbb V\rightarrow L_2(U;\mathbb V),
$$
satisfy that there exists constant $C>0$ such that for any $t\in[0,T]$, $u,v\in \mathbb V$,
\begin{eqnarray}\label{c12}
\|\tilde{K}(t,u,v)\|_{1}+\|\tilde{\sigma}(t,u,v)\|_{L_2(U;\mathbb V)}\leq C\big(1+\|u\|_{1}+\|v\|_{1}\big).
\end{eqnarray}
\end{conditionA}

\begin{remark}
As we mentioned before, the typical and important application of the interaction external force $\tilde{K}$ (also $\tilde{\sigma}$) satisfying $\mathbf{A1}$-${\mathbf{A2}}$ is the  following Stokes drag force in  fluids
$$\tilde{K}(t,u,v):=c_0(u-v),$$
where $c_0$ is a positive constant, then the particles of the system interact through their mean, i.e.~the Curie-Weiss type interaction (cf. Dawson \cite{Da83}).
\end{remark}

 It is well known that operator $B$  have the following estimates (see e.g.~\cite{FMRT})
\begin{eqnarray}
&&|\langle B(u,v),w\rangle|\leq C\|u\|_{L^2}^{\frac{1}{2}}\|u\|_{1}^{\frac{1}{2}}\|w\|_{L^2}^{\frac{1}{2}}
\|w\|_{1}^{\frac{1}{2}}\|v\|_{1},~\text{for}~ u,v,w\in \mathbb V,\label{bb3}
\\
&&|\langle B(u,u),v\rangle|\leq C\|u\|_{2}^{\frac{1}{2}}\|u\|_{1}\|u\|_{L^2}^{\frac{1}{2}}
\|v\|_{L^2},~\text{for}~ u\in W_{div}^{2,2}(\mathscr{O}), v\in\mathbb H.\label{bb4}
\end{eqnarray}
In particular,  for any $u\in \mathbb V$,
\begin{eqnarray}
\|B(u)\|_{-1}\leq C\|u\|_{L^2}\|u\|_1.\label{P2.2}
\end{eqnarray}

Let $\mathcal{Z}_T:=C([0,T];\mathbb{V})\cap L^2([0,T];W_{div}^{2,2}(\mathscr{O}))$. We denote by $\mathcal{Z}_T^{\otimes N}$  the $N$-product space of $\mathcal{Z}_T$.  The existence and uniqueness of (probabilistically) strong solutions to the $N$-particle system (\ref{eqi}) are described as follows.
\begin{theorem}\label{th3}
Suppose that $\mathbf{A1}$-$\mathbf{A2}$ hold.
For any $N\in\mathbb{N}$ and initial random variables $\xi^i\in L^p(\Omega,\mathscr{F}_0,\mathbb{P};{\mathbb{H}})\cap L^2(\Omega,\mathscr{F}_0,\mathbb{P};{\mathbb{V}})$, $i=1,\ldots,N$, with $p\geq 4$,
there exists a unique strong solution $X^{N}:=(X^{N,1},\ldots,X^{N,N})$ to (\ref{eqi}), where $X^N\in\mathcal{Z}_T^{\otimes N} $ $\mathbb{P}$-a.s..   Moreover, the following estimates hold
\begin{equation}\label{es61}
\mathbb{E}\Big[\sup_{t\in[0,T]}\|X^{N,i}_t\|_{L^2}^p\Big]+\mathbb{E}\int_0^T\|X^{N,i}_t\|_{L^2}^2\|X^{N,i}_t\|_{1}^2dt<\infty.
\end{equation}
\end{theorem}
\begin{proof}
Similar to the finite-dimensional case (cf.~\cite{Lacker}), it suffices to check that the coefficients of particle system (\ref{eqi}) satisfy the assumptions in \cite{RSZ1} (or \cite{RZZ}) on the product spaces. We omit the details to keep down the length of the paper.
\end{proof}

\vspace{2mm}
In order to state our main result of this section,
we introduce the following  McKean-Vlasov type stochastic Navier-Stokes equation
 \begin{eqnarray}\label{eqNSE111}
\left\{
 \begin{aligned}
&dX_t=\Big[A X_t-B(X_t)+K(t,X_t,\mathscr{L}_{X_t})\Big]dt+\sigma(t,X_t,\mathscr{L}_{X_t})dW_t,\\
&{\rm div}(X_t)=0 , \\
    &X_t=0, ~~\text{on}~\partial\mathscr{O},
  \end{aligned}
\right.
\end{eqnarray}
where
$$K(t,x,\mu):=\int\tilde K(t,x,y)\mu(dy),~\sigma(t,x,\mu):=\int \tilde{\sigma}(t,x,y)\mu(dy).$$

We recall the notion of martingale solution to MVSNSE (\ref{eqNSE111}).
Let $\Omega=C([0,T];\mathbb{H})$. Then we use $x$ to denote
a path in $\Omega$ and $\pi_t(x):=x_t$ to denote the coordinate process.
  Define the natural filtration $\mathscr{F}_t$ generated by
$\big\{\pi_s:~s\leq t\big\}.$

The definition of martingale solutions to MVSNSE (\ref{eqNSE111}) is presented  as follows.
\begin{definition} \label{de3}$($Martingale solution$)$  Let $\mu_0\in\mathscr{P}({\mathbb{H}})$. A probability measure $\Gamma\in\mathscr{P}(\Omega)$ is  called a martingale solution of MVSNSE (\ref{eqNSE111}) with initial distribution $\mu_0$ if

$(M1)$~~$\Gamma\circ \pi_0^{-1}=\mu_0$ and
$$\Gamma\Big(x\in \Omega:\int_0^T\|Ax_s-B(x_s)\|_{-1} ds+\int_0^T\|K(s,x_s,\mu_s)\|_{L^2}ds+\int_0^T\|\sigma(s,x_s,\mu_s)\|_{L_2(U;{\mathbb{H}})}^2ds<\infty\Big)=1,$$
where $\mu_s:=\Gamma\circ \pi_s^{-1}$.

\vspace{2mm}
$(M2)$~~for each $l\in\mathscr{V}$ the process
$$\mathscr{M}_l(t,x,\mu)=\langle x_t,l\rangle-\langle x_0,l\rangle-\int_0^t\langle Ax_s-B(x_s)+K(s,x_s,\mu_s),l\rangle ds,~t\in[0,T],$$
is a continuous square integrable $(\mathscr{F}_t)$-martingale w.r.t.~$\Gamma$,
whose quadratic variation process is given by
$$\langle \mathscr{M}_l\rangle(t,x,\mu):=\int_0^t\|\sigma(s,x_s,\mu_s)^*l\|_U^2ds,~t\in[0,T].$$
\end{definition}

Let us denote $\mathbb{V}^{\otimes N}$ the $N$-product space of $\mathbb{V}$. To investigate the convergence of $N$-particle system (\ref{eqi}) as $N\to\infty$, we fix a sequence of initial random vector $X_0^N:=(\xi^1,\ldots,\xi^N)$ which satisfies the following assumption.

\begin{conditionA}\label{H55}
  For any $N\in\mathbb{N}$,
$$\text{the law of}~X_0^N~\text{is symmetric in}~\mathscr{P}(\mathbb{V}^{\otimes N}),$$
 and
 $$\mathscr{S}^N_0=\frac{1}{N}\sum_{j=1}^N\delta_{X_0^{N,j}}\to \mu_0~\text{in probability},$$
where $\mu_0$ is the initial law of solution to MVSNSE (\ref{eqNSE111}).

\vspace{2mm}
For any $p\geq 4$
there exists a constant $C_p>0$ such that
\begin{equation}\label{c3}
\sup_{N\in\mathbb{N}}\mathbb{E}\|X_0^{N,1}\|_{L^2}^p\leq C_p.
\end{equation}
In addition, there exists a constant $C>0$ such that
\begin{equation}\label{c4}
\sup_{N\in\mathbb{N}}\mathbb{E}\|X_0^{N,1}\|_{1}^2\leq C.
\end{equation}

\end{conditionA}
\begin{remark}
The symmetry means
that for any permutation $X^{(N)}_0:=(\xi^{(1)},\ldots,\xi^{(N)} )$, the (joint) law $$\mathscr{L}_{X^{N}_0} =\mathscr{L}_{X^{(N)}_0}.$$
Combining the symmetry of the law of initial random vector $X_0^N$ and the uniqueness in law of solutions to the interacting system (\ref{eqi}), it is clear that the law of $X^N:=(X^{N,1},\ldots,X^{N,N})$ is also symmetric.
\end{remark}

Let
$$\mathscr{S}^N_t:=\frac{1}{N}\sum_{j=1}^N\delta_{X^{N,j}_t}$$
be the empirical distributions of the particle system \eref{eqi}. The following theorem shows the propagation of chaos for weakly interacting stochastic 2D Navier-Stokes systems.
\begin{theorem}\label{th4}
Suppose that $\mathbf{A1}$-$\mathbf{A3}$ hold. Then there exists an accumulation point  $\Gamma$ of the empirical measure $\mathscr{S}^N$ in $\mathscr{P}_2(C([0,T];\mathbb{H}))$, which is a solution of the martingale problem to MVSNSE (\ref{eqNSE111})
with initial law $\mu_0$.
Furthermore, if $\mathbf{H4}$ holds for $K$ and $\sigma$, then we have the following convergence
\begin{equation}\label{es111}
\lim_{N\to\infty}\mathbb{E}\Big[\mathbb{W}_{2,T,\mathbb{H}}(\mathscr{S}^N,\Gamma)^2\Big]=0.
\end{equation}
\end{theorem}

\begin{remark}
 (i) As a consequence of Theorem \ref{th4}, we can give a different proof compared to Theorem \ref{th1} for the existence of solutions of the martingale problem to  MVSNSE (\ref{eqNSE111}), which is of independent interest.

 (ii)To the best of our knowledge, this is the first result showing the propagation of chaos for the weakly interacting stochastic 2D NSE.
The contributions here also  include developing the martingale approach to establish the convergence of the empirical measure in the Wasserstein distance with state space $\mathbb{H}$.  This is novel since the  non-linear term appears in the system and we allow the interaction terms are locally Lipschitz. The next step after this work would be to get the convergence rate, this is very challenging for the infinite dimensional system.
\end{remark}

\subsection{A priori estimates}
Now we prove some uniform a priori estimates for $X^{N,i},i=1,\ldots,N$.
\begin{lemma}\label{lem11}
Assume that the assumptions in Theorem \ref{th4} hold. For any $p\geq 4$, there exists a constant $C_{p,T}>0$ which is independent of $N$  such that
\begin{equation}\label{es116}
\mathbb{E}\Big[\sup_{t\in[0,T]}\frac{1}{N}\sum_{i=1}^N\|X^{N,i}_t\|_{L^2}^p\Big]
+\mathbb{E}\Big[\int_0^{T}\frac{1}{N}\sum_{i=1}^N\Big(\|X^{N,i}_t\|_{L^2}^{p-2}\|X^{N,i}_t\|_{1}^{2}\Big)dt\Big]\leq C_{p,T},
\end{equation}
\begin{equation}\label{es117}
\mathbb{E}\Big[\sup_{t\in[0,T]}\|X^{N,1}_t\|_{L^2}^p\Big]+\mathbb{E}\Big[\int_0^{T}\|X^{N,1}_t\|_{L^2}^{p-2}\|X^{N,1}_t\|_{1}^{2}dt\Big]\leq C_{p,T},
\end{equation}
\begin{equation}\label{es1170}
\mathbb{E}\Big(\int_0^{T}\|X^{N,1}_t\|_{1}^{2}dt\Big)^{\frac{p}{2}}\leq C_{p,T}.
\end{equation}
\end{lemma}

\begin{proof}
Applying It\^{o}'s formula, we can get that for any $t\in[0,T]$,
\begin{eqnarray}\label{es11211}
\|X^{N,i}_t\|_{L^2}^2
\leq\!\!\!\!\!\!\!\!&&~\|\xi^i\|_{L^2}^2+\int_0^t2\langle AX^{N,i}_s,X^{N,i}_s\rangle ds
\nonumber \\
\!\!\!\!\!\!\!\!&&
+\frac{1}{N}\sum_{j=1}^N\int_0^t\Big(2(\tilde{K}(s,X^{N,i}_s,X^{N,j}_s),X^{N,i}_s)
+\|\tilde{\sigma}(s,X^{N,i}_s,X^{N,j}_s)\|_{L_2(U;{\mathbb{H}})}^2\Big)ds
\nonumber \\
\!\!\!\!\!\!\!\!&&+\frac{2}{N}\sum_{j=1}^N\int_0^t( X^{N,i}_s,\tilde{\sigma}(s,X^{N,i}_s,X^{N,j}_s)dW^{i}_s)
\nonumber \\
=:\!\!\!\!\!\!\!\!&&~\|\xi^i\|_{L^2}^2-2\int_0^t
\|X^{N,i}_s\|_{1}^{2}ds+\sum_{k=1}^3J^{N,i}_k(t)
\end{eqnarray}
and
\begin{eqnarray}\label{es112}
\|X^{N,i}_t\|_{L^2}^p
\leq\!\!\!\!\!\!\!\!&&~\|\xi^i\|_{L^2}^p+\frac{p}{2}\int_0^t\|X^{N,i}_s\|_{L^2}^{p-2}2\langle AX^{N,i}_s,X^{N,i}_s\rangle ds
\nonumber \\
\!\!\!\!\!\!\!\!&&+\frac{p}{2N}\sum_{j=1}^N\int_0^t\|X^{N,i}_s\|_{L^2}^{p-2}\Big(2(\tilde{K}(s,X^{N,i}_s,X^{N,j}_s),X^{N,i}_s)
\nonumber \\
\!\!\!\!\!\!\!\!&&~~~~~
+\|\tilde{\sigma}(s,X^{N,i}_s,X^{N,j}_s)\|_{L_2(U;{\mathbb{H}})}^2\Big)ds
\nonumber \\
\!\!\!\!\!\!\!\!&&+\frac{p(p-2)}{2N}\sum_{j=1}^N\int_0^t\|X^{N,i}_s\|_{L^2}^{p-4}\|\tilde{\sigma}(s,X^{N,i}_s,X^{N,j}_s)^*X^{N,i}_s\|_U^2ds
\nonumber \\
\!\!\!\!\!\!\!\!&&+\frac{p}{N}\sum_{j=1}^N\int_0^t\|X^{N,i}_s\|_{L^2}^{p-2}( X^{N,i}_s,\tilde{\sigma}(s,X^{N,i}_s,X^{N,j}_s)dW^{i}_s)
\nonumber \\
=:\!\!\!\!\!\!\!\!&&~\|\xi^i\|_{L^2}^p-p\int_0^t\|X^{N,i}_s\|_{L^2}^{p-2}
\|X^{N,i}_s\|_{1}^{2}ds+\sum_{k=1}^3 I^{N,i}_k(t).
\end{eqnarray}
For $I^{N,i}_1(t)$, by $\mathbf{A1}$ and Young's inequality, it leads to
\begin{eqnarray}\label{es113}
\mathbb{E}\Big[\sup_{t\in[0,T]}\Big(\frac{1}{N}\sum_{i=1}^NI^{N,i}_1(t)\Big)\Big]\leq\!\!\!\!\!\!\!\!&&~C_{p,T}+C_p\mathbb{E}\Big[\int_0^T\frac{1}{N}\sum_{i=1}^N\|X^{N,i}_t\|_{L^2}^pdt\Big]
\nonumber \\
\!\!\!\!\!\!\!\!&&
+C_p\mathbb{E}\Big[\int_0^T\frac{1}{N^2}\sum_{i,j=1}^N\Big(\|X^{N,i}_t\|_{L^2}^{p-2}\|X^{N,j}_t\|_{L^2}^{2}\Big)dt\Big]
\nonumber \\
\leq\!\!\!\!\!\!\!\!&&~C_{p,T}+C_{p}\mathbb{E}\Big[\int_0^T\frac{1}{N}\sum_{i=1}^N\|X^{N,i}_t\|_{L^2}^pdt\Big].
\end{eqnarray}
Using the assumption $\mathbf{A1}$, we can deduce
\begin{eqnarray*}
\!\!\!\!\!\!\!\!&&\mathbb{E}\Big[\sup_{t\in[0,T]}\Big(\frac{1}{N}\sum_{i=1}^NI^{N,i}_2(t)\Big)\Big]
\nonumber \\
\leq\!\!\!\!\!\!\!\!&&~C_p\frac{1}{N^2}\sum_{i,j=1}^N\int_0^T\|X^{N,i}_t\|_{L^2}^{p-2}\big(1+\|X^{N,i}_t\|_{L^2}^2+\|X^{N,j}_t\|_{L^2}^2\big)dt
\nonumber \\
\leq\!\!\!\!\!\!\!\!&&C_p\frac{1}{N^2}\sum_{i,j=1}^N\int_0^T\big(1+\|X^{N,i}_t\|_{L^2}^p+\|X^{N,j}_t\|_{L^2}^p\big)dt
\nonumber \\
\leq\!\!\!\!\!\!\!\!&&~C_{p,T}+C_{p}\mathbb{E}\Big[\int_0^T\frac{1}{N}\sum_{i=1}^N\|X^{N,i}_t\|_{L^2}^pdt\Big].
\end{eqnarray*}
To estimate  $I^{N,i}_3(t)$,
applying  BDG's inequality, it follows that
\begin{eqnarray}\label{es115}
\!\!\!\!\!\!\!\!&&\mathbb{E}\Big[\sup_{t\in[0,T]}\Big(\frac{1}{N}\sum_{i=1}^NI^{N,i}_3(t)\Big)\Big]
\nonumber \\
\leq\!\!\!\!\!\!\!\!&&~\frac{C_p}{N^2}\sum_{i,j=1}^N\mathbb{E}\Bigg\{\int_0^{T}\|X^{N,i}_t\|_{L^2}^{2p-2}\|\tilde{\sigma}(t,X^{N,i}_t,X^{N,j}_t)\|_{L_2(U;{\mathbb{H}})}^2dt\Bigg\}^{\frac{1}{2}}
\nonumber \\
\leq\!\!\!\!\!\!\!\!&&~\frac{C_p}{N^2}\sum_{i,j=1}^N\mathbb{E}\Bigg\{\sup_{t\in[0,T]}\|X^{N,i}_t\|_{L^2}^{p}\cdot\int_0^{T}\|X^{N,i}_t\|_{L^2}^{p-2}\big(1+\|X^{N,i}_t\|_{L^2}^2+\|X^{N,j}_t\|_{L^2}^2\big)dt\Bigg\}^{\frac{1}{2}}
\nonumber \\
\leq\!\!\!\!\!\!\!\!&&~\frac{1}{N^2}\sum_{i,j=1}^N\mathbb{E}\Bigg\{\frac{1}{2}\sup_{t\in[0,T]}\|X^{N,i}_t\|_{L^2}^{p}+C_p\int_0^{T}\|X^{N,i}_t\|_{L^2}^{p-2}\big(1+\|X^{N,i}_t\|_{L^2}^2+\|X^{N,j}_t\|_{L^2}^2\big)dt\Bigg\}
\nonumber \\
\leq\!\!\!\!\!\!\!\!&&~C_{p,T}+C_{p}\mathbb{E}\Big[\int_0^{T}\frac{1}{N}\sum_{i=1}^N\|X^{N,i}_t\|_{L^2}^pdt\Big]
+\frac{1}{2}\mathbb{E}\Big[\sup_{t\in[0,T]} \Big(\frac{1}{N}\sum_{i=1}^N\|X^{N,i}_t\|_{L^2}^p\Big)\Big].
\end{eqnarray}
Note that due to (\ref{es61}), for any $N\in\mathbb{N}$,
$$\mathbb{E}\Big[\sup_{t\in[0,T]} \Big(\frac{1}{N}\sum_{i=1}^N\|X^{N,i}_t\|_{L^2}^p\Big)\Big]\leq C_{p,N}<\infty.$$
Combining (\ref{es112})-(\ref{es115}),  it is easy to deduce that
\begin{eqnarray}\label{es1126}
\!\!\!\!\!\!\!\!&&\mathbb{E}\Big[\sup_{t\in[0,T]}\Big(\frac{1}{N}\sum_{i=1}^N\|X^{N,i}_t\|_{L^2}^p\Big)\Big]+2p\mathbb{E}\Big[\int_0^{T}\frac{1}{N}\sum_{i=1}^N\Big(\|X^{N,i}_t\|_{L^2}^{p-2}
\|X^{N,i}_t\|_{1}^2\Big)dt\Big]
\nonumber \\
\!\!\!\!\!\!\!\!&&\leq ~ C_{p,T}\Big(1+\mathbb{E}\Big(\frac{1}{N}\sum_{i=1}^N\|\xi^i\|_{L^2}^p\Big)\Big)+C_p\mathbb{E}\Big[\int_0^{T}\frac{1}{N}\sum_{i=1}^N\|X^{N,i}_t\|_{L^2}^pdt\Big].
\end{eqnarray}
By Gronwall's lemma it leads to
\begin{eqnarray*}
\!\!\!\!\!\!\!\!&&\mathbb{E}\Big[\sup_{t\in[0,T]}\Big(\frac{1}{N}\sum_{i=1}^N
\|X^{N,i}_t\|_{L^2}^p\Big)\Big]+2p\mathbb{E}\Big[\int_0^{T}\frac{1}{N}\sum_{i=1}^N
\Big(\|X^{N,i}_t\|_{L^2}^{p-2}\|X^{N,i}_t\|_{1}^2\Big)dt\Big]
\nonumber \\
\!\!\!\!\!\!\!\!&&
\leq~ C_{p,T}\Big(1+\frac{1}{N}\sum_{i=1}^N\mathbb{E}\|\xi^i\|_{L^2}^p\Big)\leq C_{p,T},
\end{eqnarray*}
where we used condition (\ref{c3}) and the law of $X_0^N$ is symmetric.
Thus we get (\ref{es116}).

For \eref{es117}, by \eref{es112} it is easy to see that
for any $t\in[0,T]$,
\begin{eqnarray}\label{vc1}
\|X^{N,1}_t\|_{L^2}^p
\leq\!\!\!\!\!\!\!\!&&~\|\xi^1\|_{L^2}^p+\frac{p}{2}\int_0^t\|X^{N,1}_s\|_{L^2}^{p-2}2\langle AX^{N,1}_s,X^{N,1}_s\rangle ds
\nonumber \\
\!\!\!\!\!\!\!\!&&+\frac{p}{2N}\sum_{j=1}^N\int_0^t\|X^{N,1}_s\|_{L^2}^{p-2}\Big
(2(\tilde{K}(s,X^{N,1}_s,X^{N,j}_s),X^{N,1}_s)
\nonumber \\
\!\!\!\!\!\!\!\!&&
+\|\tilde{\sigma}(s,X^{N,1}_s,X^{N,j}_s)\|_{L_2(U;{\mathbb{H}})}^2\Big)ds
\nonumber \\
\!\!\!\!\!\!\!\!&&+\frac{p(p-2)}{2N}\sum_{j=1}^N\int_0^t\|X^{N,1}_s\|_{L^2}^{p-4}
\|\tilde{\sigma}(s,X^{N,1}_s,X^{N,j}_s)^*X^{N,1}_s\|_U^2ds
\nonumber \\
\!\!\!\!\!\!\!\!&&+\frac{p}{N}\sum_{j=1}^N\int_0^t\|X^{N,1}_s\|_{L^2}^{p-2}( X^{N,1}_s,\tilde{\sigma}(s,X^{N,1}_s,X^{N,j}_s)dW^{1}_s)
\nonumber \\
\leq\!\!\!\!\!\!\!\!&&\|\xi^1\|_{L^2}^p-p\int_0^t\|X^{N,1}_s\|_{L^2}^{p-2}
\|X^{N,1}_s\|_{1}^{2}ds
\nonumber \\
\!\!\!\!\!\!\!\!&&+C_p\int_0^t(\|X^{N,1}_s\|_{L^2}^{p}+\frac{1}{N}\sum_{i=1}^N\|X^{N,i}_t\|_{L^2}^p+1)ds
\nonumber \\
\!\!\!\!\!\!\!\!&&+\frac{p}{N}\sum_{j=1}^N\int_0^t\|X^{N,1}_s\|_{L^2}^{p-2}( X^{N,1}_s,\tilde{\sigma}(s,X^{N,1}_s,X^{N,j}_s)dW^{1}_s).
\end{eqnarray}
Using  BDG's inequality and assumption $\mathbf{A1}$, we obtain
\begin{eqnarray}\label{vc2}
\!\!\!\!\!\!\!\!&&\mathbb{E}\Big[\sup_{t\in[0,T]}p\Big|\int_0^t\|X^{N,1}_s\|_{L^2}^{p-2}( X^{N,1}_s,\tilde{\sigma}(s,X^{N,1}_s,X^{N,j}_s)dW^{1}_s)\Big|\Big]
\nonumber \\
\!\!\!\!\!\!\!\!&&\leq~C_p\mathbb{E}\Big[\int_0^{T}\|X^{N,1}_s\|_{L^2}^{2p-2}\|\tilde{\sigma}(s,X^{N,1}_s,X^{N,j}_s)\|_{L_2(U;{\mathbb{H}})}^2ds\Big]^{\frac{1}{2}}
\nonumber \\
\!\!\!\!\!\!\!\!&&\leq~C_p\mathbb{E}\Big[\int_0^{T}\Big(1+\|X^{N,1}_t\|_{L^2}^p+
\|X^{N,j}_t\|_{L^2}^p\Big)dt\Big]
\nonumber \\
\!\!\!\!\!\!\!\!&&~~~
+\frac{1}{2}\mathbb{E}\Big[\sup_{t\in[0,T]} \|X^{N,1}_t\|_{L^2}^p\Big].
\end{eqnarray}
Note that for any $N$,
\begin{eqnarray*}
\mathbb{E}\Big[\sup_{t\in[0,T]}\|X^{N,1}_t\|_{L^2}^p\Big]
\leq~\mathbb{E}\Big[\sup_{t\in[0,T]}\sum_{i=1}^N\|X^{N,i}_t\|_{L^2}^p\Big]\leq C_{p,T,N}<\infty.
\end{eqnarray*}
By combining (\ref{es116}), \eref{vc1} and \eref{vc2}, we have
\begin{eqnarray}\label{vc3}
&&\mathbb{E}\Big[\sup_{t\in[0,T]}\|X^{N,1}_t\|_{L^2}^p\Big]
+2p\mathbb{E}\int_0^{T}\|X^{N,1}_t\|_{L^2}^{p-2}\|X^{N,1}_t\|_{1}^{2}dt
\nonumber \\
\leq \!\!\!\!\!\!\!\!&& \mathbb{E}\|\xi^1\|_{L^2}^p+C_p\mathbb{E}\int_0^{T}\|X^{N,1}_t\|_{L^2}^pdt+C_{p,T}.
\end{eqnarray}
Thus \eref{es117} follows from Gronwall's lemma.

Now we prove (\ref{es1170}). By (\ref{es11211}), following the same argument as proving \eref{es117},
\begin{eqnarray*}
\mathbb{E}\Big(\int_0^t
\|X^{N,1}_s\|_{1}^{2}ds\Big)^{\frac{p}{2}}
\leq C\mathbb{E}\|\xi^1\|_{L^2}^{p}+C\sum_{k=1}^3\mathbb{E}|J^{N,1}_k(t)|^{\frac{p}{2}}
\leq C_{p,T}.
\end{eqnarray*}
The proof is complete.
\end{proof}

\vspace{2mm}

For any $M>0$, define
\begin{equation}\label{tau1}
\tau_M^{N}:=\inf\Big\{t\geq0:\int_0^t\|X^{N,1}_s\|_{1}^{2}ds>M\Big\}\wedge\inf\Big\{t\geq0:\|X^{N,1}_t\|_{L^2}^2>M\Big\}\wedge T,
\end{equation}
where we take $\inf  \phi=\infty$.

\vspace{1mm}
The following lemma shows the uniform estimates of $\{X^{N,1}\}_{N\in\mathbb{N}}$ before the stopping time $\tau_M^{N}$ in the more regular space, which play an important role in the proof of the tightness of $\{X^{N,1}\}_{N\in\mathbb{N}}$ in $C([0,T];\mathbb{H})\cap L^2([0,T];\mathbb{V})$.

\begin{lemma}\label{lem vc}
Assume that the assumptions in Theorem \ref{th4} hold. For any $\xi^1\in L^2(\Omega,\mathscr{F}_0,\mathbb{P};\mathbb{V})$, there exists a constant $C_{T,M}>0$ independent of $N$ such that
\begin{equation}\label{esvv}
\mathbb{E}\Big[\sup_{t\in[0,\tau_M^{N}]}\|X^{N,1}_t\|_{1}^2\Big]+\mathbb{E}\Big[\int_0^{\tau_M^{N}}\|X^{N,1}_t\|_{2}^{2}dt\Big]
\leq C_{T,M}.
\end{equation}
\end{lemma}
\begin{proof}
Recall that $X^{N,1}\in \mathcal{Z}_T$ $\mathbb{P}$-a.s., then we can apply It\^{o}'s formula to get that for any $t\in[0,T]$,
\begin{eqnarray}\label{vv2}
\|X^{N,1}_t\|_{1}^2
\leq\!\!\!\!\!\!\!\!&&\|\xi^1\|_{1}^2-2\int_0^t\|AX^{N,1}_s\|_{L^2}^2ds-2
\int_0^t(B(X^{N,1}_s),AX^{N,1}_s)ds
\nonumber \\
\!\!\!\!\!\!\!\!&&+\frac{1}{N}\sum_{j=1}^N\int_0^t\Big(2((\tilde{K}(s,X^{N,1}_s,X^{N,j}_s),X^{N,1}_s))
+\|\tilde{\sigma}(s,X^{N,1}_s,X^{N,j}_s)\|_{L_2(U;{\mathbb{V}})}^2\Big)ds
\nonumber \\
\!\!\!\!\!\!\!\!&&+\frac{2}{N}\sum_{j=1}^N\int_0^t(( X^{N,1}_s,\tilde{\sigma}(s,X^{N,1}_s,X^{N,j}_s)dW^{1}_s)).
\end{eqnarray}
Using \eref{bb4} we have,
\begin{equation}\label{vv3}
|(B(X^{N,1}_s),AX^{N,1}_s)|\leq C\|X^{N,1}_s\|_{2}^{\frac{3}{2}}\|X^{N,1}_s\|_{1}
\|X^{N,1}_s\|_{L^2}^{\frac{1}{2}}\leq\frac{1}{2}   \|X^{N,1}_s\|_{2}^2+C\|X^{N,1}_s\|_{1}^4\|X^{N,1}_s\|_{L^2}^2.
\end{equation}
By assumption $\mathbf{A2}$ and Lemma \ref{lem11}, we get
\begin{eqnarray}\label{vv4}
&&\frac{1}{N}\sum_{j=1}^N\int_0^t\Big(2((\tilde{K}(s,X^{N,1}_s,X^{N,j}_s),X^{N,1}_s))
+\|\tilde{\sigma}(s,X^{N,1}_s,X^{N,j}_s)\|_{L_2(U;{\mathbb{V}})}^2\Big)ds
\nonumber \\
\leq \!\!\!\!\!\!\!\!&& C\int_0^t\Big(1+\|X^{N,1}_s\|^{2}_{1}+\frac{1}{N}\sum_{j=1}^N\|X^{N,j}_s\|_{1}^2\Big)ds
\nonumber \\
\leq\!\!\!\!\!\!\!\!&&C\int_0^t(1+\|X^{N,1}_s\|^{2}_{1})ds+C_{T}.
\end{eqnarray}
Combining \eref{vv2}-\eref{vv4} yields
\begin{eqnarray}\label{v5}
 \!\!\!\!\!\!\!\!&&\|X^{N,1}_t\|_{1}^2+\int_0^t\|X^{N,1}_s\|_{2}^2ds
\nonumber \\
\leq\!\!\!\!\!\!\!\!&&C_T+\|\xi^1\|_{1}^2
+C\int_0^t\|X^{N,1}_s\|_{1}^4\|X^{N,1}_s\|_{L^2}^2ds
\nonumber \\
\!\!\!\!\!\!\!\!&&+C\int_0^t(1+\|X^{N,1}_s\|^{2}_{1})ds+\frac{2}{N}\sum_{j=1}^N\int_0^t(( X^{N,1}_s,\tilde{\sigma}(s,X^{N,1}_s,X^{N,j}_s)dW^{1}_s)).
\nonumber \\
\end{eqnarray}
Applying Gronwall's inequality we have that for $t\in[0,\tau_M^{N}]$,
\begin{eqnarray}\label{v6}
&&\|X^{N,1}_t\|_{1}^2+\int_0^t\|X^{N,1}_s\|_{2}^2ds
\nonumber \\
\leq\!\!\!\!\!\!\!\!&&\Big(C_T+\|\xi^1\|_{1}^2+\frac{2}{N}\sum_{j=1}^N\sup_{t\in[0, \tau_M^{N}]}\Big|\int_0^t(( X^{N,1}_s,\tilde{\sigma}(s,X^{N,1}_s,X^{N,j}_s)dW^{1}_s))\Big|\Big)
\nonumber \\
\!\!\!\!\!\!\!\!&&~~\cdot\exp\Big(C_T+C\int_0^t\|X^{N,1}_s\|^{2}_{1}\|X^{N,1}_s\|^{2}_{L^2}   ds\Big).
\end{eqnarray}
Taking supremum on $t\in[0,\tau_M^n]$ and using the definition of $\tau_M^n$ we have
\begin{eqnarray}\label{v7}
&&\mathbb{E}\Big[\sup_{t\in[0, \tau_M^{N}]}\|X^{N,1}_t\|_{1}^2\Big]+\mathbb{E}\int_0^{\tau_M^{N}}\|X^{N,1}_s\|_{2}^2ds
\nonumber \\
\leq\!\!\!\!\!\!\!\!&&\Big(C_T+\mathbb{E}\|\xi^1\|_{1}^2+\frac{2}{N}\sum_{j=1}^N\mathbb{E}\Big[\sup_{t\in[0, \tau_M^{N}]}\Big|\int_0^t(( X^{N,1}_s,\tilde{\sigma}(s,X^{N,1}_s,X^{N,j}_s)dW^{1}_s))\Big|\Big]\Big)
\nonumber \\
\!\!\!\!\!\!\!\!&&~~\cdot
\exp\Big(C_T+CM^2\Big).
\end{eqnarray}
Using  BDG's inequality and assumption $\mathbf{A2}$, we obtain
\begin{eqnarray}\label{v8}
&&\frac{2}{N}\sum_{j=1}^N\mathbb{E}\left[\sup_{t\in[0, \tau_M^{N}]}\left|\int_0^t(( X^{N,1}_s,\tilde{\sigma}(s,X^{N,1}_s,X^{N,j}_s)dW^{1}_s))\right|\right]
\nonumber \\
\leq\!\!\!\!\!\!\!\!&&\frac{C}{N}\sum_{j=1}^N\mathbb{E}\Big[\int_0^{\tau_M^{N}}\|X^{N,1}_s\|_{1}^{2}
\|\tilde{\sigma}(s,X^{N,1}_s,X^{N,j}_s)\|_{L_2(U;{\mathbb{V}})}^2ds\Big]^{\frac{1}{2}}
\nonumber \\
\leq\!\!\!\!\!\!\!\!&&\frac{1}{2}\mathbb{E}\Big[\sup_{t\in[0,\tau_M^{N}]} \|X^{N,1}_t\|_{1}^2\Big]+C\mathbb{E}\Big[\int_0^{\tau_M^{N}}\Big(1+\|X^{N,1}_t\|_{1}^2+
\frac{1}{N}\sum_{i=1}^N\|X^{N,i}_t\|_{1}^2\Big)dt\Big]
\nonumber \\
\leq\!\!\!\!\!\!\!\!&&\frac{1}{2}\mathbb{E}\Big[\sup_{t\in[0,\tau_M^n]} \|X^{N,1}_t\|_{1}^2\Big]+CT+CM.
\end{eqnarray}
By \eref{v7} and \eref{v8}, it is easy to see that
\begin{eqnarray}\label{v77}
\mathbb{E}\Big[\sup_{t\in[0, \tau_M^{N}]}\|X^{N,1}_t\|_{1}^2\Big]+\mathbb{E}\int_0^{\tau_M^{N}}\|X^{N,1}_s\|_{2}^2ds\leq
C_{T,M}\Big (1+\mathbb{E}\|\xi^1\|_{1}^2\Big).
\end{eqnarray}
 The proof is complete.
\end{proof}

\subsection{Tightness of $\{\mathscr{S}^N\}_{N\geq 1}$}
In this subsection, we first prove the tightness of $\{X^{N,1}\}_{N\in\mathbb{N}}$ in
$$\mathbb{S}:=C([0,T];\mathbb{H})\cap L^2([0,T];\mathbb{V}).$$
Based on this, then we prove the tightness of $\{\mathscr{S}^N\}_{N\in\mathbb{N}}$ in $\mathscr{P}_2(\mathbb{S})$.

\begin{lemma}\label{pro2}
Under the assumptions in Theorem \ref{th4}, the sequence  $\{X^{N,1}\}_{N\in\mathbb{N}}$ is  tight in $\mathbb{S}$.
\end{lemma}
\begin{proof}
It suffices to prove that $\{X^{N,1}\}_{N\in\mathbb{N}}$ is tight in $C([0,T];\mathbb{H})$ and
$L^2([0,T];\mathbb{V})$ separately.
\\
\textbf{Step 1:} In this step, we shall prove  the tightness in $C([0,T];\mathbb{H})$.
Since the embedding $\mathbb{V}\subset \mathbb{H}$ is compact, by Aldou's tightness criterion (cf.~\cite[Theorem 1]{A1} and \cite[Lemma 3.8]{BM1}), it suffices to prove the following two claims.
 \begin{enumerate}[(i)]
  \item  For any $0<\eta<1$, there exists $R_\eta>0$ such that
\begin{eqnarray}\label{vv8}
\sup_{N\in\mathbb{N}}\mathbb{P}\left(\sup_{0\leq t\leq T}\|X^{N,1}_t\|_{1}>R_\eta\right)<\eta.
\end{eqnarray}

  \item For any $\eta>0$ and stopping time $0\leq\tau^N\leq T$,
\begin{eqnarray}\label{vv9}
\lim_{\delta\rightarrow0}\sup_{N\in\mathbb{N}}\mathbb{P}
\left(\|X^{N,1}_{(\tau^N+\delta)}-X^{N,1}_{\tau^N}\|_{L^2}>\eta\right)=0,
\end{eqnarray}
where $\tau^N+\delta:=T\wedge(\tau^N+\delta)$.
\end{enumerate}

By Lemma \ref{lem11}, we know that
 \begin{eqnarray}\label{vv10}
&&\sup_{N\in\mathbb{N}}\mathbb{P}\left(\tau_M^{N}<T\right)
\nonumber \\
\leq \!\!\!\!\!\!\!\!&&
\sup_{N\in\mathbb{N}}\mathbb{P}\left(\int_0^T\|X^{N,1}_s\|_{1}^{2}ds>M\right)
+\sup_{N\in\mathbb{N}}\mathbb{P}\left(\sup_{0\leq s\leq T}\|X^{N,1}_s\|_{L^2}^2>M\right)
\nonumber \\
\leq \!\!\!\!\!\!\!\!&&
\frac{1}{M}\sup_{N\in\mathbb{N}}\mathbb{E}\left(\int_0^T\|X^{N,1}_s\|_{1}^{2}ds\right)
+\frac{1}{M}\sup_{N\in\mathbb{N}}\mathbb{E}\left(\sup_{0\leq s\leq T}\|X^{N,1}_s\|_{L^2}^2\right)
\nonumber \\
\leq \!\!\!\!\!\!\!\!&&\frac{C_T}{M},
\end{eqnarray}
where the constant $C_{T}>0$ is independent of $M$.

 For any $R>0$, by \eref{esvv} and \eref{vv10}, we have
  \begin{eqnarray}\label{vv11}
&&\sup_{N\in\mathbb{N}}\mathbb{P}\left(\sup_{0\leq t\leq T}\|X^{N,1}_t\|_{1}>R\right)
\nonumber \\
= \!\!\!\!\!\!\!\!&&
\sup_{N\in\mathbb{N}}\mathbb{P}\left(\sup_{0\leq t\leq T}\|X^{N,1}_t\|_{1}>R,\tau_M^{N}=T\right)
+\sup_{N\in\mathbb{N}}\mathbb{P}\left(\sup_{0\leq t\leq T}\|X^{N,1}_t\|_{1}>R,\tau_M^{N}<T\right)
\nonumber \\
\leq \!\!\!\!\!\!\!\!&&
\sup_{N\in\mathbb{N}}\mathbb{P}\left(\sup_{0\leq t\leq \tau_M^{N}}\|X^{N,1}_t\|_{1}>R\right)+\sup_{N\in\mathbb{N}}\mathbb{P}\left(\tau_M^{N}<T\right)
\nonumber \\
\leq \!\!\!\!\!\!\!\!&&\frac{1}{R^2}\sup_{N\in\mathbb{N}}\mathbb{E}\left(\sup_{0\leq t\leq \tau_M^{N}}\|X^{N,1}_t\|_{1}^2\right)+\frac{C_T}{M}
\nonumber \\
\leq \!\!\!\!\!\!\!\!&&\frac{C_M}{R^2}+\frac{C_T}{M}.
\end{eqnarray}
Then, for any $\eta>0$, we can first take $M$ sufficiently large and then choose
$R$ large enough so that the right hand side of \eref{vv11} will be smaller than $\eta$. Thus the claim (i) holds.

Now we turn to prove the claim (ii). For any $\eta>0$ and stopping time $0\leq\tau^n\leq T$,
\begin{eqnarray}\label{vv12}
&&\sup_{N\in\mathbb{N}}\mathbb{P}
\left(\|X^{N,1}_{(\tau^N+\delta)}-X^{n}_{\tau^N}\|_{L^2}>\eta\right)
\nonumber \\
\leq \!\!\!\!\!\!\!\!&&
\sup_{N\in\mathbb{N}}\mathbb{P}
\left(\Big\|\int_{\tau^N}^{(\tau^N+\delta)}AX^{N,1}_sds\Big\|_{L^2}
>\frac{\eta}{4}\right)
\nonumber \\
\!\!\!\!\!\!\!\!&&+
\sup_{N\in\mathbb{N}}\mathbb{P}
\left(\Big\|\int_{\tau^N}^{(\tau^N+\delta)}B(X^{N,1}_s)ds\Big\|_{L^2}>\frac{\eta}{4}\right)
\nonumber \\
\!\!\!\!\!\!\!\!&&+
\sup_{n\in\mathbb{N}}\mathbb{P}
\left(\Big\|\frac{1}{N}\sum_{j=1}^N\int_{\tau^N}^{(\tau^N+\delta)}\tilde{K}(s,X^{N,1}_s,X^{N,j}_s)ds\Big\|_{L^2}
>\frac{\eta}{4}\right)
\nonumber \\
\!\!\!\!\!\!\!\!&&+
\sup_{N\in\mathbb{N}}\mathbb{P}
\left(\Big\|\frac{1}{N}\sum_{j=1}^N\int_{\tau^N}^{(\tau^N+\delta)}\tilde{\sigma}(s,X^{N}_s,X^{N,j}_s)dW^{1}_s\Big\|_{L^2}>\frac{\eta}{4}\right)
\nonumber \\
=: \!\!\!\!\!\!\!\!&& \text{I}+\text{II}+\text{III}+\text{IV}.
\end{eqnarray}
By H\"{o}lder's inequality, \eref{esvv} and \eref{vv10} we have
 \begin{eqnarray}\label{vv13}
\text{I}\leq\!\!\!\!\!\!\!\!&&\sup_{N\in\mathbb{N}}\mathbb{P}
\left(\delta\int_{\tau^N}^{(\tau^N+\delta)}\|AX^{N,1}_s\|_{L^2}^2ds>\frac{\eta^2}{16}\right)
\nonumber \\
=\!\!\!\!\!\!\!\!&&
\sup_{N\in\mathbb{N}}\mathbb{P}\left(\delta\int_{\tau^N}^{(\tau^N+\delta)}\|AX^{N,1}_s\|_{L^2}^2ds>\frac{\eta^2}{16},\tau_M^{N}=T\right)
\nonumber \\
\!\!\!\!\!\!\!\!&&+\sup_{N\in\mathbb{N}}\mathbb{P}\left(\delta\int_{\tau^N}^{(\tau^N+\delta)}\|AX^{N,1}_s\|_{L^2}^2ds>\frac{\eta^2}{16},\tau_M^{N}<T\right)
\nonumber \\
\leq \!\!\!\!\!\!\!\!&&
\sup_{N\in\mathbb{N}}\mathbb{P}\left(\delta\int_{0}^{\tau_M^{N}}\|AX^{N,1}_s\|_{L^2}^2ds>\frac{\eta^2}{16}\right)
+\sup_{N\in\mathbb{N}}\mathbb{P}\left(\tau_M^{N}<T\right)
\nonumber \\
\leq \!\!\!\!\!\!\!\!&&\frac{16}{\eta^2}\delta\sup_{N\in\mathbb{N}}\mathbb{E}
\int_{0}^{\tau_M^{N}}\|X^{N,1}_s\|_{2}^2ds+\frac{C_T}{M}
\nonumber \\
\leq \!\!\!\!\!\!\!\!&&\frac{C_M}{\eta^2}\delta+\frac{C_T}{M}.
\end{eqnarray}

By \eref{bb4}, there exists a constant $c>0$ such that $\|B(X^{N,1}_s)\|_{L^2}\leq c\|X^{N,1}_s\|_{2}^{\frac{1}{2}}\|X^{N,1}_s\|_{1}\|X^{N,1}_s\|_{L^2}^{\frac{1}{2}}$, then by
H\"{o}lder's inequality, \eref{esvv} and \eref{vv10} we have
 \begin{eqnarray}\label{v14}
\text{II}\leq\!\!\!\!\!\!\!\!&&\sup_{N\in\mathbb{N}}\mathbb{P}
\left(\int_{\tau^N}^{(\tau^N+\delta)}\|B(X^{N,1}_s)\|_{L^2}ds>\frac{\eta}{4}\right)
\nonumber \\
\leq\!\!\!\!\!\!\!\!&&\sup_{N\in\mathbb{N}}\mathbb{P}
\left(\int_{\tau^N}^{(\tau^N+\delta)}\|X^{N,1}_s\|_{2}^{\frac{1}{2}}\|X^{N,1}_s\|_{1}\|X^{N,1}_s\|_{L^2}^{\frac{1}{2}}ds>\frac{\eta}{4c}\right)
\nonumber \\
\leq\!\!\!\!\!\!\!\!&&
\sup_{N\in\mathbb{N}}\mathbb{P}\left(\int_{\tau^N}^{(\tau^N+\delta)}\|X^{N,1}_s\|_{2}^{\frac{1}{2}}\|X^{N,1}_s\|_{1}\|X^{N,1}_s\|_{L^2}^{\frac{1}{2}}ds>\frac{\eta}{4c},\tau_M^{N}=T\right)
+\sup_{N\in\mathbb{N}}\mathbb{P}\left(\tau_M^{N}<T\right)
\nonumber \\
\leq\!\!\!\!\!\!\!\!&&
\sup_{N\in\mathbb{N}}\mathbb{P}\left(\int_{\tau^N}^{(\tau^N+\delta)\wedge\tau_M^{N}}\|X^{N,1}_s\|_{2}^{\frac{1}{2}}\|X^{N,1}_s\|_{1}\|X^{N,1}_s\|_{L^2}^{\frac{1}{2}}ds>\frac{\eta}{4c}\right)
+\frac{C_T}{M}
\nonumber \\
\leq \!\!\!\!\!\!\!\!&&
\frac{4c}{\eta}\sup_{N\in\mathbb{N}}\Big[\Big(\mathbb{E}\int_{\tau^N}^{(\tau^N+\delta)\wedge\tau_M^{N}}\|X^{N,1}_s\|_{2}^{2}ds\Big)^{\frac{1}{4}}
\Big(\mathbb{E}\int_{\tau^N}^{(\tau^N+\delta)\wedge\tau_M^{N}}\|X^{N,1}_s\|_{L^2}^{2}ds\Big)^{\frac{1}{4}}
\nonumber \\
\!\!\!\!\!\!\!\!&&~~~~\cdot\Big(\mathbb{E}
\int_{\tau^N}^{(\tau^N+\delta)\wedge\tau_M^{N}}\|X^{N,1}_s\|_{1}^{2}ds
\Big)^{\frac{1}{2}}\Big]
+\frac{C_T}{M}
\nonumber \\
\leq \!\!\!\!\!\!\!\!&&
\frac{C_M}{\eta}\delta^{\frac{3}{4}}
\sup_{N\in\mathbb{N}}\Big(\mathbb{E}\sup_{0\leq s\leq \tau_M^{N}}\|X^{N,1}_s\|_{L^2}^{2}\Big)^{\frac{1}{4}}
\times\sup_{N\in\mathbb{N}}\Big(\mathbb{E}\int_{0}^{\tau_M^{N}}\|X^{N,1}_s\|_{2}^{2}ds
\Big)^{\frac{1}{4}}
\nonumber \\
\!\!\!\!\!\!\!\!&&~~~~\cdot\sup_{N\in\mathbb{N}}\Big(\mathbb{E}\sup_{0\leq s\leq \tau_M^{N}}\|X^{N,1}_s\|_{1}^{2}\Big)^{\frac{1}{2}}
+\frac{C_T}{M}
\nonumber \\
\leq \!\!\!\!\!\!\!\!&&\frac{C_M}{\eta^2}\delta^{\frac{3}{4}}+\frac{C_T}{M}.
\end{eqnarray}
For $\text{III}$, by the assumption $\mathbf{A1}$,
 \begin{eqnarray}\label{v15}
\text{III}\leq\!\!\!\!\!\!\!\!&&\frac{16}{\eta^2}\sup_{N\in\mathbb{N}}\frac{1}{N}\sum_{j=1}^N\mathbb{E}
\left(\int_{\tau^N}^{(\tau^N+\delta)}\|\tilde{K}(s,X^{N,1}_s,X^{N,j}_s)\|_{L^2}^2ds\right)
\nonumber \\
\leq \!\!\!\!\!\!\!\!&&
\frac{C}{\eta^2}\sup_{N\in\mathbb{N}}\mathbb{E}
\left(\int_{\tau^N}^{(\tau^N+\delta)}\Big(1+\|X^{N,1}_s\|_{L^2}^2+\frac{1}{N}\sum_{j=1}^N\|X^{N,j}_s\|_{L^2}^2\Big)ds\right)
\nonumber \\
\leq \!\!\!\!\!\!\!\!&&\frac{C}{\eta^2}\delta\Big(1+\sup_{N\in\mathbb{N}}\mathbb{E}
\big[\sup_{0\leq s\leq T}\|X^{N,1}_s\|_{L^2}^2\big]+C_T\Big)
\nonumber \\
\leq \!\!\!\!\!\!\!\!&&\frac{C_T}{\eta^2}\delta.
\end{eqnarray}
For the stochastic integral term $\text{IV}$, by the assumption $\mathbf{A1}$,
 \begin{eqnarray}\label{v16}
\text{IV}\leq\!\!\!\!\!\!\!\!&&\frac{16}{\eta^2}\sup_{N\in\mathbb{N}}\frac{1}{N}\sum_{j=1}^N\mathbb{E}
\left(\Big\|\int_{\tau^N}^{(\tau^N+\delta)}\tilde{\sigma}(s,X^{N,1}_s,X^{N,j}_s)dW^{1}_s\Big\|_{L^2}^2\right)
\nonumber \\
\leq \!\!\!\!\!\!\!\!&&
\frac{C}{\eta^2}\sup_{N\in\mathbb{N}}\frac{1}{N}\sum_{j=1}^N\mathbb{E}
\left(\int_{\tau^N}^{(\tau^N+\delta)}\|\tilde{\sigma}(s,X^{N,1}_s,X^{N,j}_s)\|_{L_2(U;\mathbb H)}^2ds\right)
\nonumber \\
\leq \!\!\!\!\!\!\!\!&&
\frac{C}{\eta^2}\sup_{N\in\mathbb{N}}\mathbb{E}
\left(\int_{\tau^N}^{(\tau^N+\delta)}\Big(1+\|X^{N,1}_s\|_{L^2}^2+\frac{1}{N}\sum_{j=1}^N\|X^{N,j}_s\|_{L^2}^2\Big)ds\right)
\nonumber \\
\leq \!\!\!\!\!\!\!\!&&\frac{C}{\eta^2}\delta\Big(1+\sup_{N\in\mathbb{N}}\mathbb{E}\big[
\sup_{0\leq s\leq T}\|X^{N,1}_s\|_{L^2}^2\big]+C_T\Big)
\nonumber \\
\leq \!\!\!\!\!\!\!\!&&\frac{C_T}{\eta^2}\delta.
\end{eqnarray}
Combining \eref{vv12}-\eref{v16} implies,
\begin{eqnarray*}
\sup_{N\in\mathbb{N}}\mathbb{P}
\left(\|X^{N,1}_{(\tau^N+\delta)}-X^{N,1}_{\tau^N}\|_{L^2}>\eta\right)\leq\frac{C_M}{\eta^2}\delta+\frac{C_M}{\eta^2}\delta^{\frac{3}{4}}+\frac{C_T}{\eta^2}\delta+\frac{C_T}{M}.
\end{eqnarray*}
Let $\delta\rightarrow0$ first, then let $M\rightarrow\infty$, we can see that the claim (ii) holds.\\
\textbf{Step 2:} Now we show that the sequence $\{X^{N,1}\}_{N\in\mathbb{N}}$ is tight in
$L^2([0,T];\mathbb{V})$. Since the embedding $W_{{div}}^{2,2}(\mathscr{O})\subset \mathbb{V}$ is compact, by the tightness criterion (cf. \cite[Theorem 5]{S3} and \cite[Lemma 5.2]{RSZ1}), we only need to show the following two claims hold.

(iii)
\begin{eqnarray}\label{vv88}
\lim_{\kappa\rightarrow\infty}\sup_{N\in\mathbb{N}}\mathbb{P}\left(\int_0^T\|X^{N,1}_t\|_{2}^2dt>\kappa\right)=0.
\end{eqnarray}

(iv) For any $\varepsilon>0$,
\begin{eqnarray}\label{vv99}
\lim_{\delta\rightarrow0^+}\sup_{N\in\mathbb{N}}\mathbb{P}
\left(\int_0^{T-\delta}\|X^{N,1}_{(t+\delta)}-X^{N,1}_{t}\|_{L^2}^2dt>\varepsilon\right)=0.
\end{eqnarray}

For the claim (iii),  by \eref{esvv} and \eref{vv10} we have
  \begin{eqnarray}\label{vv111}
&&\sup_{N\in\mathbb{N}}\mathbb{P}\left(\int_0^T\|X^{N,1}_t\|_{2}^2dt>\kappa\right)
\nonumber \\
= \!\!\!\!\!\!\!\!&&
\sup_{N\in\mathbb{N}}\mathbb{P}\left(\int_0^T\|X^{N,1}_t\|_{2}^2dt>\kappa,\tau_M^{N}=T\right)
+\sup_{N\in\mathbb{N}}\mathbb{P}\left(\int_0^T\|X^{N,1}_t\|_{2}^2dt>\kappa,\tau_M^{N}<T\right)
\nonumber \\
\leq \!\!\!\!\!\!\!\!&&
\sup_{N\in\mathbb{N}}\mathbb{P}\left(\int_0^{\tau_M^{N}}\|X^{N,1}_t\|_{2}^2dt>\kappa\right)+\sup_{N\in\mathbb{N}}\mathbb{P}\left(\tau_M^{N}<T\right)
\nonumber \\
\leq \!\!\!\!\!\!\!\!&&\frac{1}{\kappa}\sup_{N\in\mathbb{N}}\mathbb{E}\left(\int_0^{\tau_M^{N}}\|X^{N,1}_t\|_{2}^2dt\right)+\frac{C_T}{M}
\nonumber \\
\leq \!\!\!\!\!\!\!\!&&\frac{C_M}{\kappa}+\frac{C_T}{M}.
\end{eqnarray}
Let $\kappa\rightarrow\infty$ first, then let $M\rightarrow\infty$ we can see that the claim (iii) holds.

\vspace{2mm}
Finally, we prove the claim (iv).
It follows from It\^{o}'s formula that
\begin{eqnarray}
\!\!\!\!\!\!\!\!&&\|X^{N,1}_{(t+\delta)}-X^{N,1}_{t}\|_{L^2}^{2}
\nonumber \\
=\!\!\!\!\!\!\!\!&&2\int_{t} ^{t+\delta}\langle AX^{N,1}_{s}-B(X^{N,1}_{s}), X^{N,1}_{  s   }-X^{N,1}_{t}\rangle ds
\nonumber \\
 \!\!\!\!\!\!\!\!&&
+ \frac{1}{N}\sum_{j=1}^N\int_{t} ^{t+\delta}\big(2(\tilde{K}(s,X^{N,1}_s,X^{N,j}_s), X^{N,1}_{ s  }-X^{N,1}_{t})+\|\tilde{\sigma}(s,X^{N,1}_s,X^{N,j}_s)\|_{L_2(U;\mathbb H)}^2\big) ds
\nonumber \\
 \!\!\!\!\!\!\!\!&& +\frac{2}{N}\sum_{j=1}^N\int_{t} ^{t+\delta}( X^{N,1}_{ s   }-X^{N,1}_{t},  \tilde{\sigma}(s,X^{N,1}_s,X^{N,j}_s)dW^{1}_s) \nonumber\\
=:\!\!\!\!\!\!\!\!&&\text{I}(t)+\text{II}(t)+\text{III}(t).  \label{F3.9}
\end{eqnarray}
For the first term $\text{I}(t)$, by H\"{o}lder's inequality and \eref{P2.2}, there exists a constant $C>0$ such that
\begin{eqnarray}  \label{REGX1}
&&\mathbb{E}\left(\int_0^{T-\delta}\text{I}(t)dt\right)\nonumber\\
\leq\!\!\!\!\!\!\!\!&& C\mathbb{E}\left(\int_0^{T-\delta}\int_{t} ^{t+\delta}\| AX^{N,1}_{s}-B(X^{N,1}_{s})\|_{-1}
\|X^{N,1}_{ s   }-X^{N,1}_{t}\|_{1} ds dt\right)\nonumber\\
\leq\!\!\!\!\!\!\!\!&&C\left[\mathbb{E}\int_0^{T-\delta}\int_{t} ^{t+\delta}\|X^{N,1}_{s}\|_1^2+\|B(X^{N,1}_{s})\|^2_{-1}dsdt\right]^{1/2}
\cdot\left[\mathbb{E}\int_0^{T-\delta}\int_{t} ^{t+\delta}\|X^{N,1}_{ s   }-X^{N,1}_{t}\|^2_{1} dsdt\right]^{1/2}\nonumber\\
\leq\!\!\!\!\!\!\!\!&&C\left[\delta\mathbb{E}\int^{T}_0\|X^{N,1}_{s}\|^2_{1}(1+\|X^{N,1}_{s}\|_{L^2}^2)ds\right]^{1/2}\cdot\left[\delta\mathbb{E}\int^{T}_0\|X^{N,1}_{s}\|^2_{1}ds\right]^{1/2}\nonumber\\
\leq\!\!\!\!\!\!\!\!&&C_{T}\delta,
\end{eqnarray}
where we use Fubini's theorem and \eref{es117} in the third and fourth inequalities respectively.

For $\text{II}(t)$, by \eref{es116} and \eref{es117} we have
\begin{eqnarray}  \label{REGX2}
&&\mathbb{E}\left(\int_0^{T-\delta}\text{II}(t)dt\right)\nonumber\\
\leq\!\!\!\!\!\!\!\!&& C\mathbb{E}\left(\int_0^{T-\delta}\int_{t} ^{t+\delta}\Big(1+\|X^{N,1}_s\|^{2}_{L^2}+\frac{1}{N}\sum_{j=1}^N\|X^{N,j}_s\|_{L^2}^2+
\|X^{N,1}_{ s   }-X^{N,1}_{t}\|_{L^2}^2 \Big)ds dt\right)\nonumber\\
\leq\!\!\!\!\!\!\!\!&&C\delta\mathbb{E}\left[\sup_{s\in[0,T]}\left(1+\|X^{N,1}_{s}\|^2_{L^2}+\frac{1}{N}\sum_{j=1}^N\|X^{N,j}_s\|_{L^2}^2\right)\right]\nonumber\\
\leq\!\!\!\!\!\!\!\!&&C_{T}\delta.
\end{eqnarray}

For $\text{III}(t)$, by martingale property it follows that
\begin{eqnarray}  \label{REGX3}
\mathbb{E}\left(\int_0^{T-\delta}\text{III}(t)dt\right)=\!\!\!\!\!\!\!\!&&\frac{2}{N}\sum_{j=1}^N\int_0^{T-\delta}\mathbb{E}
\left[\int_{t} ^{t+\delta}( X^{N,1}_{ s   }-X^{N,1}_{t},  \tilde{\sigma}(s,X^{N,1}_s,X^{N,j}_s)dW^{1}_s)\right]dt \nonumber\\
=\!\!\!\!\!\!\!\!&&0.
\end{eqnarray}
%applying Burkholder-Davies-Gundy's inequality, we get
%\begin{eqnarray}  \label{REGX2}
%&&\mathbb{E}\left(\int_0^{T-\delta}|I_{3}(t)|dt\right)\nonumber \\
%\!\!\!\!\!\!\!\!&&\leq~\frac{C}{N}\sum_{j=1}^N\mathbb{E}\int_0^{T-\delta}\Big[\int_{t} ^{t+\delta}
%\|X^{N,1}_{s+\delta}-X^{N,1}_{t}\|_{L^2}^2
%\|\tilde{\sigma}(s,X^{N,1}_s,X^{N,j}_s)\|_{L_2(U;{\mathbb{H}})}^2ds\Big]^{\frac{1}{2}}dt\nonumber\\
%\!\!\!\!\!\!\!\!&&\leq~\frac{C}{N}\sum_{j=1}^N\mathbb{E}\int_0^{T-\delta}\Big[\int_{t} ^{t+\delta}
%\|X^{N,1}_{s+\delta}-X^{N,1}_{t}\|_{L^2}^2
%\|\tilde{\sigma}(s,X^{N,1}_s,X^{N,j}_s)\|_{L_2(U;{\mathbb{H}})}^2ds\Big]^{\frac{1}{2}}dt\nonumber\\
%\leq\!\!\!\!\!\!\!\!&& C\mathbb{E}\left(\int_0^{T-\delta}\int_{t} %^{t+\delta}\Big(1+\|X^{N,1}_s\|^{2}_{L^2}+\frac{1}{N}\sum_{j=1}^N\|X^{N,j}_s\|_{L^2}^2+
%\|X^{N,1}_{s+\delta}-X^{N,1}_{t}\|_{L^2}^2 \Big)ds dt\right)\nonumber\\
%\leq\!\!\!\!\!\!\!\!&&C\delta\mathbb{E}\left[\sup_{s\in[0,T]}\left(1+\|X^{N,1}_{s}\|^2_{L^2}+\frac{1}{N}\sum_{j=1}^N\|X^{N,j}_s\|_{L^2}^2\right)\right]\nonumber\\
%\leq\!\!\!\!\!\!\!\!&&C_{T}\delta.
%\end{eqnarray}
Combining the estimates \eref{F3.9}-\eref{REGX3}, it follows that
\begin{eqnarray}
\mathbb{E}\left(\int_0^{T-\delta}\|X^{N,1}_{(t+\delta)}-X^{N,1}_{t}\|_{L^2}^{2}dt\right)\leq C_{T}\delta. \label{F3.13}
\end{eqnarray}
Then for any $\varepsilon>0$,
\begin{eqnarray*}
\sup_{N\in\mathbb{N}}\mathbb{P}
\left(\int_0^{T-\delta}\|X^{N,1}_{(t+\delta)}-X^{N,1}_{t}\|_{L^2}^2dt>\varepsilon\right)
\leq\!\!\!\!\!\!\!\!&&\frac{1}{\varepsilon}\sup_{N\in\mathbb{N}}\mathbb{E}\left(\int_0^{T-\delta}\|X^{N,1}_{(t+\delta)}-X^{N,1}_{t}\|_{L^2}^{2}dt\right)
\nonumber\\
\leq\!\!\!\!\!\!\!\!&&\frac{1}{\varepsilon}C_{T}\delta,
\end{eqnarray*}
 which implies the claim (iv) holds. The proof is complete.
\end{proof}

\vspace{2mm}
We now turn to prove the tightness of sequence  $\{\mathscr{S}^{N}\}_{N\in\mathbb{N}}$ in $\mathscr{P}_2(\mathbb{S})$.
\begin{lemma}\label{pro13}
Under the assumptions in Theorem \ref{th4}, the sequence  $\{\mathscr{S}^{N}\}_{N\in\mathbb{N}}$ is  tight in $\mathscr{P}_2(\mathbb{S})$.
\end{lemma}

\begin{proof}
We separate the proof by the following two steps.\\
\textbf{Step 1:} In this step, we shall prove that the sequence $\{\mathscr{S}^{N}\}_{N\in\mathbb{N}}$ is  tight in $\mathscr{P}(\mathbb{S})$ w.r.t.~the topology of weak convergence, whose proof is inspired by Sznitman \cite{S1}. For any $\varepsilon>0$, in view of Proposition \ref{pro2}, we could choose a set $\mathscr{U}_\varepsilon$ that is a compact subset in $\mathbb{S}$ such that
$$\mathbb{P}(X^{N,1}\notin \mathscr{U}_\varepsilon)\leq\varepsilon^2.$$
Then by the symmetry of the law of $(X^{N,1},\ldots,X^{N,N})$, we deduce that
$$\mathbb{E}\Big[\mathscr{S}^{N}(\mathscr{U}_\varepsilon^c)\Big]=\frac{1}{N}\sum_{j=1}^N\mathbb{P}(X^{N,j}\notin \mathscr{U}_\varepsilon)\leq \varepsilon^2.$$
Let $\Phi_\varepsilon=\Big\{\Gamma\in\mathscr{P}(\mathbb{S}):\Gamma(\mathscr{U}_{\varepsilon2^{-m}})\geq 1-\varepsilon2^{-m},m\geq 1\Big\}$. Then it is easy to see that $\Phi_\varepsilon$ is compact in $\mathscr{P}(\mathbb{S})$ and
\begin{eqnarray*}
\mathbb{P}(\mathscr{S}^{N}\notin\Phi_\varepsilon)\leq \!\!\!\!\!\!\!\!&& \sum_{m=1}^\infty\mathbb{P}(\mathscr{S}^{N}(\mathscr{U}_{\varepsilon2^{-m}}^c)>\varepsilon2^{-m})
\nonumber \\
\leq\!\!\!\!\!\!\!\!&&\sum_{m=1}^\infty(\varepsilon2^{-m})^{-1}\mathbb{E}\Big[\mathscr{S}^{N}(\mathscr{U}_{\varepsilon2^{-m}}^c)\Big]
\nonumber \\
\leq\!\!\!\!\!\!\!\!&&\sum_{m=1}^\infty\varepsilon2^{-m}=\varepsilon.
\end{eqnarray*}
Thus $\{\mathscr{S}^{N}\}_{N\in\mathbb{N}}$ is  tight in $\mathscr{P}(\mathbb{S})$.\\
\textbf{Step 2:} Now we show that the sequence $\{\mathscr{S}^{N}\}_{N\in\mathbb{N}}$ is  tight in $\mathscr{P}_2(\mathbb{S})$ w.r.t.~the $L^2$-Wasserstein distance.
Since $\{\mathscr{S}^{N}\}_{N\in\mathbb{N}}$ is  tight in $\mathscr{P}(\mathbb{S})$ from \textbf{Step 1} and $\{\mathscr{S}^{N}\}_{N\in\mathbb{N}}\subset \mathscr{P}_2(\mathbb{S})$ $\mathbb{P}$-a.s., there exists a set $\mathscr{N}_\varepsilon\subset\mathscr{P}_2(\mathbb{S})$ that is relatively compact in $\mathscr{P}(\mathbb{S})$ such that we have
\begin{equation}\label{es1118}
\mathbb{P}(\mathscr{S}^{N}\notin \mathscr{N}_\varepsilon)\leq \varepsilon.
\end{equation}
For any $1<q\leq\frac{p}{2}$, let
$$a_m=m^{\frac{1}{2q-2}}2^{\frac{m}{2q-2}},~~~~~b_m=\frac{\varepsilon m}{\sup_{N\in\mathbb{N}}\mathbb{E}\Big[\sup_{t\in[0,T]}\|X^{N,1}_t\|_{L^2}^{2q}+\Big(\int_0^T\|X^{N,1}_t\|_{1}^{2}dt\Big)^{q}\Big]\vee 1},$$
and set
$$\Psi_\varepsilon=\bigcap_{m\in\mathbb{N}}\Big\{\nu\in\mathscr{P}_2(\mathbb{S}):\int d_{T}(u,0)^2\mathbf{1}_{\{d_{T}(u,0)\geq a_m\}}\nu(du)<\frac{1}{b_m}\Big\}.$$
Then we deduce from (\ref{es117}), (\ref{es1170})  that
\begin{eqnarray}\label{es1119}
\mathbb{P}(\mathscr{S}^{N}\notin\Psi_\varepsilon)\leq \!\!\!\!\!\!\!\!&& \sum_{m=1}^\infty\mathbb{P}\Big(\frac{1}{N}\sum_{i=1}^Nd_{T}(X^{N,i},0)^2\mathbf{1}_{\{d_{T}(X^{N,i},0)\geq a_m\}}\geq\frac{1}{b_m}\Big)
\nonumber \\
\leq\!\!\!\!\!\!\!\!&&\sum_{m=1}^\infty \frac{b_m}{N} \sum_{i=1}^N \mathbb{E}\Big[d_{T}(X^{N,i},0)^2\mathbf{1}_{\{d_{T}(X^{N,i},0)\geq a_m\}}\Big]
\nonumber \\
\leq\!\!\!\!\!\!\!\!&&\sum_{m=1}^\infty\frac{b_m}{a_m^{2q-2}}\mathbb{E}\Big[\sup_{t\in[0,T]}\|X^{N,1}_t\|_{L^2}^{2q}+\Big(\int_0^T\|X^{N,1}_t\|_{1}^{2}dt\Big)^{q}\Big]
\nonumber \\
\leq\!\!\!\!\!\!\!\!&&\varepsilon.
\end{eqnarray}
Thus combining (\ref{es1118}) and (\ref{es1119}), it is easy to see
\begin{equation*}
\mathbb{P}(\mathscr{S}^{N}\notin \mathscr{N}_\varepsilon\cap\Psi_\varepsilon)\leq 2\varepsilon.
\end{equation*}
According to Lemma \ref{lem00}, the set $\mathscr{N}_\varepsilon\cap\Psi_\varepsilon$ is relatively compact in $\mathscr{P}_2(\mathbb{S})$, therefore we conclude that the sequence $\{\mathscr{S}^{N}\}_{N\in\mathbb{N}}$ is  tight in $\mathscr{P}_2(\mathbb{S})$. The proof is completed.
\end{proof}

\subsection{Proof of Theorem \ref{th4}}\label{sec3.3}
In this subsection, we will give the detailed proof of Theorem \ref{th4}. By Lemma \ref{pro13} and the Prokhorov theorem,  we know that $\mathscr{S}^{N}$ (as random variable in $\mathscr{P}_2(\mathbb{S})$) converges weakly to $\mathscr{S}$ (here choose a subsequence if necessary), namely,
$$\text{the law of}~\mathscr{S}^{N}~\text{converges weakly to the law of}~\mathscr{S}~\text{in}~\mathscr{P}(\mathscr{P}_2(\mathbb{S})).$$
Moreover, by (\ref{es117}),  (\ref{es1170}) and  Fatou's lemma we can get
\begin{eqnarray*}
\mathbb{E}\Big[\int d_T(x,0)^{p}\mathscr{S}(dx)\Big]\leq \!\!\!\!\!\!\!\!&& \mathbb{E}\Big[\liminf_{R\to\infty}\int(d_T(x,0)^{p}\wedge R)\mathscr{S}(dx)\Big]
\nonumber \\
\leq \!\!\!\!\!\!\!\!&& \liminf_{R\to\infty}\liminf_{N\to\infty}\mathbb{E}\Big[\int(d_T(x,0)^{p}\wedge R)\mathscr{S}^{N}(dx)\Big]
\nonumber \\
\leq\!\!\!\!\!\!\!\!&& \sup_{N\in\mathbb{N}}\mathbb{E}\Big[\sup_{t\in[0,T]}\|X^{N,1}_t\|_{L^2}^{p}+\Big(\int_0^T\|X^{N,1}_t\|_{1}^{2}dt\Big)^{\frac{p}{2}}\Big]
<\infty.
\end{eqnarray*}
Therefore we have
\begin{equation}\label{es1132}
\mathscr{S}\in \mathscr{P}_p(\mathbb{S})\subset \mathscr{P}_2(\mathbb{S})~~\mathbb{P}\text{-a.s.}.
\end{equation}

For any $(x,\nu,l)\in \mathbb{S}\times\mathscr{P}_2(\mathbb{S})\times \mathscr{V}$, let
\begin{eqnarray}\label{eq11m}
\!\!\!\!\!\!\!\!&&\mathscr{M}_l(t,x,\nu)
\nonumber \\
:=\!\!\!\!\!\!\!\!&&\langle x_t,l\rangle-\langle x_0,l\rangle-\int_0^t\int\langle \tilde{\mathbf{F}}(s,x_s,x'_s),l\rangle \nu(dx')ds
\nonumber \\
:=\!\!\!\!\!\!\!\!&&\langle x_t,l\rangle-\langle x_0,l\rangle-\int_0^t\int\langle Ax_s-B(x_s)+\tilde{K}(s,x_s,x'_s),l\rangle \nu(dx')ds,~t\in[0,T].
\end{eqnarray}
Note that  by Lemma \ref{Property B2}  and $\mathbf{A1}$-$\mathbf{A2}$, for any $l\in\mathscr{V}$
\begin{eqnarray}\label{es446}
\!\!\!\!\!\!\!\!&&\langle \tilde{\mathbf{F}}(t,u_1,v_1)-\tilde{\mathbf{F}}(t,u_2,v_2),l\rangle
\nonumber \\
\leq\!\!\!\!\!\!\!\!&& C\|u_1-u_2\|_{L^2}\Big[1+\|u_1\|_{L^2}+\|u_2\|_{L^2}
+\|v_1\|_{L^2}+\|v_2\|_{L^2}\Big]\|l\|_{4}
\nonumber \\
\!\!\!\!\!\!\!\!&&
+C\|v_1-v_2\|_{L^2}\Big[1+\|v_1\|_{L^2}+\|v_2\|_{L^2}\Big]\|l\|_{L^2}.
\end{eqnarray}

We now denote
\begin{eqnarray*}
G(t,x,x'):=G^{(1)}(t,x)+G^{(2)}(t,x,x'),
\end{eqnarray*}
where
$$G^{(1)}(t,x):=\langle x_t,l\rangle-\langle x_0,l\rangle,$$
$$G^{(2)}(t,x,x'):=-\int_0^t\langle \tilde{\mathbf{F}}(s,x_s,x'_s),l\rangle ds,$$
and for fixed $m\in\mathbb{N}$, let $g_1,\ldots,g_m$ be a real valued sequence in $C_b(\mathbb{H})$, $\nu\in\mathscr{P}_2(\mathbb{S})$, $0\leq s<t\leq T$, $0\leq s_1<\cdots<s_m\leq s$,
$$\Pi(\nu):=\int\int\Big(G(t,x,x')-G(s,x,x')\Big)g_1(x_{s_1})\cdots g_m(x_{s_m})\nu(dx')\nu(dx).$$
For any $R>0$, we define
\begin{eqnarray}\label{es1120}
G_R(t,x,x'):=G^{(1)}_R(t,x)+G^{(2)}_R(t,x,x'),
\end{eqnarray}
where
$$G^{(1)}_R(t,x):=\langle x_t,l\rangle\cdot\chi_R(\langle x_t,l\rangle)-\langle x_0,l\rangle\cdot\chi_R(\langle x_0,l\rangle),$$
$$G^{(2)}_R(t,x,x'):=-\int_0^t\langle \tilde{\mathbf{F}}(s,x_s,x'_s),l\rangle\cdot\chi_R(\langle \tilde{\mathbf{F}}(s,x_s,x'_s),l\rangle)ds,$$
where $\chi_R\in C^{\infty}_c(\mathbb{R})$ is a cut-off function with
$$\chi_R(r)=\begin{cases} 1,~~~~|r|\leq R&\quad\\
0,~~~~|r|>2R,&\quad\end{cases}$$
and denote
$$\Pi_R(\nu):=\int\int\Big(G_R(t,x,x')-G_R(s,x,x')\Big)g_1(x_{s_1})\cdots g_m(x_{s_m})\nu(dx')\nu(dx).$$
Note that for any $t\in[0,T]$, the function
$$\mathbb{S}\times \mathbb{S}\ni(x,x')\mapsto G_R(t,x,x')~\text{is bounded continuous}.$$
In fact, the bounedness of $G_R$ is obvious, if $(x^n,{x'}^n)\to(x,x')$ in $\mathbb{S}\times \mathbb{S}$, it is straightforward that for any $t\in[0,T]$,
\begin{equation*}
\lim_{n\to\infty}G^{(1)}_R(t,x^{n})=G^{(1)}_R(t,x),
\end{equation*}
by (\ref{es446}) and the dominated convergence theorem,  we have
\begin{equation*}
\lim_{n\to\infty}G^{(2)}_R(t,x^{n},{x'}^n)=G^{(2)}_R(t,x,x').
\end{equation*}
Then we can conclude
\begin{equation*}
\lim_{n\to\infty}G_R(t,x^{n},{x'}^n)=G_R(t,x,x').
\end{equation*}
Furthermore, since $g_1,\ldots,g_m$ are bounded and continuous, it follows that
$$\nu\mapsto\Pi_R(\nu)~~\text{is continuous and bounded}. $$

Since $\mathbb{P}$-a.s. $\mathscr{S},\mathscr{S}^N\in \mathscr{P}_2(\mathbb{S})$, we have $\mathbb{P}$-a.s.
\begin{eqnarray}
\!\!\!\!\!\!\!\!&&\Pi_R(\mathscr{S}^N)=\frac{1}{N^2}\sum_{i,j=1}^{N}\Big(G_R(t,X^{N,i},X^{N,j})-G_R(s,X^{N,i},X^{N,j})\Big)g_1(X^{N,i}_{s_1})\cdots g_m(X^{N,i}_{s_m}),
\nonumber \\
\!\!\!\!\!\!\!\!&&\Pi(\mathscr{S}^N)=\frac{1}{N^2}\sum_{i,j=1}^{N}\Big(G(t,X^{N,i},X^{N,j})-G(s,X^{N,i},X^{N,j})\Big)g_1(X^{N,i}_{s_1})\cdots g_m(X^{N,i}_{s_m}),\label{es1129}
 \\
\!\!\!\!\!\!\!\!&&\Pi_R(\mathscr{S})=\int\int\Big(G_R(t,x,x')-G_R(s,x,x')\Big)g_1(x_{s_1})\cdots g_m(x_{s_m})\mathscr{S}(dx')\mathscr{S}(dx),
\nonumber \\
\!\!\!\!\!\!\!\!&&\Pi(\mathscr{S})=\int\int\Big(G(t,x,x')-G(s,x,x')\Big)g_1(x_{s_1})\cdots g_m(x_{s_m})\mathscr{S}(dx')\mathscr{S}(dx).\label{es1130}
\end{eqnarray}

We intend to prove that $\mathbb{P}$-a.s.~$\Pi(\mathscr{S})=0$ by the following two lemmas.
%Based on this, we are in the position to show that $\mathscr{S}(\omega)$ is a martingale solution to McKean-Vlasov SPDE (\ref{eqNSE111}).

\begin{lemma}\label{lem5}
Under the assumptions in Theorem \ref{th4}, we have
$$\mathbb{E}|\Pi(\mathscr{S}^N)|^2\to0~~\text{as}~N\to\infty.$$
\end{lemma}

\begin{proof}
Note that
\begin{eqnarray}\label{es77}
\Pi(\mathscr{S}^N)=\frac{1}{N}\sum_{i=1}^{N}\big(\mathcal{M}_t^i-\mathcal{M}_s^i\big)g_1(X^{N,i}_{s_1})\cdots g_m(X^{N,i}_{s_m}),
\end{eqnarray}
where
\begin{eqnarray*}
\mathcal{M}_t^i:=\!\!\!\!\!\!\!\!&&\langle X^{N,i}_t,l\rangle-\langle \xi^i,l\rangle-\int_0^t\langle AX^{N,i}_s-B(X^{N,i}_s),l\rangle ds
\nonumber \\
\!\!\!\!\!\!\!\!&&~~-\frac{1}{N}\sum_{j=1}^{N}\int_0^t\langle\tilde{K}(s,X^{N,i}_s,X^{N,j}_s),l\rangle ds
\nonumber \\
=\!\!\!\!\!\!\!\!&&\frac{1}{N}\sum_{j=1}^{N}\langle\int_0^t\tilde{\sigma}(s,X^{N,i}_s,X^{N,j}_s) dW^i_s,l\rangle,
\end{eqnarray*}
which turns out to be  a square integrable martingale. Hence we deduce that
\begin{eqnarray}\label{es75}
\mathbb{E}\Big[\mathcal{M}_t^i\mathcal{M}_s^k|\mathscr{F}_s^N\Big]=\mathbb{E}\Big[\mathcal{M}_s^i\mathcal{M}_s^k|\mathscr{F}_s^N\Big],
\end{eqnarray}
where $\mathscr{F}_s^N:=\sigma\Big\{X_r^{N,i}:r\leq s,1\leq i\leq N\Big\}$.

Since for $i\neq k$, $W^i$ and $W^k$ are independent, it leads to
$$\langle \mathcal{M}^i,\mathcal{M}^k\rangle_t=0,$$
which implies that $\mathcal{M}^i\mathcal{M}^k$ is a martingale for $i\neq k$, i.e.,
\begin{eqnarray}\label{es76}
\mathbb{E}\Big[\mathcal{M}_t^i\mathcal{M}_t^k|\mathscr{F}_s^N\Big]=\mathbb{E}\Big[\mathcal{M}_s^i\mathcal{M}_s^k|\mathscr{F}_s^N\Big].
\end{eqnarray}
Combining (\ref{es75}) and (\ref{es76}), for $i\neq k$, we have
$$\mathbb{E}\Big[\big(\mathcal{M}_t^i-\mathcal{M}_s^i\big)\big(\mathcal{M}_t^k-\mathcal{M}_s^k\big)\Big]=0.$$
Recall (\ref{es77}), noting that $g_1,\ldots,g_m$ are bounded, it follows that
\begin{eqnarray}\label{es74}
\mathbb{E}|\Pi(\mathscr{S}^N)|^2\leq\!\!\!\!\!\!\!\!&& \frac{1}{N^2}\sum_{i,k=1}^{N}\mathbb{E}\Big[\big(\mathcal{M}_t^i-\mathcal{M}_s^i\big)\big(\mathcal{M}_t^k-\mathcal{M}_s^k\big)\Big].
\nonumber \\
=\!\!\!\!\!\!\!\!&&\frac{1}{N^2}\sum_{i=1}^{N}\mathbb{E}|\mathcal{M}_t^i-\mathcal{M}_s^i|^2
\nonumber \\
\leq\!\!\!\!\!\!\!\!&&\frac{C_T}{N^2}\sum_{i=1}^{N}\Big(1+\mathbb{E}\big[\sup_{t\in[0,T]}\|X^{N,1}_t\|_{L^2}^2\big]\Big)
\nonumber \\
\leq\!\!\!\!\!\!\!\!&&\frac{C_T}{N}\to 0 ~\text{as}~N\to\infty,
\end{eqnarray}
which implies the assertion.
\end{proof}

\begin{lemma}\label{lem4}
Under the assumptions in Theorem \ref{th4}, we have
$$\mathbb{E}|\Pi(\mathscr{S}^N)|\to\mathbb{E}|\Pi(\mathscr{S})|~~\text{as}~N\to\infty.$$
\end{lemma}
\begin{proof}
Note that
\begin{eqnarray}\label{es1124}
\mathbb{E}|\Pi(\mathscr{S}^N)|-\mathbb{E}|\Pi(\mathscr{S})|=\!\!\!\!\!\!\!\!&&\mathbb{E}\Big[|\Pi(\mathscr{S}^N)|-|\Pi_R(\mathscr{S}^N)|\Big]+\mathbb{E}\Big[|\Pi_R(\mathscr{S}^N)|-|\Pi_R(\mathscr{S})|\Big]
\nonumber \\
\!\!\!\!\!\!\!\!&&+\mathbb{E}\Big[|\Pi_R(\mathscr{S})|-|\Pi(\mathscr{S})|\Big].
\end{eqnarray}
Since $\Pi_R$ is continuous and bounded, by the dominated convergence theorem
$$\lim_{N\to\infty}\mathbb{E}|\Pi_R(\mathscr{S}^N)-\Pi_R(\mathscr{S})|=0,$$
which implies
\begin{equation}\label{es69}
\lim_{N\to\infty}\mathbb{E}|\Pi_R(\mathscr{S}^N)|=\mathbb{E}|\Pi_R(\mathscr{S})|.
\end{equation}
As for the first term on right hand side of (\ref{es1124}), we have
\begin{eqnarray}\label{es1125}
\!\!\!\!\!\!\!\!&&\mathbb{E}|\Pi_R(\mathscr{S}^{N})-\Pi(\mathscr{S}^N)|
\nonumber \\
\leq\!\!\!\!\!\!\!\!&&C\sup_{t\in[0,T]}\Big\{\frac{1}{N}\sum_{i=1}^N\mathbb{E}\big|G_R^{(1)}(t,X^{N,i})-G^{(1)}(t,X^{N,i})\big|\Big\}
\nonumber \\
\!\!\!\!\!\!\!\!&&+C\sup_{t\in[0,T]}\Big\{\frac{1}{N^2}\sum_{i,j=1}^N\mathbb{E}\big|G_R^{(2)}(t,X^{N,i},X^{N,j})-G^{(2)}(t,X^{N,i},X^{N,j})\big|\Big\}.
\end{eqnarray}
Due to $\mathbf{A1}$ and Lemma \ref{lem11},
\begin{eqnarray*}
\!\!\!\!\!\!\!\!&&\frac{1}{N^2}\sum_{i,j=1}^N\mathbb{E}\big|G_R^{(2)}(t,X^{N,i},X^{N,j})-G^{(2)}(t,X^{N,i},X^{N,j})\big|
\nonumber \\
\leq\!\!\!\!\!\!\!\!&&\frac{1}{N^2}\sum_{i,j=1}^N\mathbb{E}\Big[\int_0^T|\langle \tilde{\mathbf{F}}(s,X^{N,i}_s,X^{N,j}_s),l\rangle|\cdot \mathbf{1}_{\big\{|\langle \tilde{\mathbf{F}}(s,X^{N,i}_s,X^{N,j}_s),l\rangle|\geq R\big\}}ds\Big]
\nonumber \\
\leq\!\!\!\!\!\!\!\!&&\|l\|_{1}\frac{1}{N^2}\sum_{i,j=1}^N\Big[\Big(\int_0^T\mathbb{E}\|\tilde{\mathbf{F}}(s,X^{N,i}_s,X^{N,j}_s)\|_{-1}^{2}ds\Big)^{\frac{1}{2}}
\nonumber \\
\!\!\!\!\!\!\!\!&&
~~\cdot \Big(\int_0^T\mathbb{P}\big(|\langle \tilde{\mathbf{F}}(s,X^{N,i}_s,X^{N,j}_s),l\rangle|\geq R\big)ds\Big)^{\frac{1}{2}}\Big]
\nonumber \\
\leq\!\!\!\!\!\!\!\!&&\|l\|_{1}\frac{1}{N^2}\sum_{i,j=1}^N\Big(\mathbb{E}\int_0^T\|AX^{N,i}_s-B(X^{N,i}_s)+\tilde{K}(s,X^{N,i}_s,X^{N,j}_s)\|_{-1}^{2}ds\Big)\Big/R
\nonumber \\
\leq\!\!\!\!\!\!\!\!&&C\|l\|_{1}\Big(\mathbb{E}\int_0^T\big(1+\frac{1}{N}\sum_{i=1}^N\|X^{N,i}_s\|_{L^2}^{2}\|X^{N,i}_s\|_{1}^{2}
\nonumber \\
\!\!\!\!\!\!\!\!&&
~+\frac{1}{N}\sum_{i=1}^N\|X^{N,i}_s\|_{1}^{2}+\frac{1}{N}\sum_{j=1}^N\|X^{N,j}_s\|_{L^2}^{2}\big)ds\Big)\Big/R
\nonumber \\
\leq\!\!\!\!\!\!\!\!&&C_{T}\|l\|_{1}\Big/R.
\end{eqnarray*}
Therefore it follows that
\begin{equation}\label{es1127}
\lim_{R\to\infty}\sup_{N\in\mathbb{N}}\sup_{t\in[0,T]}\Big\{\frac{1}{N^2}\sum_{i,j=1}^N\mathbb{E}\big|G_R^{(2)}(t,X^{N,i},X^{N,j})-G^{(2)}(t,X^{N,i},X^{N,j})\big|\Big\}=0.
\end{equation}
Analogously, in view of Lemma \ref{lem11}, it gives that
\begin{equation}\label{es1128}
\lim_{R\to\infty}\sup_{N\in\mathbb{N}}\sup_{t\in[0,T]}\Big\{\frac{1}{N}\sum_{i=1}^N\mathbb{E}\big|G_R^{(1)}(t,X^{N,i})-G^{(1)}(t,X^{N,i})\big|\Big\}=0.
\end{equation}
Using (\ref{es1125})-(\ref{es1128}), we conclude
\begin{equation}\label{es71}
\lim_{R\to\infty}\sup_{N\in\mathbb{N}}\mathbb{E}|\Pi_R(\mathscr{S}^{N})-\Pi(\mathscr{S}^N)|=0.
\end{equation}
Following similar calculations shows that
\begin{equation}\label{es72}
\lim_{R\to\infty}\mathbb{E}|\Pi_R(\mathscr{S})-\Pi(\mathscr{S})|=0.
\end{equation}
Consequently, combining (\ref{es69}), (\ref{es71}) and (\ref{es72}), we complete the proof.
\end{proof}

\vspace{1mm}
We now return to the proof of Theorem \ref{th4}.  According to Lemmas \ref{lem5} and \ref{lem4}, it is easy to show that $\mathbb{P}$-a.s.~$\Pi(\mathscr{S})=0$.
%Let $\mathbb{C}\subset C_b({\mathbb{V}}^*)$ be a countable set. Let $U$ be the set of all $x\in \Omega$ such that $\mathscr{S}_{[0,M]}(x)\in \mathscr{P}_q(\mathbb{S}_M)$ for each $M\in\mathbb{N}$, $\mathscr{S}_0(x)=\mu_0$ and $\Pi(\mathscr{S}(x))=0$ for each $s,t\in[0,T]\cap\mathbb{Q}$, $s<t$, $m\in\mathbb{N}$, $t_1,\ldots,t_m\in[0,s]\cap\mathbb{Q}$ and $g_1,\ldots,g_m\in\mathbb{C}$.
Since $C_b(\mathbb{H})$ is separable, there exists a countable and dense subsect $\mathscr{D}_0\subset C_b(\mathbb{H})$. Let
$$\tilde{\Omega}:=\Big\{\tilde{x}\in\Omega:\mathscr{S}(\tilde{x})\in\mathscr{P}_2(\mathbb{S})~\text{and}~ \Pi(\mathscr{S})=0~\text{for}~g_1,\ldots,g_m\in\mathscr{D}_0\Big\}.$$
Then $\mathbb{P}(\tilde{\Omega})=1$, and for all $\tilde{x}\in \tilde{\Omega}$,
\begin{eqnarray}\label{mar11}
\mathscr{M}_l(t,x,\mathscr{S}(\tilde{x}))=\!\!\!\!\!\!\!\!&&\langle x_{t},l\rangle-\langle x_{0},l\rangle-\int_0^t\langle Ax_{s}-B(x_{s}),l\rangle ds
\nonumber \\
\!\!\!\!\!\!\!\!&&-\int_0^t\langle \int\tilde{K}(s,x_{s},x'_s)\mathscr{S}(\tilde{x})(dx'),l\rangle ds,
%\nonumber \\
%=\!\!\!\!\!\!\!\!&&\langle x_{t},l\rangle-\langle x_{0},l\rangle-\int_0^t\langle Ax_{s}-B(x_{s}),l\rangle ds
%\nonumber \\
%\!\!\!\!\!\!\!\!&&-\int_0^t\langle \int_{\mathbb{H}}\tilde{K}(s,x_{s},y)\mathscr{S}_s(\tilde{x})(dy),l\rangle ds,
\end{eqnarray}
is a $\mathscr{S}(\tilde{x})$-martingale. By (\ref{P2.2}), (\ref{linear1}) and (\ref{es1132}), once we can prove that the quadratic variation process of (\ref{mar11}) is
\begin{equation}\label{es62}
\int_0^{t}\Big\|\Big(\int\tilde{\sigma}(s,x_s,x'_s)\mathscr{S}(\tilde{x})(dx')\Big)^*l\Big\|_U^2ds,
\end{equation}
then $\mathscr{S}(\tilde{x})$ is a solution of martingale problem (\ref{eqNSE111}) in the sense of Definition \ref{de3}.

We turn to prove (\ref{es62}). In fact, it is sufficient to prove that
\begin{eqnarray}\label{es1131}
\!\!\!\!\!\!\!\!&&\mathbb{E}\Big|\int\int\Big(G^2(t,x,x')-G^2(s,x,x')
\nonumber \\
\!\!\!\!\!\!\!\!&&-\int_s^t\Big\|\Big(\int\tilde{\sigma}(r,x_r,x'_r)\mathscr{S}(dx')\Big)^*l\Big\|_U^2dr\Big)g_1(x_{s_1})\cdots g_m(x_{s_m})\mathscr{S}(dx')\mathscr{S}(dx)\Big|=0.~~~~
\end{eqnarray}
By (\ref{es116}), H\"{o}lder's inequality and BDG's inequality, we have for some $p'>1$,
\begin{eqnarray*}
\!\!\!\!\!\!\!\!&&\mathbb{E}\Big|\int\int G^2(t,x,x')\mathscr{S}^N(dx')\mathscr{S}^N(dx)\Big|^{p'}
\nonumber \\
=\!\!\!\!\!\!\!\!&&\mathbb{E}\Big|\frac{1}{N^2}\sum_{i,j=1}^{N}G^2(t,X^{N,i},X^{N,j})\Big|^{p'}
\nonumber \\
\leq\!\!\!\!\!\!\!\!&&\frac{C}{N^{2}}\sum_{i,j=1}^{N}\mathbb{E}\Big|G(t,X^{N,i},X^{N,j})\Big|^{2p'}
\nonumber \\
\leq\!\!\!\!\!\!\!\!&&\frac{C_T}{N^2}\sum_{i,j=1}^{N}\mathbb{E}\int_0^T\|\tilde{\sigma}(s,X^{N,i}_s,X^{N,j}_s)\|_{L_2(U;{\mathbb{H}})}^{2p'}ds
\nonumber \\
\leq\!\!\!\!\!\!\!\!&&C_{p',T}.
\end{eqnarray*}
Note that $p'>1$, by a similar argument as the proof of Lemma \ref{lem4}, it is easy to see
\begin{eqnarray*}
\!\!\!\!\!\!\!\!&&\lim_{N\to\infty}\mathbb{E}\Big|\int\int G^2(t,x,x')\mathscr{S}^N(dx')\mathscr{S}^N(dx)
\nonumber \\
\!\!\!\!\!\!\!\!&&~~~~~~~~~~-\int\int G^2(t,x,x')\mathscr{S}(dx')\mathscr{S}(dx)\Big|=0
\end{eqnarray*}
and
\begin{eqnarray*}
\!\!\!\!\!\!\!\!&&\lim_{N\to\infty}\mathbb{E}\Big|\int\int_0^t\Big\|\Big(\int\tilde{\sigma}(s,x_s,x'_s)\mathscr{S}^N(dx')\Big)^*l\Big\|_{U}^2ds\mathscr{S}^N(dx)
\nonumber \\
\!\!\!\!\!\!\!\!&&~~~~~~~~~~-\int\int_0^t\Big\|\Big(\int\tilde{\sigma}(s,x_s,x'_s)\mathscr{S}(dx')\Big)^*l\Big\|_{U}^2ds\mathscr{S}(dx)\Big|=0,
\end{eqnarray*}
which implies that (\ref{es1131}) holds.

Since $\mathscr{S}^N_0$ converges to $\mu_0$  in probability, we deduce that
\begin{equation}\label{es63}
\mathscr{S}_0=\mu_0~~\mathbb{P}\text{-a.s.}.
\end{equation}
 Therefore, for any subsequence (still denoted by $(N)_{N\in\mathbb{N}}$) there exists a subsequence $(N_k)_{k\in\mathbb{N}}$  such that $\mathscr{S}^{N_k}$ weakly converges to a martingale solution $\mathscr{S}$ to (\ref{eqNSE111})  in $\mathscr{P}_2(\mathbb{S})$, as $k\to\infty$.

Furthermore, if the assumption $\mathbf{H4}$ holds for $K$ and $\sigma$, by Lemma \ref{uniqueness} we know that the solutions of the martingale problem is unique.  Then $\mathscr{S}^N$  converges weakly to the unique martingale solution $\Gamma$ in $\mathscr{P}_2(\mathbb{S})$ as $N\to\infty$. Note that all subsequential limits of $\{\mathbb{P}\circ(\mathscr{S}^N)^{-1}\}$ are identified as $\delta_{\Gamma}$, it implies that
$$\mathbb{W}_{2,T,{\mathbb{H}}}(\mathscr{S}^N,\Gamma)\to 0~~\text{in probability}.$$
To verify (\ref{es111}), it is sufficient to show that the family $\{\mathbb{W}_{2,T,{\mathbb{H}}}(\mathscr{S}^N,\Gamma)^2\}_{N\in\mathbb{N}}$ is uniformly integrable. By (\ref{es1132}) we have
$$\Gamma\in \mathscr{P}_p(\mathbb{S}).$$
This combining with (\ref{es116}) implies the uniform integrability of family $\{\mathbb{W}_{2,T,{\mathbb{H}}}(\mathscr{S}^N,\Gamma)^2\}_{N\in\mathbb{N}}$.
Hence we finish the proof of Theorem \ref{th4}. \hspace{\fill}$\Box$

\section{Appendix}\label{appendix}
\setcounter{equation}{0}
 \setcounter{definition}{0}
%\subsection{Some lemmas}
The classical Skorokhod theorem can only be applied in metric space. In this work, we use the following Jakubowski's version of the Skorokhod theorem in the form presented by Brze\'{z}niak and Ondrej\'{a}t \cite{BO}.
\begin{lemma}\label{sko1}$($Skorokhod Theorem$)$
Let $\mathscr{Y}$ be a topological space such that there exists a sequence of continuous functions $f_m:\mathscr{Y}\to \mathbb{R}$ that separates points of  $\mathscr{Y}$. Let us denote by $\mathscr{S}$ the $\sigma$-algebra generated by the maps $f_m$. Then

(i) every compact subset of $\mathscr{Y}$ is metrizable;

(ii) if $(\mu_m)$ is tight sequence of probability measures on $(\mathscr{Y},\mathscr{S})$, then there exists a subsequence  denoted also by $(m)$, a probability space $(\Omega,\mathscr{F},\mathbb{P})$ with $\mathscr{Y}$-valued Borel measurable variables $\xi_m$, $\xi$ such that $\mu_m$ is the law of $\xi_m$ and $\xi_m$ converges to $\xi$ almost surely on $\Omega$. Moreover, the law of $\xi$ is a Random measure.
\end{lemma}

In order to apply the above type of Skorokhod theorem, we recall the definitions of the countably generated Borel space and the standard
Borel space (cf. \cite[Chapter V, Definition 2.1 and 2.2]{P1}).
\begin{definition}
(Countably generated Borel space) A
Borel space $(X, \mathscr{B})$ is said to be countably generated if there exists a denumerable class
$\mathcal{D} \subset \mathscr{B}$ such that $\mathcal{D}$ generates $\mathscr{B}$.
\end{definition}

\begin{definition}\label{de5}
(Standard Borel space) A countably
generated Borel space $(X, \mathscr{B})$ is called standard if there exists a Polish
space $Y$ such that the $\sigma$-algebras $\mathscr{B}$ and $\mathscr{Y}$ are $\sigma$-isomorphic.
\end{definition}

\vspace{2mm}
We introduce the following lemma (cf.~\cite{CD1}), which is crucial in the proof of the tightness of  sequence  $\{\mathscr{S}^{N}\}_{N\in\mathbb{N}}$ in $\mathscr{P}_2(\mathbb{S})$.
\begin{lemma}\label{lem00}
If $(E,d)$ is a complete separable metric space, for any $ p>1$, any subset $\mathscr{K}\subset\mathscr{P}_p(E)$, relatively compact for
the topology of weak convergence of probability measures, any $x_0\in E$, and any
sequences $(a_m)_{m\geq 1}$ and $(b_m)_{m\geq 1}$ of positive real numbers tending to $\infty$ with $m$, then the
set
$$\mathscr{K}\cap\Big\{\mu\in\mathscr{P}_p(E):\forall m\geq 1,\int_{\{d(x_0,x)\geq a_m\}}d(x_0,x)^pd\mu(x)<\frac{1}{b_m}\Big\}$$
is relatively compact for the Wasserstein distance $\mathbb{W}_p$.
\end{lemma}

%\subsection{Proof of Lemma \ref{lem6}}\label{sec6.2}
Before presenting the proof of Lemma \ref{lem6}, we give the following  compactness result, which is a minor modification of \cite[Lemma C.2]{GRZ}.
\begin{lemma}\label{lem55}
Let $\mathbb{X}\subset\mathbb{H}$ compactly. Let $\mathscr{N}_1\in\mathfrak{M}^q$ for some $q\geq 1$, and $\mathscr{Z}$ a subset of $C([0,T];\mathbb{X}^*)$. If $ \mathscr{Z}$ is equi-continuous in $C([0,T];\mathbb{X}^*)$ and
\begin{equation}\label{es29}
\sup_{x\in \mathscr{Z}}\sup_{t\in[0,T]}\|x_t\|_{\mathbb{H}}+\sup_{x\in \mathscr{Z}}\Big(\int_0^T\mathscr{N}_1(x_t)dt+\int_0^T\|x_t\|_{\mathbb{V}}^qdt\Big)<\infty.
\end{equation}
Then  $\mathscr{Z}\subset \mathbb{K}_2=C([0,T];\mathbb{X}^*)\cap L^{q}([0,T];\mathbb{Y})\cap L^{q}_w([0,T];{\mathbb{V}})$ and is relatively compact in $\mathbb{K}_2$.

\end{lemma}

\begin{proof}
Recall \cite[Lemma C.2]{GRZ} which has proved that  $\mathscr{Z}$ is relatively compact in $\mathbb{K}_1=C([0,T];\mathbb{X}^*)\cap L^{q}([0,T];\mathbb{Y})$. Thus it suffices to show that $\mathscr{Z}$ is relatively compact in $L^{q}_w([0,T];{\mathbb{V}})$. In fact, by (\ref{es29}) we have
$$\sup_{x\in \mathscr{Z}}\int_0^T\|x_t\|_{\mathbb{V}}^qdt <\infty,$$
the assertion follows directly from the Banach-Alaoglu Theorem.
\end{proof}

\vspace{2mm}
Now we are in the position to finish the proof of Lemma \ref{lem6}.

\noindent\textbf{Proof of Lemma \ref{lem6}}~
For $\gamma\in(0,1)$ and $R>0$, set
$$K_{\gamma,R}:=\Big\{x\in C([0,T];\mathbb{X}^*): \sup_{t\in[0,T]}\|x_t\|_{\mathbb{H}}+\sup_{s\neq t\in[0,T]}\frac{\|x_t-x_s\|_{\mathbb{X}^*}}{|t-s|^{\gamma}}+\int_0^T\mathscr{N}_1(x_t)dt+\int_0^T\|x_t\|_{\mathbb{V}}^qdt\leq R\Big\}.$$
Then $K_{\gamma,R}$ is compact in $\mathbb{K}_2$ by Lemma \ref{lem55}. According to Lemmas \ref{lem13} and \ref{lem2}, (\ref{es00}) and Chebyschev's inequality, we have for some $\gamma\in(0,1)$ and any $R>0$,
$$\mathbb{P}^n(X^n_{\cdot}\in K_{\gamma,R})\leq \frac{C_T}{R},$$
which implies that Lemma \ref{lem6} holds. \hspace{\fill}$\Box$

\vspace{5mm}

\noindent\textbf{\large Acknowledgements} The authors would like to thank the referees for their very constructive suggestions and thank Professor Michael  R\"{o}ckner for many helpful discussions and
insightful suggestions, in particular for the use of Jakubowski's version of the Skorokhod theorem.

\end{document}